\documentclass[a4paper,11pt]{article}
\usepackage{amsmath,amsthm,amssymb,enumitem,xcolor}

\usepackage[hang]{footmisc}
\usepackage[nosort,nocompress,noadjust]{cite}

\usepackage[bookmarks=false,hyperfootnotes=false,colorlinks,
    linkcolor={red!60!black},
    citecolor={blue!50!black},
    urlcolor={blue!80!black}]{hyperref}

\renewcommand{\eqref}[1]{\hyperref[#1]{(\ref{#1})}}

\newlist{enumlist}{enumerate}{2}
\setlist[enumlist,1]{labelindent=0cm,label=\arabic*.,ref=\arabic*,labelwidth=2.5ex,labelsep=0.5ex,leftmargin=3ex,align=left,topsep=0.5ex,itemsep=1ex,parsep=1ex}
\setlist[enumlist,2]{labelindent=0cm,label=\theenumlisti.\arabic*.,ref=\arabic*,labelwidth=5ex,labelsep=0.5ex,leftmargin=5.5ex,align=left,topsep=0.5ex,itemsep=1ex,parsep=1ex}

\newlist{itemlist}{itemize}{1}
\setlist[itemlist]{labelindent=0cm,label=$\bullet$,labelwidth=2.5ex,labelsep=0.5ex,leftmargin=3ex,align=left,topsep=0.5ex,itemsep=1ex,parsep=1ex}

\numberwithin{equation}{section}

{\theoremstyle{definition}\newtheorem{definition}{Definition}[section]
\newtheorem*{definition*}{Definition}

\newtheorem{remark}[definition]{Remark}
\newtheorem{example}[definition]{Example}
\newtheorem*{example*}{Example}
\newtheorem*{examples*}{Examples}}

\newtheorem{proposition}[definition]{Proposition}
\newtheorem{lemma}[definition]{Lemma}
\newtheorem{theorem}[definition]{Theorem}
\newtheorem{corollary}[definition]{Corollary}

\newtheorem{letterthm}{Theorem}

{\theoremstyle{definition}}

\renewcommand{\Re}{\operatorname{Re}}

\newcommand{\C}{\mathbb{C}}

\newcommand{\eps}{\varepsilon}
\newcommand{\al}{\alpha}
\newcommand{\be}{\beta}
\newcommand{\ot}{\otimes}
\newcommand{\recht}{\rightarrow}
\newcommand{\Z}{\mathbb{Z}}
\newcommand{\vphi}{\varphi}

\newcommand{\id}{\mathord{\text{\rm id}}}
\newcommand{\om}{\omega}
\newcommand{\N}{\mathbb{N}}
\newcommand{\ovt}{\mathbin{\overline{\otimes}}}

\newcommand{\Om}{\Omega}
\newcommand{\cD}{\mathcal{D}}
\newcommand{\si}{\sigma}
\newcommand{\R}{\mathbb{R}}

\newcommand{\cH}{\mathcal{H}}
\newcommand{\cZ}{\mathcal{Z}}

\newcommand{\cG}{\mathcal{G}}
\newcommand{\cK}{\mathcal{K}}

\newcommand{\cF}{\mathcal{F}}
\newcommand{\T}{\mathbb{T}}
\newcommand{\actson}{\curvearrowright}

\newcommand{\cW}{\mathcal{W}}
\newcommand{\cU}{\mathcal{U}}

\newcommand{\cR}{\mathcal{R}}

\newcommand{\cV}{\mathcal{V}}

\newcommand{\cS}{\mathcal{S}}

\newcommand{\cL}{\mathcal{L}}

\newcommand{\per}{\operatorname{per}}
\newcommand{\mutil}{\widetilde{\mu}}
\newcommand{\Xtil}{\widetilde{X}}

\newcommand{\symdiff}{\vartriangle}
\newcommand{\cUtil}{\widetilde{\cU}}
\newcommand{\omtil}{\widetilde{\omega}}
\newcommand{\gammatil}{\widetilde{\gamma}}
\newcommand{\xtil}{\widetilde{x}}
\newcommand{\ztil}{\widetilde{z}}
\newcommand{\nutil}{\widetilde{\nu}}
\newcommand{\Var}{\operatorname{Var}}
\newcommand{\ctil}{\widetilde{c}}
\newcommand{\rhotil}{\widetilde{\rho}}

\begin{document}

\begin{center}
{\boldmath\LARGE\bf Nonsingular Bernoulli actions \vspace{0.5ex}\\ of arbitrary Krieger type}

\bigskip

{\sc by Tey Berendschot\footnote{\noindent KU~Leuven, Department of Mathematics, Leuven (Belgium).\\ E-mails: tey.berendschot@kuleuven.be and stefaan.vaes@kuleuven.be.}\textsuperscript{,}\footnote{\noindent T.B.\ and S.V.\ are supported by long term structural funding~-- Methusalem grant of the Flemish Government.} and Stefaan Vaes\textsuperscript{1,2,}\footnote{S.V.\ is supported by FWO research project G090420N of the Research Foundation Flanders.}}
\end{center}

\begin{abstract}\noindent
\noindent We prove that every infinite amenable group admits Bernoulli actions of any possible Krieger type, including type II$_\infty$ and type III$_0$. We obtain this result as a consequence of general results on the ergodicity and Krieger type of nonsingular Bernoulli actions $G \curvearrowright \prod_{g \in G} (X_0,\mu_g)$ with arbitrary base space $X_0$, both for amenable and for nonamenable groups. Earlier work focused on two point base spaces $X_0 = \{0,1\}$, where type II$_\infty$ was proven not to occur.
\end{abstract}

\section{Introduction}

Both in operator algebras and in ergodic theory, a very prominent class of group actions is given by the Bernoulli construction. Given a countable group $G$ and a standard probability space $(X_0,\mu_0)$, one considers the translation action of $G$ on the infinite product $(X,\mu) = (X_0,\mu_0)^G$. Some of the most striking superrigidity theorems for von Neumann algebras and orbit equivalence relations in the past decades were proven for such probability measure preserving (pmp) Bernoulli actions, see e.g.\ \cite{Pop03,Pop05}.

In recent years, motivated by the success of Popa's deformation/rigidity theory for II$_1$ factors, there has been an increasing interest in group actions that are not measure preserving and their associated von Neumann algebras of type III, see e.g.\ \cite{AIM19,HI15,HMV16}. In this context, nonsingular Bernoulli actions are a very natural family of group actions to consider. Given a standard Borel space $X_0$, one still considers the translation action of $G$ on $X = X_0^G$, but $X$ is now equipped with the product of a family of potentially distinct probability measures $(\mu_g)_{g \in G}$. Even seemingly basic questions as when such an action is ergodic and what its Krieger type may be, turn out to be very subtle and intimately related to properties of the acting group $G$.

Early work on these questions focussed on the construction of concrete examples, mainly for the group $\Z$. In \cite{Ham81}, it was shown that there exist ergodic Bernoulli actions of $\Z$ of type III (i.e.\ without an equivalent finite or $\si$-finite invariant measure), while it was only proven much later in \cite{Kos09} that there exist Bernoulli actions of $\Z$ of type III$_1$. Still in the spirit of constructing interesting large families of examples, it was proven in great generality in \cite{Kos12,DL16} when a Bernoulli action of $\Z$ with all $\mu_n$ being identical for negative $n < 0$, is ergodic and of type III$_1$, while \cite{VW17} provided the first systematic approach to Bernoulli actions of nonamenable groups $G$. In particular, it was proven in \cite{VW17}, that if a nonamenable group $G$ admits a Bernoulli action of type III, then $G$ must have a positive first $L^2$-Betti number. It was conjectured in \cite{VW17} that the converse holds. Recently, in \cite{BKV19}, this conjecture was proven as part of a series of results providing the first entirely general criteria for ergodicity and Krieger type of nonsingular Bernoulli actions $G \actson \{0,1\}^G$ with a \emph{two point base space} $\{0,1\}$.

One of the main results of \cite{BKV19}, solving a question going back to Krengel, is that a nonsingular Bernoulli action of $\Z$ with a two point base space $X_0 = \{0,1\}$ is never of type II$_\infty$. In \cite{VW17,BKV19}, it was shown that all non virtually cyclic groups with more than one end (like the free groups, or free product groups) admit Bernoulli actions of type III$_\lambda$ with $\lambda$ strictly positive, while infinite locally finite groups were proven to admit Bernoulli actions of all possible types. Very recently in \cite{KS20}, it was proven that all infinite amenable groups admit Bernoulli actions of type III$_\lambda$ with $\lambda$ strictly positive.

All this left open the question whether the group of integers, or any other non locally finite group, admits Bernoulli actions of type II$_\infty$ or type III$_0$. Our first main result says that the answer is yes, if we allow for an infinite base space $X_0$. The original question of Krengel thus gets the following surprising answer: a Bernoulli action of $\Z$ can be of type II$_\infty$ if and only if we allow the base space $X_0$ to be infinite.

\begin{letterthm}\label{thmA.amenable-all-types}
Every infinite amenable group admits nonsingular Bernoulli actions of all possible types: II$_1$, II$_\infty$ and III$_\lambda$ for all $\lambda \in [0,1]$.
\end{letterthm}

We prove Theorem \ref{thmA.amenable-all-types} as a consequence of several general results on the ergodicity and Krieger type of Bernoulli actions with an arbitrary base space $X_0$. To present these results, we first introduce a few notations.

Let $X_0$ be a standard Borel space and $G$ a countably infinite group. Assume that $(\mu_g)_{g \in G}$ is a family of equivalent probability measures on $X_0$. Consider the Bernoulli action
\begin{equation}\label{eq.Bernoulli-action}
G \actson (X,\mu) = \prod_{g \in G} (X_0,\mu_g) : (g \cdot x)_h = x_{g^{-1} h} \quad\text{for all $x \in X$ and $g,h \in G$.}
\end{equation}
By Kakutani's criterion for the equivalence of product measures, the Bernoulli action $G \actson (X,\mu)$ is nonsingular if and only if for every $g \in G$,
\begin{equation}\label{eq.kakutani}
\sum_{h \in G} H^2(\mu_{gh},\mu_h) < +\infty \; ,
\end{equation}
where $H$ denotes the Hellinger distance between two probability measures. Recall here that for probability measures $\mu$ and $\nu$ on $X_0$,
$$H^2(\mu,\nu) = \frac{1}{2} \int_{X_0} \Bigl| \sqrt{d\mu/d\zeta} - \sqrt{d\nu/d\zeta}\Bigr|^2 \, d\zeta = 1 - \int_{X_0} \sqrt{d\mu/d\zeta} \, \sqrt{d\nu/d\zeta} \, d \zeta \; ,$$
where $\zeta$ is any probability measure such that $\mu$ and $\nu$ are absolutely continuous w.r.t.\ $\zeta$.

In \cite{BKV19}, general ergodicity and Krieger type results were obtained for $X_0 = \{0,1\}$ and assuming that $\mu_g(0) \in [\delta,1-\delta]$ for all $g \in G$ and some $\delta > 0$. When $X_0$ is allowed to be infinite, new phenomena arise as can already be seen from our Theorem \ref{thmA.amenable-all-types}, and new methods are needed. First, for infinite base spaces $X_0$, there are two types of boundedness assumptions replacing $\mu_g(0) \in [\delta,1-\delta]$; the strong condition
\begin{equation}\label{eq.boundedness-strong}
\exists C > 0 \;\;\text{such that}\;\; C^{-1} \leq \frac{d\mu_g}{d\mu_e}(x) \leq C \quad\text{for all $g \in G$ and a.e.\ $x \in X_0$~;}
\end{equation}
and the weaker condition
\begin{equation}\label{eq.boundedness-weak}
\sup_{g \in G} \Bigl|\log \frac{d\mu_g}{d\mu_e}(x) \Bigr| < +\infty \quad\text{for a.e.\ $x \in X_0$.}
\end{equation}
When $X_0$ is a finite set, both conditions \eqref{eq.boundedness-strong} and \eqref{eq.boundedness-weak} are obviously equivalent. It turns out that we can determine the Krieger type of a nonsingular Bernoulli action under the weaker assumption \eqref{eq.boundedness-weak}, while it is the stronger assumption \eqref{eq.boundedness-strong} that rules out types II$_\infty$ and III$_0$. Also the recent paper \cite{KS20} relies on the strong boundedness assumption \eqref{eq.boundedness-strong} and thus only provides examples of Bernoulli shifts of type III$_\lambda$ with $\lambda \in (0,1]$.

The precise description of the Krieger type of a nonsingular Bernoulli action under the weaker assumption \eqref{eq.boundedness-weak} is a bit technical and postponed to Theorems \ref{thm.Krieger-type-Bernoulli} and \ref{thm.rule-out-III0}. Under the stronger assumption \eqref{eq.boundedness-strong}, the formulation goes as follows. Recall that an essentially free group action is called dissipative if it admits a fundamental domain.

\begin{letterthm}\label{thmB.type-amenable}
Let $G$ be an infinite amenable group and $G \actson (X,\mu) = \prod_{g \in G} (X_0,\mu_g)$ a nonsingular Bernoulli action. Assume that the stronger boundedness property \eqref{eq.boundedness-strong} holds. Also assume that $G \actson (X,\mu)$ is not dissipative.

Then, $G \actson (X,\mu)$ is weakly mixing. It is of type II$_1$ if and only if $\mu \sim \nu^G$ for some probability measure $\nu \sim \mu_e$ on $X_0$. Otherwise, it is of type III$_\lambda$ with $\lambda \in (0,1]$. The precise value of $\lambda$ is determined by the $T$-invariant, which is described in Theorem \ref{thm.Krieger-type-Bernoulli}.
\end{letterthm}

For nonamenable groups $G$, we prove a result that is similar to Theorem \ref{thmB.type-amenable}. The assumption that $G \actson (X,\mu)$ is not dissipative has to be replaced by a stronger conservativeness (i.e.\ recurrence) assumption that we discuss below. On the other hand, for nonamenable groups $G$, type II$_\infty$ and type III$_0$ are already ruled out under the weaker boundedness assumption \eqref{eq.boundedness-weak}.

\begin{letterthm}\label{thmC.type-nonamenable}
Let $G$ be a nonamenable group and $G \actson (X,\mu) = \prod_{g \in G} (X_0,\mu_g)$ a nonsingular Bernoulli action. Assume that the weaker boundedness property \eqref{eq.boundedness-weak} holds. Also assume that $G \actson (X,\mu)$ satisfies the conservativeness property \eqref{eq.strong-conservative-assumption} below.

Then, $G \actson (X,\mu)$ is weakly mixing. It is of type II$_1$ if and only if $\mu \sim \nu^G$ for some probability measure $\nu \sim \mu_e$ on $X_0$. Otherwise, it is of type III$_\lambda$ with $\lambda \in (0,1]$. The precise value of $\lambda$ is determined by the $T$-invariant, which is described in Theorem \ref{thm.Krieger-type-Bernoulli}.
\end{letterthm}

We now discuss when nonsingular Bernoulli actions are conservative or dissipative, and we refer to Section \ref{sec.conservative} for details. Recall from \cite[Theorem 3.1]{VW17} that a nonsingular Bernoulli action gives rise to a $1$-cocycle of the group $G$ with values in $\ell^2(G) \ot L^2(X_0,\mu_e)$ given by
\begin{equation}\label{eq.1-cocycle-bernoulli}
\sqrt{2} \, c_g(h,x) = \Bigl(\frac{d\mu_{g^{-1}h}}{d\mu_e}(x)\Bigr)^{1/2} - \Bigl(\frac{d\mu_{h}}{d\mu_e}(x)\Bigr)^{1/2} \; .
\end{equation}
The Kakutani criterion \eqref{eq.kakutani} precisely says that
$$\|c_g\|^2 = \sum_{h \in G} H^2(\mu_{gh},\mu_h) < +\infty \; .$$

By \cite[Proposition 4.1]{VW17}, if $\sum_{g \in G} \exp(- \|c_g\|^2) <+\infty$, the action $G \actson (X,\mu)$ is dissipative. One cannot hope to formulate a sufficient condition for conservativeness purely in terms of the growth of $g \mapsto \|c_g\|^2$, because basically any growth type can be realized by taking $X_0 = \{0,1\}$ and $\sum_{g \in G} \mu_g(0) < +\infty$, in which case $\mu$ is even an atomic measure. However, defining $C(g) \in [0,+\infty]$ by
\begin{equation}\label{eq.notation-C}
C(g) = \sum_{h \in G} \frac{1}{2} \log \int_{X_0} \frac{d\mu_h}{d\mu_{gh}} \, d\mu_h \; ,
\end{equation}
we will see in \eqref{eq.estimate-C} that
$$\|c(g)\|^2 \leq C(g)$$
and it turns out that a sufficiently slow growth of $C(g)$ implies conservativeness of the Bernoulli action $G \actson (X,\mu)$. More precisely, we prove in Lemma \ref{lem.strongly-conservative} that $G \actson (X,\mu)$ is strongly conservative (in the sense of \cite{BKV19}, see Definition \ref{def.strongly-conservative} below) if
\begin{equation}\label{eq.strong-conservative-assumption}
\limsup_{s \recht +\infty} \frac{\log |\{g \in G \mid C(g^{\pm 1}) \leq s\}|}{s} > 6 \; .
\end{equation}

Apart from proving Theorems \ref{thmA.amenable-all-types}, \ref{thmB.type-amenable} and \ref{thmC.type-nonamenable}, we also apply our results to give the following broad family of type III$_1$ Bernoulli actions with diffuse base space $X_0 = \R$. We also show that when all $\mu_g$ are Gaussian measures with the same variance, then the associated Bernoulli action is conjugate to a nonsingular Gaussian action, as introduced in \cite{AIM19}.

Assume that $\vphi : \R \recht (0,+\infty)$ is a continuous function with $\int_\R \vphi(t) \, dt = 1$. For every $s \in \R$, consider the translated probability measure $\nu_s$ on $\R$ given by $d\nu_s(t) = \vphi(t+s) \, dt$. Let now $G$ be any countable group and $F : G \recht \R$ any bounded function implementing an $\ell^2$-cocycle $c_g(h) = F(g^{-1}h) - F(h)$ with $c_g \in \ell^2(G)$. We then consider the equivalent probability measures $\mu_g$ on $\R$ given by $\mu_g = \nu_{F(g)}$ and consider the Bernoulli action
\begin{equation}\label{eq.example-diffuse}
G \actson (X,\mu) = \prod_{g \in G} (\R,\mu_g) \; .
\end{equation}

Assume that $c_g$ is not a coboundary, i.e.\ that $F$ is not equal to a constant plus a square summable function. Under very general conditions, the Bernoulli action $G \actson (X,\mu) = \prod_{g \in G}(\R,\mu_g)$ is nonsingular, weakly mixing and of stable type III$_1$. To formulate the result, define the Poincar\'{e} exponent of a $1$-cocycle $c_g$ as in \cite[Definition 3.1]{AIM19} by
\begin{equation}\label{eq.poincare}
\delta(c) = \limsup_{s \recht +\infty} \frac{\log |\{g \in G \mid \|c_g\|^2 \leq s\}|}{s} = \inf \Bigl\{ \kappa > 0 \Bigm| \sum_{g \in G} \exp(- \kappa \|c_g\|^2) < +\infty \Bigr\} \; .
\end{equation}
Although the $1$-cocycle $c : G \recht \ell^2(G)$ should not be confused with the $1$-cocycle of the nonsingular Bernoulli action considered in \eqref{eq.1-cocycle-bernoulli}, there is a strong relation between both and the Bernoulli action in \eqref{eq.example-diffuse} tends to be dissipative if $\delta(c)$ is too small. The precise relation is however subtle; see Remark \ref{rem.about-growth}.

\begin{letterthm}\label{thmD.examples-diffuse}
Let $\vphi : \R \recht (0,+\infty)$ be a continuous function with integral $1$. Let $G$ be a countable group and $F : G \recht \R$ a bounded function such that $c_g : h \mapsto F(g^{-1}h) - F(h)$ belongs to $\ell^2(G)$ for all $g \in G$. Define the measures $\mu_g$ on $\R$ by $d\mu_g(t) = \vphi(t+F(g))\, dt$. Let $G \actson (X,\mu)$ be the Bernoulli action as in \eqref{eq.example-diffuse}. Consider the following two conditions.
\begin{enumlist}
\item The function $\log \vphi$ is Lipschitz with constant $M \geq 0$.
\item The function $\log \vphi$ is differentiable and $(\log \vphi)'$ is Lipschitz with constant $M \geq 0$.
\end{enumlist}
In both cases, $G \actson (X,\mu)$ is nonsingular. In case~1, if $\delta(c) > 9 M^2 / 2$, the action $G \actson (X,\mu)$ is weakly mixing and it is of stable type III$_1$, unless $c_g$ is a coboundary, in which case it is of type II$_1$.

In case~2, if $\delta(c) > 6 M$ and if $\int_\R |t|^\al \, \vphi(t) \, dt <+\infty$ for some $\al > 2$, the same conclusions hold.
\end{letterthm}

Taking the standard Gaussian density $\vphi(t) = (2\pi)^{-1/2} \exp(- t^2/2)$, the nonsingular Bernoulli actions in Theorem \ref{thmD.examples-diffuse} are as follows canonically isomorphic with the nonsingular Gaussian action associated in \cite{AIM19} to the isometric action $G \actson^\pi \ell^2_\R(G)$ given by $(\pi(g)\xi)(h) = \xi(g^{-1} h) - c_g(h)$. Denoting $d\nu(t) = \vphi(t) \, dt$, the map $\Theta : (X,\mu) \recht (\R,\nu)^G : \Theta(x)_g = x_g + F(g)$ is measure preserving. Identifying $(\R,\nu)^G$ with the Gaussian probability space of $\ell^2_\R(G)$ by taking the canonical orthonormal basis of $\ell^2_\R(G)$, the map $\Theta$ is $G$-equivariant.

Then, \cite[Theorems C and 5.2]{AIM19}
say that $G \actson (X,\mu)$ is weakly mixing if $\delta(c) > 1/2$, and of type III$_0$ or III$_1$ if $\delta(c) > 1$ and $c$ is not a coboundary. The second case of Theorem \ref{thmD.examples-diffuse} says that $G \actson (X,\mu)$ is of type III$_1$ if $\delta(c) > 6$. Of course, in the Gaussian setting, there is much more symmetry and it should be no surprise that our general Theorem \ref{thmD.examples-diffuse} does not provide the optimal bound on $\delta(c)$ for this very special case.

In the final section of this paper, we consider Bernoulli actions with a two point base space $\{0,1\}$ and solve the following problem that was left open in \cite{BKV19}. If $\Z \actson (X,\mu) = \prod_{n \in \Z} (\{0,1\},\mu_n)$ is any such Bernoulli action for the group of integers and if $\mu_n(0)$ does not converge to $0$ or $1$ when $|n| \recht +\infty$, it was proven in \cite{BKV19} that one of the following holds: the action is dissipative; we have $\mu \sim \nu^\Z$ and the action is of type II$_1$; or the action is of type III$_1$. When $\mu_n(0)$ converges to $0$ or $1$, it was proven in \cite{BKV19} that the action is either dissipative or of type III, but it remained open whether the second alternative can actually occur. We prove that it indeed does.

\begin{letterthm}\label{thmE.Z-marginals-to-zero}
There exist $\mu_n(0) \in (0,1)$ such that $\mu_n(0) \recht 0$ as $|n| \recht +\infty$ and such that the Bernoulli action $\Z \actson (X,\mu) = \prod_{n \in \Z} (\{0,1\},\mu_n)$ is weakly mixing and of type III$_1$.
\end{letterthm}

We finally discuss the main methods used in this paper. To prove ergodicity and determine the Krieger type of a nonsingular Bernoulli action \eqref{eq.Bernoulli-action}, we proceed in several steps. Under the appropriate assumptions, we prove that $G$-invariant functions $F \in L^\infty(X)^G$ are automatically invariant under the action $\cS_G \actson (X,\mu)$ of the group $\cS_G$ of finite permutations of the countable set $G$, which acts by permuting finitely many coordinates. We also relate Krieger's associated flow of $G \actson (X,\mu)$ to the associated flow of $\cS_G \actson (X,\mu)$. This is mainly done in Proposition \ref{prop.ergodic-general} and Lemma \ref{lem.link-to-permutation-action}, whose proofs are following quite closely the methods of \cite{BKV19}, with the idea of using the permutation action and the tail equivalence relation going back to \cite{Kos18,Dan18}.

In \cite{AP77}, very general ergodicity results for the permutation action $\cS_G \actson (X,\mu)$ were obtained. However, its Krieger type has so far only been considered for finite base spaces, and mainly for $X_0 = \{0,1\}$, see \cite{SV77}. In Theorems \ref{thm.T-invariant-permutation-action} and \ref{thm.permutation-condition-periodic}, we determine the type of the permutation action in great generality. This is in turn based on a reduction of the associated flow of $\cS_G \actson (X,\mu)$ to the flow on the tail boundary of a time dependent random walk on $\R$, to which we can apply the results of \cite{Ore66}. As we explain in Section \ref{sec.tail-boundary}, it is quite remarkable how \cite{Ore66} has been overlooked by the ergodic theory and operator algebra community.

\section{Preliminaries}

\subsection{Unconditional convergence a.e.\ of infinite sums and products}

Let $(X,\mu)$ be a standard probability space and $F_i : X \recht \R$, $i \in I$, a countable family of measurable maps. We say that $\sum_{i\in I} F_i(x)$ converges \emph{unconditionally a.e.}\ if there exists a measurable map $F : X \recht \R$ such that for any increasing sequence of finite subsets $I_n \subset I$ with $\bigcup_{n=1}^\infty I_n = I$, we have
$$\lim_{n \recht +\infty} \sum_{i \in I_n} F_i(x) = F(x) \quad\text{for a.e.\ $x \in X$.}$$
Similarly, when $F_i : X \recht \C \setminus \{0\}$, we say that $\prod_{i \in I} F_i(x)$ converges \emph{unconditionally a.e.}\ if there exists a measurable map $F : X \recht \C \setminus \{0\}$ such that for any increasing sequence of finite subsets $I_n \subset I$ with $\bigcup_{n=1}^\infty I_n = I$, we have
$$\lim_{n \recht +\infty} \prod_{i \in I_n} F_i(x) = F(x) \quad\text{for a.e.\ $x \in X$.}$$

Assume now that $(X,\mu) = \prod_{i \in I} (X_i,\mu_i)$. For all $\kappa > 0$, consider the cutoff function $T_\kappa : \R \recht \R$ given by
\begin{equation}\label{eq.cut-off}
T_\kappa : \R \recht \R : T_\kappa(t) = \begin{cases} -\kappa &\;\;\text{if $t \leq -\kappa$},\\ t &\;\;\text{if $-\kappa \leq t \leq \kappa$,} \\ \kappa &\;\;\text{if $t \geq \kappa$.}\end{cases}
\end{equation}
If $F_i : X_i \recht \R$ are measurable functions and if there exist $s_i \in \R$ such that
\begin{equation}\label{eq.cond-a}
\sum_{i \in I} \int_{X_i} T_\kappa(F_i(x)-s_i)^2 \, d\mu_i(x) < +\infty \; ,
\end{equation}
then van Kampen's version of Kolmogorov's three series theorem says that there exist $t_i \in \R$ such that $\sum_{i \in I} (F_i(x_i) - t_i)$ converges unconditionally a.e.\ on the product space $(X,\mu)$. It is however important to note that one cannot necessarily take $t_i = s_i$, because condition \eqref{eq.cond-a} only determines $s_i$ up to a square summable family, while the unconditional convergence a.e.\ determines $t_i$ up to an absolutely summable family. When \eqref{eq.cond-a} holds, one may take
$$t_i = s_i + \int_{X_i} T_\kappa(F_i(x) - s_i) \, d\mu_i(x) \; .$$

Similarly, when $\vphi_i : X_i \recht \C \setminus \{0\}$ are such that
$$
\int_{X_i} |\vphi_i(x)|^2 \, d\mu_i(x) = 1 \quad\text{and}\quad \int_{X_i} \vphi_i(x) \, d\mu_i(x) \in [0,+\infty) \; ,
$$
where the latter can always be realized by multiplying $\vphi_i$ with a constant in $\T$, and if
\begin{equation}\label{eq.cond-b}
\sum_{i \in I} \Bigl(1 - \int_{X_i} \vphi_i(x) \, d\mu_i(x)\Bigr) < +\infty \; ,
\end{equation}
then $\prod_{i \in I} \vphi_i(x_i)$ converges unconditionally a.e.\ on the product space $(X,\mu)$.

\subsection{Nonsingular group actions, and Bernoulli actions}

An action of a countable group $G$ by measurable transformations of a standard probability space $(X,\mu)$ is called \emph{nonsingular} if $\mu(g \cdot \cU)=0$ if and only if $\mu(\cU) = 0$ for all $g \in G$ and $\cU \subset X$ measurable. Recall that the \emph{Maharam extension} of such a nonsingular action is given by
$$G \actson X \times \R : g \cdot (x,t) = (g \cdot x, t + \log d(g^{-1}\mu)/d\mu(x)) \; .$$
This action commutes with the translation action of $\R$ in the second variable. Therefore, if $G \actson (X,\mu)$ is ergodic, the action of $\R$ on the von Neumann algebra of $G$-invariant functions $L^\infty(X \times \R)^G$ is an ergodic action of $\R$. This is Krieger's \emph{associated flow} of $G \actson (X,\mu)$.

The ergodic nonsingular action $G \actson (X,\mu)$ admits a (finite or $\sigma$-finite) invariant measure that is equivalent to $\mu$ if and only if the associated flow is conjugate to the translation action $\R \actson \R$. Otherwise, $G \actson (X,\mu)$ is of type~III. In that case, if the associated flow is periodic, i.e.\ conjugate to the translation action $\R \actson \R/p\Z$ for $p \neq 0$, the action is said to by of type III$_\lambda$ with $\lambda = \exp(-|p|)$. If the associated flow is trivial, i.e.\ the Maharam extension remains ergodic, then the action is said to be of type III$_1$. If the associated flow is properly ergodic, $G \actson (X,\mu)$ is said to be of type III$_0$.

By the fundamental work of Connes, Feldman, Krieger, Ornstein and Weiss, for ergodic nonsingular actions of \emph{amenable} groups, the type and associated flow are a complete invariant for the \emph{orbit equivalence relation} of $G \actson (X,\mu)$, and by seminal work of Connes and Takesaki, this is even a complete invariant for the associated von Neumann algebras.

Also recall Connes' \emph{$T$-invariant}, which for an ergodic nonsingular action $G \actson (X,\mu)$ can be defined as the subgroup of $\R$ consisting of the \emph{eigenvalues} $s \in \R$ of the associated flow $\R \actson (Z,\eta)$. An element $s \in \R$ is called an eigenvalue, if there exists a measurable $F : Z \recht \T$ such that $F(t \cdot z) = \exp(its) \, F(z)$ for all $t \in \R$ and a.e.\ $z \in Z$.

A nonsingular action $G \actson (X,\mu)$ is called \emph{weakly mixing} if it is ergodic and if the diagonal action $G \actson X \times Y$ remains ergodic for any ergodic \emph{probability measure preserving (pmp)} action $G \actson (Y,\eta)$. A weakly mixing $G \actson (X,\mu)$ is said to be of \emph{stable type III$_\lambda$} if each of these product actions is of the same type III$_\lambda$.

Whenever $G$ is a countable group and $(\mu_g)_{g \in G}$ is a family of equivalent probability measures on a standard Borel space $X_0$, we consider the Bernoulli action $G \actson (X,\mu)$ \eqref{eq.Bernoulli-action}. This action is nonsingular if and only if the Kakutani criterion \eqref{eq.kakutani} holds. In that case, the Radon-Nikodym derivative is given by
$$\frac{d(g^{-1}\mu)}{d\mu}(x) = \prod_{h \in G} \frac{d\mu_{gh}}{d\mu_h}(x_h) \; ,$$
where, for every $g \in G$, the product converges unconditionally a.e. In that case, we also have that
$$1 - H^2(g^{-1}\mu,\mu) = \prod_{h \in G} (1-H^2(\mu_{gh},\mu_h)) \; .$$
The product space $(X,\mu)$ is nonatomic, unless there exist atoms $a_g \in X_0$ such that $\sum_{g \in G} (1-\mu_g(a_g)) < +\infty$. We always tacitly rule out this trivial atomic situation. Then any nonsingular Bernoulli action $G \actson (X,\mu)$ is \emph{essentially free} (see e.g.\ \cite[Lemma 2.2]{BKV19}), meaning that for every $g \in G \setminus \{e\}$, the set $\{x \in X \mid g \cdot x = x\}$ has measure zero.

\subsection{\boldmath The tail boundary of random walks on $\R$}\label{sec.tail-boundary}

Let $(\mu_n)_{n \geq 1}$ be a sequence of probability measures on $\R$ describing the transition probabilities of a random walk on $\R$. So, fixing an initial probability measure $\mu_0$ on $\R$ that is equivalent to the Lebesgue measure, we consider the Markov chain $X_n$, where $X_0$ has distribution $\mu_0$, where the increments $X_n-X_{n-1}$ have distribution $\mu_n$ for all $n \geq 1$ and the random variables $X_0,X_1-X_0,X_2-X_1,\ldots$ are independent. The \emph{tail boundary} of this random walk is defined as the intersection of the $\si$-algebras $\bigcap_{n=1}^\infty \si(X_m \mid m \geq n)$. It comes equipped with a natural action of $\R$ given by translation. More concretely, define
\begin{equation}\label{eq.not-1}
(\Om,\eta) = \prod_{n=0}^\infty (\R,\mu_n) \quad\text{with}\quad X_n(\om) = \sum_{k=0}^n \om_k \; .
\end{equation}
We consider the nonsingular action $\R \actson (\Om,\mu)$ given by translation in the zero'th coordinate $\om_0$. Defining
\begin{equation}\label{eq.not-2}
(\Om_n,\eta_n) = \prod_{m = n}^\infty (\R,\mu_m) \quad\text{and}\quad \pi_n : \Om \recht \R \times \Om_{n+1} : \pi_n(\om) = (X_n(\om) , \om_{n+1},\om_{n+2},\ldots) \; ,
\end{equation}
we consider the von Neumann subalgebras $A_n \subset L^\infty(\Om,\eta)$ given by $A_n = (\pi_n)_*(L^\infty(\R \times \Om_{n+1}))$ and define the tail boundary as $A = \bigcap_{n=1}^\infty A_n$. The action $\R \actson L^\infty(\Om,\eta)$ leaves each $A_n$ globally invariant and thus defines the action $\al : \R \actson A$. Up to conjugacy, this action does not depend on the choice of $\mu_0$ and also does not depend on translation of the measures $\mu_n$. We call $\R \actson A$ the \emph{tail boundary flow} associated with $(\mu_n)_{n \geq 1}$. Note that the tail boundary flow is ergodic.

Equivalently (see e.g.\ \cite[Section 2]{CW88}), elements $F \in A$ are represented as follows by bounded sequences of harmonic functions. Denote by $E_n$ the measure preserving conditional expectation of $L^\infty(\Om,\eta)$ onto the subalgebra of functions that only depend on the variables $\om_0,\ldots,\om_n$. Whenever $F \in A$ and $n \in \N$, we have $F \in A_n$, so that there exists a unique $F_n \in L^\infty(\R)$ with $\|F_n\|_\infty \leq \|F\|_\infty$ and $(E_n(F))(\om) = F_n(X_n(\om))$ for a.e.\ $\om \in \Om$. Since $E_n = E_n \circ E_{n+1}$, it follows that
\begin{equation}\label{eq.harmonic}
F_n = \mu_{n+1}*F_{n+1} \; .
\end{equation}
Conversely, whenever $F_n \in L^\infty(\R)$ is a sequence of functions satisfying \eqref{eq.harmonic} and $\sup_n \|F_n\|_\infty < +\infty$, the sequence $F_n \circ X_n$ is a bounded martingale, converging a.e.\ to $F \in A$.

By \cite[Theorem 3.1]{CW88}, the flow of weights of an infinite tensor product of type~I factors (ITPFI) is described as a tail boundary flow. A similar result holds for other constructions of type III factors and it is thus important to have a criterion when the tail boundary is trivial (the type III$_1$ case) or the tail boundary action is periodic (the type III$_\lambda$ case). ITPFI factors can be of type III$_0$, in which case the tail boundary flow is properly ergodic.

Under the appropriate tightness assumption on the measures $\mu_n$, this type III$_0$ behavior is ruled out. The results of \cite{Ore66} provide very general sufficient conditions for the periodicity of the tail boundary flow. These results seem to have been overlooked by the ergodic theory and operator algebra community, since several special cases have been reproved in the past decades. In particular, the available ergodicity results for cocycles $\cR \recht \R$ on the tail equivalence relation $\cR$ on $\prod_n (X_n,\mu_n)$ are immediate corollaries of \cite{Ore66}.

Before stating the needed results from \cite{Ore66}, we prove a simpler result on the eigenvalue group of the tail boundary flow. Recall that $s \in \R$ is called an eigenvalue of the ergodic action $\R \actson^\al A$ if and only if there exists an $F \in \cU(A)$ such that $\al_t(F) = \exp(i t s) F$ for all $t \in \R$.

The following proposition is an immediate consequence of the convergence criterion for infinite products, and is similar to \cite[Theorem 4.2]{CW88}. For completeness, we include a proof. Recall the notation $T_\kappa$ introduced in \eqref{eq.cut-off}.

\begin{proposition}\label{prop.T-set-tail}
Let $(\mu_n)_{n \geq 1}$ be a sequence of probability measures on $\R$ and consider the associated tail boundary flow.
\begin{enumlist}
\item The tail boundary flow is conjugate to the translation action $\R \actson \R$ if and only if there exist $t_n \in \R$ such that
\begin{equation}\label{eq.tail-semifinite}
\sum_{n =1}^\infty \int_\R T_\kappa(t-t_n)^2 \, d\mu_n(t) < +\infty
\end{equation}
for some $\kappa > 0$ (equivalently, all $\kappa > 0$).

\item For $p \neq 0$, we have that $2\pi / p$ is an eigenvalue of the tail boundary flow if and only if there exist $t_n \in \R$ such that
\begin{equation}\label{eq.T-set-criterion-tail}
\sum_{n=1}^\infty \int_\R d(t-t_n,p \Z)^2 \, d\mu_n(t) < +\infty \; ,
\end{equation}
where $d(t,p\Z) = \min \{|t - p n| \mid n \in \Z\}$ denotes the distance from $t$ to $p\Z$.
\end{enumlist}
\end{proposition}
\begin{proof}
We start by proving the second point. First assume that $2\pi / p$ is an eigenvalue of the tail boundary flow. Take a unitary $F \in \cU(A)$ such that $\al_t(F) = \exp(2\pi i t / p) F$ for all $t \in \R$. The harmonic sequence $F_n \in L^\infty(\R)$ representing $F$ then satisfies the same equivariance property. So we can take $\lambda_n \in \C$ such that $F_n(s) = \lambda_n \, \exp(2\pi i s / p)$ for a.e.\ $s \in \R$. Choose $r_n \in \R$ such that $\lambda_n = |\lambda_n| \, \exp(2\pi i r_n / p)$. We then find $s_k \in \R$ such that
$$
F_n(X_n(\om)) = |\lambda_n| \, \prod_{k=0}^n \exp(2 \pi i (\om_k + s_k) / p) \; .
$$
Since $F_n(X_n(\om)) \recht F(\om)$ for a.e.\ $\om \in \Om$, we get that $|\lambda_n| \recht 1$ and that the infinite product of unitaries $\om_k \mapsto \exp(2 \pi i (\om_k + s_k) / p)$ is convergent. Therefore,
$$\sum_{n=1}^\infty ( 1 - |z_n|) <+\infty \quad\text{where}\quad z_n = \int_\R \exp(2 \pi i t / p) \, d\mu_n(t) \; .$$
Choose $t_n \in \R$ such that $|z_n| = \exp(- 2 \pi i t_n /p) z_n$. Then,
\begin{align*}
|z_n| &= \Re |z_n| = \Re\bigl(\exp(- 2 \pi i t_n /p) z_n\bigr) = \Re\Bigl( \int_\R \exp(2 \pi i (t-t_n) / p) \, d\mu_n(t)\Bigr) \\ &= \int_\R \cos(2 \pi (t-t_n) / p) \, d\mu_n(t) \; .
\end{align*}
We conclude that
$$\sum_{n=1}^\infty \int_\R \sin^2(\pi(t-t_n)/p) \, d\mu_n(t) < +\infty \; .$$
Since for every $x \in \R$, we have $4 d(x,\Z)^2 \leq \sin^2(\pi x) \leq \pi^2 \, d(x,\Z)^2$, it follows that \eqref{eq.T-set-criterion-tail} holds.

Conversely, assume that $p \neq 0$ and that \eqref{eq.T-set-criterion-tail} holds. Write
$$z_n := \int_\R \exp(2 \pi i (t-t_n) / p) \, d\mu_n(t) \; .$$
Then, $|z_n| \leq 1$ for all $n \geq 1$ and
$$\sum_{n=1}^\infty (1-\Re z_n) = 2 \sum_{n=1}^\infty \int_\R \sin^2(\pi (t-t_n) / p) \, d\mu_n(t) \leq \frac{2 \pi^2}{p^2} \sum_{n=1}^\infty \int_\R d(t-t_n, p\Z)^2 \, d\mu_n(t) <+\infty \; .$$
Take $t_n' \in \R$ such that $|z_n| = \exp(- 2 \pi i t_n' / p) \, z_n$. Defining $\theta_n(t) = \exp(2 \pi i (t-t_n-t_n')/p)$, it follows that the product
$$F(\om) = \prod_{n=1}^\infty \theta_n(\om_n)$$
converges unconditionally a.e. By construction, $F$ defines an eigenvector for the tail boundary flow, with eigenvalue $2\pi / p$.

To prove the first point, if the tail boundary flow is conjugate to the translation action $\R \actson \R$, its eigenvalue group is $\R$ and therefore, \eqref{eq.T-set-criterion-tail} holds for all $p \neq 0$. We can now use the following argument of Moore (see \cite[Lemma 3.5]{M65}) to prove that \eqref{eq.tail-semifinite} holds for all $\kappa > 0$. Define the probability measures $\eta_n = \mu_n \times \mu_n$ on $\R^2$. We first prove that
\begin{equation}\label{eq.with-extra-explain}
\sum_{n=1}^\infty \int_{\R^2} d(t-s,p\Z)^2 \, d\eta_n(t,s) < +\infty \quad\text{for all $p \neq 0$.}
\end{equation}
To prove \eqref{eq.with-extra-explain}, fix $p \neq 0$. Take $t_n \in \R$ such that \eqref{eq.T-set-criterion-tail} holds. Define the Hilbert space $V = \bigoplus_{n =1}^\infty L^2(\R^2,\eta_n)$. Since \eqref{eq.T-set-criterion-tail} holds, we can define the vectors $\xi = (\xi_n)$ and $\zeta = (\zeta_n)$ in $V$ by $\xi_n(t,s) = d(t-t_n,p\Z)$ and $\zeta_n(t,s) = d(s-t_n,p\Z)$. Since $d(t-s,p\Z) \leq \xi_n(t,s) + \zeta_n(t,s)$ and $\xi + \zeta \in V$, we conclude that \eqref{eq.with-extra-explain} holds.

Again comparing $d(t,p\Z)^2$ with $\sin^2(\pi/p t)$, it follows from \eqref{eq.with-extra-explain} that
$$M(r) := \sum_{n=1}^\infty \int_{\R^2} \sin^2(r(t-s)) \, d\eta_n(t,s) < +\infty \quad\text{for all $r \in \R$.}$$
We can then pick $M > 0$ and a Borel set $\cU \subset [-M,M]$ of positive Lebesgue measure $\lambda(\cU)$ such that $M(r) \leq M$ for all $r \in \cU$. It follows that
\begin{equation}\label{eq.isfinite}
+\infty > \int_{\cU} M(r) \, dr = \sum_{n=1}^\infty \int_{\R^2} H(t-s) \, d\eta_n(t,s) \quad\text{where}\quad H(t) = \int_{\cU} \sin^2(tr) \, dr \; .
\end{equation}
Fix $\kappa > 0$. Since $H$ is a continuous function with the following properties
$$H(t) > 0 \;\;\text{for all $t \neq 0$,}\quad \lim_{t \recht 0} \frac{H(t)}{t^2} = \int_{\cU} r^2 \, dr > 0 \;\; , \quad \lim_{|t| \recht +\infty} H(t) = \lambda(\cU) / 2 > 0 \; ,$$
we can take $\delta > 0$ such that $\delta \, T_\kappa(t)^2 \leq H(t)$ for all $t \in \R$. It then follows from \eqref{eq.isfinite} that
$$\sum_{n=1}^\infty \int_{\R^2} T_\kappa(t-s)^2 \, d\eta_n(t,s) < +\infty \; .$$
Lemma \ref{lem.technical} below now allows us to choose $t_n \in \R$ such that \eqref{eq.tail-semifinite} holds. Once \eqref{eq.tail-semifinite} holds for a single $\kappa > 0$, it holds for all $\kappa > 0$.

Conversely, if \eqref{eq.tail-semifinite} holds for some $\kappa > 0$, we find $t'_n \in \R$ such that
$$F(\om) = \sum_{n=1}^\infty (\om_n + t'_n)$$
converges unconditionally a.e. By construction, $F$ is an $\R$-equivariant map from the tail boundary flow to $\R$. It follows that the tail boundary flow is conjugate to the translation action $\R \actson \R$ (see e.g.\ Lemma \ref{lem.trivial-flow-result} below for a more general result).
\end{proof}

\begin{lemma}\label{lem.technical}
Let $\mu$ be a probability measure on $\R$ and $\kappa > 0$. There exists $a \in \R$ such that
$$\int_\R T_\kappa(t-a)^2 \, d\mu(t) \leq 8 \, \int_{\R^2} T_\kappa(t-s)^2 \, d\mu(t) \, d\mu(s) \; .$$
\end{lemma}
\begin{proof}
Write $\eps = \int_{\R^2} T_\kappa(t-s)^2 \, d\mu(t) \, d\mu(s)$. If $\eps > \kappa^2/8$, there is nothing to prove. So assume that $\eps \leq \kappa^2/8$. When $|t-s| \geq \kappa/2$, we have $T_\kappa(t-s)^2 \geq \kappa^2/4$ and thus
$$(\mu \times \mu)\bigl(\bigl\{(s,t) \in \R^2 \bigm| |t-s| \geq \kappa/2 \bigr\}\bigr) \leq \frac{4}{\kappa^2} \, \eps \; .$$
This means that
$$\int_\R \mu\bigl(\R \setminus [t-\kappa/2,t+\kappa/2]\bigr) \, d\mu(t) \leq \frac{4}{\kappa^2} \, \eps \; .$$
We can thus choose $b \in \R$ such that
$$\mu\bigl(\R \setminus [b-\kappa/2,b+\kappa/2]\bigr) \leq \frac{4}{\kappa^2} \, \eps \; .$$
In particular, $\mu\bigl([b-\kappa/2,b+\kappa/2]\bigr) \geq 1/2$. Define the probability measure $\mu_0$ on $\R$ by restricting $\mu$ to the interval $I := [b-\kappa/2,b+\kappa/2]$ and normalizing. Define $a \in I$ as the mean of $\mu_0$ given by
$$a = \int_\R t \, d\mu_0(t) \; .$$
We then have
\begin{align*}
\int_{\R} T_\kappa(t-a)^2 \, d\mu(t) & \leq \kappa^2 \, \mu(\R \setminus I) + \int_{\R} (t-a)^2 \, d\mu_0(t) \leq 4 \eps + \frac{1}{2} \int_{\R^2} (t-s)^2 \, d\mu_0(t) \, d\mu_0(s) \\
& \leq 4 \eps + 2 \int_{I \times I} (t-s)^2 \, d\mu(t) \, d\mu(s) = 4 \eps + 2 \int_{I \times I} T_\kappa(t-s)^2 \, d\mu(t) \, d\mu(s) \\
& \leq 6 \eps \leq 8 \eps \; .
\end{align*}
\end{proof}

Once we can rule out that the tail boundary flow is properly ergodic, it is transitive and determined by its eigenvalue group. This eigenvalue group can then be $\R$ (for the flow $\R \actson \R$), or $(2\pi / p) \Z$ (for the flow $\R \actson \R/p\Z$), or trivial (for the trivial flow). The main results of \cite{Ore66} provide sufficient conditions to rule out proper ergodicity of the tail boundary flow. We will make use of the following criteria, essentially contained in \cite{Ore66}.

We say that an action $\R \actson^\al A=L^\infty(Z,\eta)$ of $\R$ is \emph{periodic} if there exists a $p \neq 0$ such that $\al_p(F) = F$ for all $F \in A$. In particular, the trivial action is called periodic.

\begin{theorem}[{\cite{Ore66}}]\label{thm.tail-periodic-criterion}
Let $(\mu_n)_{n \geq 1}$ be a sequence of probability measures on $\R$ and consider the associated tail boundary flow.
\begin{enumlist}
\item Assume that $C > 0$ and $m \in \N$ such that $\mu_n([-C,C]) = 1$ for all $n \geq m$. If $\sum_{n=m}^\infty \Var \mu_n = +\infty$, the tail boundary flow is periodic. If $\sum_{n=m}^\infty \Var \mu_n < +\infty$, the tail boundary flow is given by the translation action $\R \actson \R$.

\item Assume that there exists a $C > 0$ such that $\inf_n \mu_n([-C,C]) > 0$. Also assume that the tail boundary flow is not periodic. There then exists a sequence $t_n \in [-C,C]$ such that for all $D \geq C$ and all $\eps > 0$,
    \begin{equation}\label{eq.weak-limit}
    \lim_{n \recht +\infty} \mu_n([-D,D] \setminus [t_n - \eps,t_n + \eps]) = 0 \; .
    \end{equation}

\item If there exists a $C > 0$ and a subset $I \subset \N$ such that writing $\nu_n = \mu_n([-C,C])^{-1} \mu_n|_{[-C,C]}$, we have
$$\sum_{n \in I} \mu_n(\R \setminus [-C,C]) <+\infty \quad\text{and}\quad \sum_{n \in I} \Var \nu_n = +\infty \; ,$$
then the tail boundary flow is periodic.

\item If there exists a $C > 0$ such that writing $\nu_n = \mu_n([-C,C])^{-1} \mu_n|_{[-C,C]}$, we have
$$\mu_n(\R \setminus [-C,C]) = o(\Var \nu_n) \;\;\text{when $n \recht +\infty$, and}\quad \sum_{n \in \N} \Var \nu_n = +\infty \; ,$$
then the tail boundary flow is periodic.

\item Assume that $\sup_n \int_\R t^2 \, d\mu_n(t) <+\infty$ and assume that there exists a $C > 0$ such that
$$\int_{\R \setminus [-C,C]} t^2 \, d\mu_n(t) = o(\Var \mu_n) \;\;\text{when $n \recht +\infty$, and}\quad \sum_{n=1}^\infty \Var \mu_n = +\infty \; .$$
Then the tail boundary flow is periodic.
\end{enumlist}
\end{theorem}
\begin{proof}
1.\ This is \cite[Theorem 3.1]{Ore66}.

2.\ Assume that $C > 0$ is such that $\inf_n \mu_n([-C,C]) > 0$. By the argument of \cite[Lemma 2.1]{Ore66}, we can choose a sequence $t_n \in [-C,C]$ such that for every $\eps > 0$, we have $\inf_n \mu_n([t_n-\eps,t_n+\eps]) > 0$. Assume now that we are given $D \geq C$ and $\eps > 0$ such that \eqref{eq.weak-limit} does not hold. We prove that the tail boundary flow is periodic.

Define the probability measures $\zeta_n$ on $\R$ by
$$\zeta_n(\cU) = \frac{\mu_n\bigl([-D,D] \cap (\cU - t_n)\bigr)}{\mu_n([-D,D])} \; .$$
Note that the measures $\zeta_n$ are all supported on the interval $I = [-C-D,C+D]$. Since \eqref{eq.weak-limit} does not hold, the sequence $\zeta_n(I \setminus [-\eps,\eps])$ is not converging to zero. Therefore, $\zeta_n$ is not weakly converging to the Dirac measure in $0$. Let $\zeta$ be any weak limit point of the sequence $(\zeta_n)_{n \in \N}$ and take $a \neq 0$ in the support of $\zeta$. Choose an arbitrary $\delta > 0$. We find that $\limsup_n \zeta_n([a-\delta,a+\delta]) > 0$. Then also
$$\limsup_n \mu_n([t_n+a-\delta,t_n+a+\delta]) > 0 \; ,$$
so that
$$\sum_{n=1}^\infty \mu_n\bigl(\bigl\{ t \in \R \bigm| |t-t_n-a| \leq \delta \}\bigr) = +\infty \; .$$
By \cite[Theorem 4.2]{Ore66}, every $F \in A$ satisfies $\al_a(F) = F$.

3.\ We first prove 3 in the case where $I = \N$. We use the notation \eqref{eq.not-1} and define the subsets $\cW_k \subset \Om$ by
$$\cW_k = \{\om \in \Om \mid \om_n \in [-C,C] \;\;\text{for all $n \geq k$}\;\} \; .$$
By our assumptions, $\eta(\Om \setminus \cW_k) \recht 0$ when $k \recht +\infty$. Whenever $F \in A \subset L^\infty(\Om,\eta)$ is a bounded measurable function on the tail boundary of $(\mu_n)_{n \geq 1}$ and $k \in \N$, the restriction of $F$ to $\cW_k$ can be viewed as a function on the tail boundary associated with the probability measures $\mu_n$, $1 \leq n \leq k-1$, and $\nu_n$, $n \geq k$. By point~1, this tail boundary flow is periodic. When $k$ varies, it follows from the second point of Proposition \ref{prop.T-set-tail} that all these restricted tail boundary flows have the same eigenvalue group. So either they are all trivial, in which case it follows that $F$ is constant a.e., or they are all conjugate to the translation action $\R \actson \R / p \Z$ for the same $p \neq 0$. In that case, $\al_p(F)|_{\cW_k} = F|_{\cW_k}$ for all $k$ and all $F \in A$, so that $\al_p(F) = F$ and $\R \actson^\al A$ is periodic.

We then consider the general case. Write
$$(\Om,\eta) = (\R,\mu_0) \times (Y,\eta_Y) \times (Z,\eta_Z) \quad\text{with}\quad Y = \prod_{n \in I} (\R,\mu_n) \quad\text{and}\quad Z = \prod_{n \in \N \setminus I} (\R,\mu_n) \; .$$
We know from the previous paragraph that the tail boundary flow defined by $(\mu_n)_{n \in I}$ is periodic. Take $p \neq 0$ such that $\al_p(F) = F$ for all bounded measurable functions on the tail boundary of $(\mu_n)_{n \in I}$. If now $F$ is a bounded measurable function on the tail boundary of $(\mu_n)_{n \in \N}$, then for a.e.\ $z \in Z$, the function $(t,y) \mapsto F(t,y,z)$ can be viewed as a function on the tail boundary of $(\mu_n)_{n \in I}$. We conclude that $\al_p(F) = F$.

4.\ By Lemma \ref{lem.select-subset} below, there exists a subset $I \subset \N$ such that the assumptions of point~3 hold.

5.\ Define $\nu_n = \mu_n([-C,C])^{-1} \mu_n|_{[-C,C]}$. Our assumptions imply that
$$|\Var \nu_n - \Var \mu_n| = o(\Var \mu_n) \quad\text{when $n \recht +\infty$.}$$
Our assumptions also imply that $\mu_n(\R \setminus [-C,C]) = o(\Var \mu_n)$. So, the conditions of point~4 are satisfied.
\end{proof}

In the proof of Theorem \ref{thm.tail-periodic-criterion}, we needed the following elementary lemma.

\begin{lemma}\label{lem.select-subset}
Let $a_n,b_n \in [0,+\infty)$ be bounded sequences and assume that $b_n = o(a_n)$ as $n \recht +\infty$. Assume that $\sum_{n=1}^\infty a_n = +\infty$. There then exists a subset $I \subset \N$ such that
$$\sum_{i \in I} b_i < +\infty \quad\text{and}\quad \sum_{i \in I} a_i = +\infty \; .$$
\end{lemma}
\begin{proof}
Assume that $a_n \leq M$ for all $n \in \N$. We can then inductively choose $n_1 < m_1 < n_2 < m_2 < \cdots$ such that $b_n/a_n \leq 2^{-k}$ for all $n \geq n_k$ and such that $\sum_{i=n_k}^{m_k} a_i$ lies between $M$ and $2M$. Defining $I = \bigcup_{k=1}^\infty ([n_i,m_i] \cap \N)$, the lemma is proven.
\end{proof}

\subsection{Strong conservativeness}\label{sec.conservative}

Recall that an essentially free, nonsingular group action $G \actson (X,\mu)$ is called \emph{conservative} (or recurrent) if for every nonnegligible subset $\cU \subset X$, there are infinitely many $g \in G$ such that $\mu(g \cdot \cU \cap \cU) > 0$. At the opposite end, the action is called \emph{dissipative} if it admits a fundamental domain: a measurable subset $\cU \subset X$ such that $(g \cdot \cU)_{g \in G}$ is a partition of $X$, up to measure zero.

Let $G \actson (X,\mu)$ be a nonsingular action. As in \cite{BKV19}, when $G$ is \emph{nonamenable}, we need a strengthening of conservativeness. The main reason for this is the lack of an ergodic theorem for nonsingular actions of nonamenable groups. This forces us to consider more general ergodic averages
\begin{equation}\label{eq.good-average}
\Phi_n(F) = \sum_{g \in G} p_n(g,x) \, F(g \cdot x) \quad\text{where}\quad \sum_{g \in G} p_n(g,x) = 1 \quad\text{for all $n \in \N$, $x \in X$.}
\end{equation}
By \cite[Lemma 4.1]{BKV19}, taking
\begin{equation}\label{eq.good-choice-pn}
p_n(g,x) = \sum_{h \in G} \eta_n(h) \, \frac{\eta_n(hg^{-1}) \frac{d(g^{-1}\mu)}{d\mu}(x)}{\sum_{k \in G} \eta_n(hk^{-1}) \frac{d(k^{-1}\mu)}{d\mu}(x)}
\end{equation}
where $\eta_n$ is an arbitrary sequence of finitely supported probability measures on $G$, we have that each $\Phi_n$ is a contraction on $L^1(X,\mu)$, i.e.\ $\|\Phi_n(F)\|_{1,\mu} \leq \|F\|_{1,\mu}$ for all $F \in L^1(X,\mu)$ and $n \in \N$.

We recall the following definition from \cite{BKV19}, saying that $\eta_n$ is strongly recurrent if in the definition of the ergodic averages in \eqref{eq.good-average}, we may ignore each individual term asymptotically.

\begin{definition}[{\cite[Definition 4.2]{BKV19}}]\label{def.strongly-conservative}
Let $G \actson (X,\mu)$ be a nonsingular action of a countable group $G$ on a standard probability space $(X,\mu)$. A sequence of finitely supported probability measures $\eta_n$ on $G$ is called \emph{strongly recurrent} for $G \actson (X,\mu)$ if $\lim_{n \recht +\infty} \|p_n(e,\cdot)\|_{1,\mu} = 0$, where $p_n$ is defined by \eqref{eq.good-choice-pn}.

We say that $G \actson (X,\mu)$ is \emph{strongly conservative} if such a strongly recurrent sequence of probability measures on $G$ exists.
\end{definition}

Note that by \cite[Formula (4.5)]{BKV19}, if $\eta_n$ is strongly recurrent, then for every $g \in G$, we have $\lim_{n \recht +\infty} \|p_n(g,\cdot)\|_{1,\mu} = 0$.

Also recall from \cite[Proposition 4.3]{BKV19} that every strongly conservative action is conservative; and that if $G$ is an infinite amenable group and $G \actson (X,\mu)$ is conservative, then the uniform probability measures $\eta_n$ on a F{\o}lner sequence $\cF_n \subset G$ are strongly recurrent.

As could already be seen from \eqref{eq.notation-C}, the following distance like function between equivalent probability measures $\mu \sim \nu$ on a standard Borel space $X$ plays an important role in this paper:
$$D(\mu,\nu) := \frac{1}{2} \log \int_X \frac{d\mu}{d\nu} \, d\mu \in [0,+\infty] \; .$$
This expression is not symmetric in $\mu$ and $\nu$, and hence does not define a metric. Nevertheless, it is a natural quantity to consider. Since $t \mapsto t^{-2}$ is a convex function, one has
$$1 \leq \Bigl(\int_X \sqrt{\frac{d\nu}{d\mu}} \, d\mu\Bigr)^{-2} \leq \int_X \frac{d\mu}{d\nu} \, d\mu \; ,$$
so that
$$H^2(\mu,\nu) \leq - \log (1 - H^2(\mu,\nu)) \leq D(\mu,\nu) \; .$$
In Lemma \ref{lem.estimate-projections}, we will see that $D(\mu,\nu)$ decreases if the measures $\mu$ and $\nu$ are pushed forward via $\pi : X \recht Y$, i.e.\ $D(\pi_*(\mu),\pi_*(\nu)) \leq D(\mu,\nu)$.

The following result, providing a checkable sufficient condition for conservativeness in terms of the ``distances'' $D(\mu, g \cdot \mu)$, is a refinement of the method of \cite[Proposition 4.1]{VW17}.

\begin{lemma}\label{lem.strongly-conservative}
Let $G$ be a countable group and $G \actson (X,\mu)$ a nonsingular action on the standard probability space $(X,\mu)$. Assume that
\begin{equation}\label{eq.enough-small}
\limsup_{s \recht +\infty} \frac{\log |\{g \in G \mid D(\mu,g^{\pm 1} \cdot \mu) \leq s\}|}{s} > 6 \; .
\end{equation}
Then, $G \actson (X,\mu)$ is strongly conservative. Moreover, there exists $\al_0 > 0$ and a sequence of probability measures $\eta_n$ on $G$ such that $\eta_n$ is strongly recurrent for the Maharam extension $G \actson (X \times \R , \mu \times \nu_\al)$ for all $\al \in (0,\al_0)$ and $d\nu_\al(t) = (\al/2) \exp(-\al |t|) \, dt$.
\end{lemma}
\begin{proof}
By \eqref{eq.enough-small}, we can fix $s_n \recht +\infty$, finite subsets $\cF_n \subset G$ and $\kappa_0 > 3$ such that
$$\int_X \frac{d\mu}{d(g^{\pm 1}\mu)}\, d\mu \leq \exp(s_n) \quad\text{for all $g \in \cF_n$, and}\quad |\cF_n| \geq \exp(\kappa_0 s_n) \quad\text{for all $n \in \N$.}$$
Choose any $3 < \kappa_1 < \kappa_0$. Put $\al_0 = (\kappa_1/3) - 1 > 0$. We prove that the conclusions of the lemma hold for $\eta_n$ the uniform probability measure on $\cF_n$.

Fix $\al \in (0,\al_0)$ and denote by $\mu_\al$ the probability measure $\mu \times \nu_\al$ on $X \times \R$. We denote by
$$\om(g,x,t) = \frac{d(g^{-1} \cdot \mu_\al)}{d\mu_\al}(x,t)$$
the Radon-Nikodym cocycle for the Maharam extension. Writing,
$$p_n(x,t) = \frac{1}{|\cF_n|} \sum_{h \in \cF_n} \Bigl(\sum_{k \in \cF_n^{-1} h} \om(k,x,t)\Bigr)^{-1} \; ,$$
we thus have to prove that $\lim_{n \recht +\infty} \|p_n\|_{1,\mu_\al} = 0$.

First note that by the H\"{o}lder inequality, we have for all $n \in \N$ and $g,h \in \cF_n$,
\begin{align*}
\int_X \Bigl(\frac{d\mu}{d(h^{-1}g \cdot \mu)}\Bigr)^{1/3} \, d\mu &= \int_X \Bigl(\frac{d\mu}{d(g \cdot \mu)}\Bigr)^{1/3} \, \Bigl(\frac{d(h\cdot \mu)}{d\mu}\Bigr)^{4/3} \, d\mu \\
&\leq \Bigl( \int_X \frac{d\mu}{d(g\cdot \mu)} \, d\mu \Bigr)^{1/3} \, \Bigl(\int_X \Bigl(\frac{d(h\cdot \mu)}{d\mu}\Bigr)^2 \, d\mu\Bigr)^{2/3} \\
&\leq \exp(s_n/3) \, \Bigl(\int_X \frac{d\mu}{d(h^{-1}\cdot \mu)} \, d\mu\Bigr)^{2/3} \leq \exp(s_n) \; .
\end{align*}
For all $\gamma \leq 1/3$, $n \in \N$ and $g,h \in \cF_n$, we have
\begin{equation}\label{eq.my-estimate-3}
\int_X \Bigl(\frac{d\mu}{d(h^{-1}g \cdot \mu)}\Bigr)^\gamma \, d\mu \leq \Bigl(\int_X \Bigl(\frac{d\mu}{d(h^{-1}g \cdot \mu)}\Bigr)^{1/3} \, d\mu\Bigr)^{3\gamma} \leq \exp(3 \gamma s_n) \leq \exp(s_n) \; .
\end{equation}

Write $\beta = 1/\kappa_1$. For all $s,t \in \R$,
$$\Bigl(\frac{d\nu_\al}{d(-s + \nu_\al)}(t)\Bigr)^\be = \exp(-\al\be|t| + \al\be|t+s|) \leq \exp(\al\be |s|) \leq \exp (\al \be s) + \exp(-\al \be s) \; .$$
We then find that
\begin{equation}\label{eq.est-rn-mah}
\om(g,x,t)^{-\be} \leq \Bigl(\frac{d\mu}{d(g^{-1} \cdot \mu)}(x)\Bigr)^{\be(1+\al)} + \Bigl(\frac{d\mu}{d(g^{-1} \cdot \mu)}(x)\Bigr)^{\be(1-\al)} \; .
\end{equation}
Since $\om(e,x,t) = 1$ and since $e \in \cF_n^{-1}g$ for all $g \in \cF_n$, we have that
$$\Bigl(\sum_{k \in \cF_n^{-1} h} \om(k,x,t)\Bigr)^{-1} \leq 1 \; .$$
Since $t \leq t^\be$ for all $t \in (0,1]$ and since $t \mapsto t^{-\be}$ is convex, we find that
\begin{align*}
p_n(x,t) &\leq \frac{1}{|\cF_n|} \sum_{h \in \cF_n} \Bigl(\sum_{k \in \cF_n^{-1} h} \om(k,x,t)\Bigr)^{-\be} =  \frac{1}{|\cF_n|^{1+\be}} \sum_{h \in \cF_n} \Bigl(\frac{1}{|\cF_n|}\sum_{k \in \cF_n^{-1} h} \om(k,x,t)\Bigr)^{-\be}\; , \\
&\leq \frac{1}{|\cF_n|^{2+\be}} \sum_{h \in \cF_n} \sum_{k \in \cF_n^{-1} h} \om(k,x,t)^{-\be} = \frac{1}{|\cF_n|^{2+\be}} \sum_{g,h \in \cF_n}  \om(g^{-1}h,x,t)^{-\be} \; .
\end{align*}
By construction, $\be(1+\al) \leq 1/3$ and $\be(1-\al) \leq 1/3$. Using \eqref{eq.est-rn-mah} and \eqref{eq.my-estimate-3}, we conclude that
$$\|p_n\|_{1,\mu_\al} \leq 2 \frac{\exp(s_n)}{|\cF_n|^{\be}} \leq 2 \exp((1-\be \kappa_0)s_n) \recht 0 \; ,$$
because $1- \be \kappa_0 < 0$.
\end{proof}

When $G \actson (X,\mu) = \prod_{g \in G} (X_0,\mu_g)$ is a nonsingular Bernoulli action, we defined in \eqref{eq.notation-C} the quantity $C(g)$. Considering the associated $1$-cocycle $c_g \in \ell^2(G) \ot L^2(X_0,\mu_e)$ and writing $G$ as the union of an increasing sequence of finite subsets $\cF_n \subset G$, we conclude using Fatou's lemma that
\begin{equation}\label{eq.estimate-C}
\begin{split}
\|c_g\|^2 & = H^2(\mu,g^{-1}\mu) \leq D(\mu,g^{-1}\mu) = \frac{1}{2} \log \int_X \Bigl(\prod_{h \in G} \frac{d\mu_h}{d\mu_{gh}}(x_h)\Bigr) \, d\mu(x)
\\ & \leq \liminf_{n \recht +\infty} \sum_{h \in \cF_n} \frac{1}{2} \log \int_{X_0} \frac{d\mu_h}{d\mu_{gh}}(x) \, d\mu_h(x) = \sum_{h \in G} D(\mu_h,\mu_{gh}) = C(g) \; .
\end{split}
\end{equation}

\section{Ergodicity and type for permutation actions}

Let $X_0$ be a standard Borel space and $(\mu_n)_{n \geq 1}$ a sequence of probability measures on $X_0$ that are all equivalent: $\mu_n \sim \mu_1$. We then consider the permutation action
$$\cS_\infty \actson (X,\mu) = \prod_{n=1}^\infty (X_0,\mu_n) : (\si^{-1} \cdot x)_n = x_{\si(n)} \; ,$$
where $\cS_\infty$ is the (countable) group of finite permutations of $\N = \{1,2,\ldots\}$. Since all $\mu_n$ are equivalent, the permutation action $\cS_\infty \actson (X,\mu)$ is nonsingular. We always implicitly assume that there is no $x_0 \in X_0$ such that all $\mu_n$ are the Dirac measure at $x_0$, since then $(X,\mu)$ essentially consists of a single point.

In \cite{AP77}, the ergodicity of $\cS_\infty \actson (X,\mu)$ is studied. When $X_0$ is a finite set, \cite[Theorem 1.6]{AP77} provides a necessary and sufficient condition for the ergodicity of $\cS_\infty \actson (X,\mu)$. When $X_0$ is infinite, in particular when the measures $\mu_n$ have a nonatomic part, only sufficient conditions for the ergodicity of the permutation action are known. One such sufficient condition, which fits well with our study of nonsingular Bernoulli actions, is the following:
\begin{equation}\label{eq.local-bound}
\text{for a.e.\ $x \in X_0$, we have that}\;\; \sup_{n \geq 1} \Bigl|\log \frac{d\mu_n}{d\mu_1}(x)\Bigr| < +\infty \; .
\end{equation}
This can be deduced from the main results of \cite{AP77}. For completeness, we include an elementary proof, which moreover illustrates well some of the weak limit techniques that are used in this paper.

\begin{proposition}\label{prop.ergodic-permutation-action}
Let $X_0$ be a standard Borel space and $(\mu_n)_{n \geq 1}$ a sequence of probability measures on $X_0$ that are all equivalent. Assume that \eqref{eq.local-bound} holds. Then the permutation action $\cS_\infty \actson (X,\mu) = \prod_{n \geq 1} (X_0,\mu_n)$ is ergodic and the measure $\mu$ is nonatomic.
\end{proposition}

Note that Proposition \ref{prop.ergodic-permutation-action} says in particular that under condition \eqref{eq.local-bound}, the permutation action $\cS_\infty \actson (X,\mu)$ is not of type~I.

\begin{proof}
Let $F \in L^\infty(X)$ be a bounded $\cS_\infty$-invariant function. We prove that $F$ is constant a.e. Fix $C > 0$ such that
\begin{equation}\label{eq.set-U}
\cU = \Bigl\{x \in X_0 \Bigm| C^{-1} \leq \frac{d\mu_n}{d\mu_1}(x) \leq C \;\text{for all $n \geq 1$} \; \Bigr\}
\end{equation}
has positive measure. By our assumption, taking $C \recht +\infty$, the measure of $X_0 \setminus \cU$ tends to zero.

Write $(\Xtil,\mutil) = (X_1,\mu_1) \times (X_1,\mu_1) \times \prod_{n \geq 2} (X_n,\mu_n)$, where we ``double the first variable''. Consider the natural measure preserving factor maps
$$\pi_0 : \Xtil \recht X : \pi_0(x) = (x_0,x_2,x_3,\ldots) \quad\text{and}\quad \pi_1 : \Xtil \recht X : \pi_1(x) = (x_1,x_2,x_3,\ldots) \; .$$
Whenever $H \in L^\infty(X,\mu)$, we write $V_i(H) = H \circ \pi_i$. Define the subsets $\cV_n \subset \Xtil$ by
$$\cV_n = \{x \in \Xtil \mid x_0,x_1,x_n \in \cU\} \; .$$
Denote by $p_n = 1_{\cV_n}$ the function that is equal to $1$ on $\cV_n$ and equal to zero elsewhere. Denote by $\si_n \in \cS_\infty$ the flip of $1$ and $n$. We also write $\si_n(H) = H \circ \si_n$ when $H \in L^\infty(X)$.

We claim that for all $H \in L^\infty(X)$,
\begin{equation}\label{eq.myclaim}
\lim_{n \recht +\infty} \|p_n \, V_0(\si_n(H)) - p_n \, V_1(\si_n(H))\|_2 = 0 \; .
\end{equation}
To prove this claim, define the subsets $\cW_n \subset X$ by
$$\cW_n := \{x \in X \mid x_1,x_n \in \cU\} \; .$$
Write $q_n = 1_{\cW_n}$. Note that for all $x \in \cW_n$, we have
$$\frac{d(\si_n \mu)}{d\mu}(x) = \frac{d\mu_n}{d\mu_1}(x_1) \, \frac{d\mu_1}{d\mu_n}(x_n) \leq C^2 \; .$$
So, for every $H \in L^\infty(X)$, we have that
$$\|p_n \, V_0(\si_n(H)) \|_2 = \|p_n \, V_0(q_n \, \si_n(H))\|_2 \leq \|q_n \, \si_n(H)\|_2 \leq C \, \|H\|_2 \; .$$
Therefore, it suffices to check \eqref{eq.myclaim} when $H$ only depends on finitely many coordinates $x_n$. In that case, $p_n \, V_0(\si_n(H)) = p_n \, V_1(\si_n(H))$ for all $n$ large enough. So, the claim is proven.

Applying the claim to the $\cS_\infty$-invariant function $F \in L^\infty(X)$, we conclude that
\begin{equation}\label{eq.nextstep}
\lim_{n \recht +\infty} \|p_n (V_0(F) - V_1(F))\|_2 = 0 \; .
\end{equation}
Note that
$$\mu_n(\cU) = \int_{\cU} \frac{d\mu_n}{d\mu_1} \, d\mu_1 \geq C^{-1} \, \mu_1(\cU) > 0 \; .$$
Take a subsequence $n_k$ such that $\mu_{n_k}(\cU) \recht a > 0$. Define the subset $\cV \subset \Xtil$ by
$$\cV := \{x \in \Xtil \mid x_0,x_1 \in \cU\}$$
and write $p = 1_\cV$. Since $p_{n_k} \recht a \, p$ weakly, it follows from \eqref{eq.nextstep} that $p(V_0(F) - V_1(F)) = 0$. This holds for all $C > 0$ and their corresponding set $\cU$ defined in \eqref{eq.set-U}. It follows that $V_0(F) = V_1(F)$.

So we have proven that $F$ does not depend on the first variable. Since $F$ is $\cS_\infty$-invariant, it follows that $F$ does not depend on the first $n$ variables, for all $n \in \N$. Hence, $F$ is constant a.e. So, $\cS_\infty \actson (X,\mu)$ is ergodic.

If $x \in X$ is an atom for $\mu$, we have that $\sum_{n=1}^\infty (1-\mu_n(x_n)) < +\infty$. In particular, for $n$ large enough, we have that $\mu_n(x_n) > 0$. Since all $\mu_n \sim \mu_1$, we can thus fix an atom $x_0$ for $\mu_1$. By \eqref{eq.local-bound}, we find $\eps > 0$ such that $\mu_n(x_0) \geq \eps$ for all $n \geq 1$. Since the sequence $1-\mu_n(x_n)$ is summable, this forces $x_n = x_0$ for all $n$ large enough. Define $x' \in X$ by $x'_n = x_0$ for all $n$. We conclude that the singleton $\{x'\}$ has positive measure and is $\cS_\infty$-invariant. By ergodicity, $\mu(x') = 1$. This means that every $\mu_n$ is the Dirac measure in $x_0$, contradicting our standing assumption.
\end{proof}

Under the hypothesis \eqref{eq.local-bound}, we can go much further and determine the Krieger type of the permutation action $\cS_\infty \actson (X,\mu)$. We formulate this as Theorems \ref{thm.T-invariant-permutation-action} and \ref{thm.permutation-condition-periodic} below. They will be deduced from the following technical result.

\begin{lemma}\label{lem.link-to-tail-boundary}
Let $X_0$ be a standard Borel space and $(\mu_n)_{n \geq 1}$ a sequence of probability measures on $X_0$ that are all equivalent. Assume that \eqref{eq.local-bound} holds. Whenever $\be : X_0 \recht \R$ is a measurable map, consider
\begin{equation}\label{eq.maps-gamma-n}
\gamma_n : X_0 \recht \R : \gamma_n(x) = -\log \frac{d\mu_n}{d\mu_1}(x) + \be(x) \quad\text{and}\quad \zeta_n = (\gamma_n)_*(\mu_n) \; .
\end{equation}
\begin{enumlist}
\item For any choice of $\be$, the tail boundary flow of $(\zeta_n)_{n \in \N}$ is a factor of the associated flow of $\cS_\infty \actson (X,\mu)$.
\item There exists a $\be : X_0 \recht \R$ such that the tail boundary flow of $(\zeta_n)_{n \in \N}$ is conjugate to the associated flow of $\cS_\infty \actson (X,\mu)$.
\item If $n_k$ is a subsequence such that $\lim_{k \recht +\infty} \log(d\mu_{n_k}/d\mu_1(x)) = \be(x)$ for a.e.\ $x \in X_0$, then the tail boundary flow of $(\zeta_n)_{n \in \N}$ is conjugate to the associated flow of $\cS_\infty \actson (X,\mu)$.
\end{enumlist}
\end{lemma}
\begin{proof}
Consider the Maharam extension $\cS_\infty \actson \R \times X$ of $\cS_\infty \actson (X,\mu)$. Let $\be : X_0 \recht \R$ be any measurable map and define $\gamma_n$ and $\zeta_n$ as in \eqref{eq.maps-gamma-n}. We first prove that the tail boundary flow of $(\zeta_n)_{n \in \N}$ is a factor of the associated flow of $\cS_\infty \actson (X,\mu)$. Realize this tail boundary flow as $\R \actson (B,\zeta)$. For every $n \geq 1$, write
\begin{equation}\label{eq.Yn}
(Y_n,\eta_n) = \prod_{m \geq n+1} (X_0,\mu_m) \; .
\end{equation}
By definition, we find $\R$-equivariant factor maps $\theta : \R \times X \recht B$ and $\theta_n : \R \times Y_n \recht B$ such that
$$\theta(t,x) = \theta_n(t+\gamma_1(x_1)+\cdots+\gamma_n(x_n),x_{n+1},x_{n+2},\ldots) \; .$$
It follows that $\theta$ is $\cS_\infty$-invariant. Hence, $\theta$ induces an $\R$-equivariant factor map from the associated flow of $\cS_\infty \actson (X,\mu)$ to $B$.

Fix an $\cS_\infty$-invariant $F \in L^\infty(\R \times X)$ that generates the von Neumann subalgebra of all $\cS_\infty$-invariant functions. We prove that there exists a measurable map $\be : X_0 \recht \R$ such that the map
\begin{equation}\label{eq.condition-beta}
\R \times X \recht \R : (t,x) \mapsto F(t-\be(x_1),x_1,x_2,\ldots)
\end{equation}
is essentially independent of the coordinate $x_1$.

For every $H \in L^\infty(\R)$, we denote by $\per(H) \subset \R$ the closed subgroup of all $a \in \R$ satisfying $H(a+t) = H(t)$ for a.e.\ $t \in \R$. By Proposition \ref{prop.ergodic-permutation-action}, the action $\cS_\infty \actson (X,\mu)$ is ergodic. So, by \cite[Lemma 6.6]{BKV19}, we are in one of the following cases.

{\bf Case 1.} The function $F$ is essentially constant.

{\bf Case 2.} For a.e.\ $x \in (X,\mu)$, we have that $\per F(\cdot,x) = \{0\}$.

{\bf Case 3.} There exists a $p \in \R$, $p \neq 0$, such that $\per F(\cdot,x) = p \Z$ for a.e.\ $x \in (X,\mu)$.

In case~1, \eqref{eq.condition-beta} obviously holds for any choice of $\be$. Next assume that we are in case~2. We equip $\R$ with the probability measure $\nu$ given by $d\nu(t) = (1/2) \exp(-|t|) \, dt$. For every $n \in \N$, we denote by $\si_n \in \cS_\infty$ the flip of $1$ and $n$. We write
$$\al_n(x) = \log \frac{d\mu_n}{d\mu_1}(x) \; .$$
For all $H \in L^\infty(\R \times X)$, we write $(\si_n(H))(t,x) = H(\si_n \cdot (t,x))$ and $(\Gamma_n(H))(t,x) = H(t-\al_n(x_1),x)$.

Using the notation in \eqref{eq.Yn}, as in the proof of Proposition \ref{prop.ergodic-permutation-action}, we define $(\Xtil,\mutil) = (X_0,\mu_1) \times (X_0,\mu_1) \times (Y_1,\eta_1)$ and we consider the measure preserving factor maps
\begin{align*}
& \pi_0 : (\R \times \Xtil,\nu \times \mutil) \recht (\R \times X,\nu \times \mu) : \pi_0(t,x) = (t,x_0,x_2,x_3,\ldots) \quad\text{and}\\
& \pi_1 : (\R \times \Xtil,\nu \times \mutil) \recht (\R \times X,\nu \times \mu) : \pi_1(t,x) = (t,x_1,x_2,x_3,\ldots) \; .
\end{align*}
For $H \in L^\infty(\R \times X)$, we write $V_i(H) = H \circ \pi_i$. Fix $C > 0$ and define $\cU \subset X_0$ as in \eqref{eq.set-U}. Assume that $\cU$ has positive measure. Define $\cV_n \subset \R \times \Xtil$ by
$$\cV_n := \{(t,x) \in \R \times \Xtil \mid x_0,x_1,x_n \in \cU\} \; .$$
Put $p_n = 1_{\cV_n}$. We claim that
\begin{equation}\label{eq.newclaim}
\lim_{n \recht +\infty} \|p_n V_0(\Gamma_n(\si_n(H))) - p_n V_1(\Gamma_n(\si_n(H)))\|_2 = 0 \quad\text{for all $H \in L^\infty(\R \times X)$.}
\end{equation}
Since there exists a $D > 0$ such that $\|p_n V_i(\Gamma_n(\si_n(H)))\|_2 \leq D \, \|H\|_2$ for all $n \in \N$, $i \in \{0,1\}$ and $H \in L^\infty(\R \times X)$, it suffices to prove \eqref{eq.newclaim} for a function $H$ that only depends on finitely many coordinates. In that case, we have that $p_n V_0(\Gamma_n(\si_n(H))) = p_n V_1(\Gamma_n(\si_n(H)))$ for all $n$ large enough. So, \eqref{eq.newclaim} is proven.

In particular, \eqref{eq.newclaim} holds for the $\cS_\infty$-invariant function $F \in L^\infty(\R \times X)$. We then take a subsequence $n_k$ such that
\begin{equation}\label{eq.conv-ae}
\lim_{k \recht +\infty} |1_{\cU}(x_{n_k}) \, F(t-\al_{n_k}(x_0),x_0,x_2,\ldots) - 1_{\cU}(x_{n_k}) \, F(t-\al_{n_k}(x_1),x_1,x_2,\ldots)| = 0
\end{equation}
for a.e.\ $(t,x) \in \R \times \cU \times \cU \times Y_1$. Since $\inf_n \mu_n(\cU)\geq C^{-1}\mu_1(\cU) > 0$, we may take this subsequence such that $\mu_{n_k}(\cU) \recht \rho > 0$.

Define the subset $\cUtil \subset \cU \times \cU$ consisting of all $(x_0,x_1) \in \cU \times \cU$ with the following two properties: for a.e.\ $(t,x_2,x_3,\ldots) \in \R \times Y_1$, \eqref{eq.conv-ae} holds, and for a.e.\ $(x_2,x_3,\ldots) \in Y_1$ and for all $i \in \{0,1\}$, we have that $\per(F(\cdot,x_i,x_2,x_3,\ldots)) = \{0\}$. Note that $(\cU \times \cU) \setminus \cUtil$ has measure zero.

Fix $(x_0,x_1) \in \cUtil$. Let $(a_0,a_1) \in \R^2$ be a limit point of the bounded sequence $(\al_{n_k}(x_0),\al_{n_k}(x_1))$. Take a subsequence $k_m$ such that
$$\lim_{m \recht +\infty} (\al_{n_{k_m}}(x_0),\al_{n_{k_m}}(x_1)) = (a_0,a_1) \; .$$
For $i \in \{0,1\}$ and $m \in \N$, define $H_{i,m} \in L^\infty(\R \times Y_1)$ by
$$H_{i,m}(t,x_2,x_3,\ldots) = 1_{\cU}(x_{n_{k_m}}) \, F(t-\al_{n_{k_m}}(x_i),x_i,x_2,x_3,\ldots) \; .$$
For each $i \in \{0,1\}$, we have that $(H_{i,m})_{m \in \N}$ is a bounded sequence in $L^2(\R \times Y_1)$ that converges weakly to
$$(t,x_2,x_3,\ldots) \mapsto \rho \, F(t-a_i,x_i,x_2,x_3,\ldots) \; .$$
By our assumptions, $H_{0,m} - H_{1,m} \recht 0$ a.e. It thus follows that
$$F(t-a_0,x_0,x_2,\ldots) = F(t-a_1,x_1,x_2,\ldots)$$
for a.e.\ $(t,x_2,x_3,\ldots) \in \R \times Y_1$. We conclude that for all $(x_0,x_1) \in \cUtil$ and every limit point $b$ of $(-\al_{n_k}(x_0) + \al_{n_k}(x_1))_k$, we have
$$F(t,x_0,x_2,\ldots) = F(t-b,x_1,x_2,\ldots)$$
for a.e.\ $(t,x_2,x_3,\ldots) \in \R \times Y_1$. Since $\per(F(\cdot,x_i,x_2,x_3,\ldots)) = \{0\}$, all these limit points $b$ must be equal. We have thus shown that for a.e.\ $(x_0,x_1) \in \cU \times \cU$, the sequence $-\al_{n_k}(x_0) + \al_{n_k}(x_1)$ converges to a limit $\Om(x_0,x_1)$ satisfying
\begin{equation}\label{eq.play-with-ae}
F(t,x_0,x_2,\ldots) = F(t-\Om(x_0,x_1),x_1,x_2,\ldots)
\end{equation}
for a.e.\ $(t,x_2,x_3,\ldots) \in \R \times Y_1$. We can then fix $x_0 \in \cU$ such that \eqref{eq.play-with-ae} holds for a.e.\ $x_1 \in \cU$ and a.e.\ $(t,x_2,x_3,\ldots) \in \R \times Y_1$. Writing $\be(x_1) := \Om(x_0,x_1)$, we conclude that the function $F(t-\be(x_1),x_1,x_2,\ldots)$ is essentially independent of the variable $x_1 \in \cU$. By the lack of periodicity, this function $\be : \cU \recht \R$ is uniquely determined a.e.\ up to a constant. Therefore, taking larger and larger $C > 0$, we find a measurable $\be : X_0 \recht \R$ such that \eqref{eq.condition-beta} holds.

In case~3, we replace $\R$ by $\R/p\Z$ and view $F$ as a function on $\R/p\Z \times X$ with the property that for a.e.\ $x \in X$, the periodicity of $F(\cdot,x)$ is trivial, as a subgroup of $\R/p\Z$. The same argument as above, again provides a measurable $\be : X_0 \recht \R$ such that \eqref{eq.condition-beta} holds.

Define $\gamma_n : X_0 \recht \R : \gamma_n(x) = -\al_n(x) + \be(x)$. Since $F$ is $\cS_\infty$-invariant and since \eqref{eq.condition-beta} holds, it follows that
$$(t,x) \mapsto F(t-\gamma_1(x_1)-\gamma_2(x_2)-\cdots - \gamma_n(x_n),x_1,\ldots,x_n,x_{n+1},\ldots)$$
is essentially independent of the coordinates $x_1,\ldots,x_n$. So, we find $F_n \in L^\infty(\R \times Y_n)$ such that
\begin{equation}\label{eq.thatsit}
F(t,x) = F_n(t+\gamma_1(x_1)+\cdots+\gamma_n(x_n),x_{n+1},\ldots)
\end{equation}
for a.e.\ $(t,x) \in \R \times X$. We can then view $F$ as a bounded function on the tail boundary of $(\zeta_n)_{n \in \N}$ where $\zeta_n = (\gamma_n)_*(\mu_n)$. Conversely, any $F \in L^\infty(\R \times X)$ that can be represented as in \eqref{eq.thatsit} with $F_n \in L^\infty(\R \times Y_n)$ for all $n \in \N$, is $\cS_\infty$-invariant. We have thus proven that the associated flow of $\cS_\infty \actson (X,\mu)$ is conjugate to the tail boundary flow of $(\zeta_n)_{n \geq 1}$.

To prove the third point, if we are given from the start a subsequence $n_k$ such that $\al_{n_k}(x)$ converges pointwise a.e., we can choose the subsequence in \eqref{eq.conv-ae} as a subsubsequence of $n_k$, so that indeed $\be$ can be chosen as the pointwise limit a.e.\ of $(\al_{n_k})_k$.
\end{proof}

As we did with the tail boundary flow, we first describe the $T$-invariant of $\cS_\infty \actson (X,\mu)$ and then provide sufficient conditions to rule out type III$_0$. When $X_0 = \{0,1\}$, the following result was proven in \cite[Theorem 1.2]{SV77}. We again use the notation $d(t,K)$ to denote the distance from $t \in \R$ to a closed subset $K \subset \R$.

\begin{theorem}\label{thm.T-invariant-permutation-action}
Let $X_0$ be a standard Borel space and $(\mu_n)_{n \geq 1}$ a sequence of probability measures on $X_0$ that are all equivalent. Assume that \eqref{eq.local-bound} holds and consider the permutation action $\cS_\infty \actson (X,\mu) = \prod_{n \geq 1} (X_0,\mu_n)$.
\begin{enumlist}
\item The action is of type II$_1$ if and only if $\mu \sim \nu^\N$ for some probability measure $\nu \sim \mu_1$ on $X_0$.

\item The action is of type II$_\infty$ if and only if there exists a probability measure $\nu \sim \mu_1$ on $X_0$ and subsets $\cU_n \subset X_0$ such that
$$\sum_{n=1}^\infty \mu_n(X_0 \setminus \cU_n) < +\infty \;\; , \;\; \sum_{n=1}^\infty H^2(\mu_n , \nu(\cU_n)^{-1} \nu|_{\cU_n}) < +\infty \quad\text{and}\quad \sum_{n=1}^\infty \nu(X_0 \setminus \cU_n) = +\infty \; .$$

\item For $p \neq 0$, we have that $2\pi /p$ belongs to the $T$-invariant of the permutation action if and only if there exists a probability measure $\nu \sim \mu_1$ on $X_0$ and $t_n \in \R$ such that
    \begin{equation}\label{eq.condition-T-inv-perm}
    \sum_{n=1}^\infty \int_{X_0} d\Bigl(\log \frac{d\mu_n}{d\nu}(x) - t_n , p \Z\Bigr)^2 \, d\mu_n(x) < +\infty \; .
    \end{equation}
    In that case, we find probability measures $\nu_n \sim \mu_n$ and $\rho_n > 0$ such that
    $$\sum_{n=1}^\infty H^2(\mu_n,\nu_n) < +\infty \quad\text{and}\quad \frac{d\nu_n}{d\nu}(x) \in \rho_n \exp(p \Z) \quad\text{for all $n \in \N$ and a.e.\ $x \in X_0$.}$$
\end{enumlist}
\end{theorem}

\begin{proof}
By Proposition \ref{prop.ergodic-permutation-action}, the action $\cS_\infty \actson (X,\mu)$ is ergodic and $\mu$ is nonatomic. So, $\cS_\infty \actson (X,\mu)$ is not of type~I. By Lemma \ref{lem.link-to-tail-boundary}, we can fix a measurable function $\be : X_0 \recht \R$ such that the associated flow of $\cS_\infty \actson (X,\mu)$ is given by the tail boundary flow of the probability measures $\zeta_n = (\gamma_n)_*(\mu_n)$ defined in \eqref{eq.maps-gamma-n}. Write $\al_n(x) = \log d\mu_n/d\mu_1(x)$.

First assume that the action is semifinite, i.e.\ of type II$_1$ or type II$_\infty$. By the first point of Proposition \ref{prop.T-set-tail}, we find $t_n \in \R$ such that for $\kappa > 0$,
\begin{equation}\label{eq.version}
\sum_{n=1}^\infty \int_{X_0} T_\kappa(\al_n(x) - t_n - \be(x))^2 \, d\mu_n(x) < +\infty \; .
\end{equation}
For every $C > 0$, write
$$\cV_C = \bigl\{x \in X_0 \bigm| \sup_n |\log \al_n(x)| \leq C \;\;\text{and}\;\; |\be(x)|\leq C \bigr\} \; .$$
Since \eqref{eq.local-bound} holds, we have that $\mu_1(X_0 \setminus \cV_C) \recht 0$ when $C \recht +\infty$.

Since the Radon-Nikodym derivatives $d\mu_1/d\mu_n$ are uniformly bounded on $\cV_C$, we get that
$$\sum_{n=1}^\infty \int_{\cV_C} T_\kappa(\al_n(x) - t_n - \be(x))^2 \, d\mu_1(x) < +\infty$$
for all $C > 0$. It follows in particular that for all $C > 0$, the sequence $\al_n(x) - t_n - \be(x)$ converges to zero for a.e.\ $x \in \cV_C$. Hence, the sequence converges to zero for a.e.\ $x \in X_0$. Since $\al_n(x)$ is a bounded sequence when $x \in \cV_C$, we conclude that the sequence $t_n$ is bounded. Write $D = \sup_n |t_n|$. By Fatou's lemma,
$$\int_{X_0} \exp(\be(x)) \, d\mu_1(x) \leq \liminf_n \exp(-t_n) \int_{X_0} \frac{d\mu_n}{d\mu_1} \, d\mu_1 \leq \exp(D) < +\infty \; .$$
Adding a constant to $\be$ and modifying $t_n$ accordingly, we may assume that $\be = \log (d\nu/d\mu_1)$ for some probability measure $\nu \sim \mu_1$.

So, \eqref{eq.version} is saying that
\begin{equation}\label{eq.newversion}
\sum_{n=1}^\infty \int_{X_0} T_\kappa\Bigl( \log \frac{d\nu}{d\mu_n}(x) + t_n \Bigr)^2 \, d\mu_n(x) < +\infty
\end{equation}
for all $\kappa > 0$. Fix any $C > 0$ and define
$$\cU_n = \Bigl\{x \in X_0 \Bigm| |\log \frac{d\nu}{d\mu_n}(x) + t_n| \leq C \Bigr\} \; .$$
It follows from \eqref{eq.newversion} that
\begin{equation}\label{eq.good-eq}
\sum_{n=1}^\infty \mu_n(X_0 \setminus \cU_n) < +\infty \; .
\end{equation}
It also follows from \eqref{eq.newversion} that
$$\sum_{n=1}^\infty \int_{\cU_n} \Bigl( \log \frac{d\nu}{d\mu_n}(x) + t_n \Bigr)^2 \, d\mu_n(x) < +\infty \; .$$
Since the function $s \mapsto \exp(s/2)$ is Lipschitz on $[-C,C]$, we conclude that
$$\sum_{n=1}^\infty \int_{\cU_n} \Bigl| \exp(t_n/2) \sqrt{\frac{d\nu}{d\mu_n}}(x) -1 \Bigr|^2 \, d\mu_n(x) < +\infty \; .$$
Using \eqref{eq.good-eq}, we get that
\begin{equation}\label{eq.almost-done-one}
\sum_{n=1}^\infty \int_{X_0} \Bigl| \exp(t_n/2) \sqrt{\frac{d\nu}{d\mu_n}}(x) \, 1_{\cU_n}(x) -1 \Bigr|^2 \, d\mu_n(x) < +\infty \; .
\end{equation}

Whenever $r \in [-1,1]$ and $a \geq 0$, one has $1-r \leq 2(a^2 - 2 a r + 1)$. When $\xi_0,\eta$ are unit vectors in a Hilbert space, the real part of $\langle \xi_0,\eta\rangle$ belongs to $[-1,1]$. It follows that
$$\|\xi_0 - \eta \| \leq 2 \, \| a \xi_0 - \eta \|$$
for all $a \geq 0$. This means that for every nonzero vector $\xi$ and every unit vector $\eta$ in a Hilbert space, we have that
\begin{equation}\label{eq.hilbert-space-ineq}
\Bigl\| \|\xi\|^{-1} \xi - \eta \Bigr\| \leq 2 \, \|\xi - \eta\| \; .
\end{equation}
Define the probability measures $\nu_n = \nu(\cU_n)^{-1} \nu|_{\cU_n}$. It then follows from \eqref{eq.almost-done-one} and \eqref{eq.hilbert-space-ineq} that
$$\sum_{n=1}^\infty H^2(\nu_n,\mu_n) = \sum_{n=1}^\infty \int_{X_0} \Bigl| \nu(\cU_n)^{-1/2} \sqrt{\frac{d\nu}{d\mu_n}}(x) \, 1_{\cU_n}(x) -1 \Bigr|^2 \, d\mu_n(x) < +\infty \; .$$
It then also follows that
\begin{equation}\label{eq.my-function-F}
F : X \recht \R : F(x) = \sum_{n=1}^\infty \bigl(\log \frac{d\mu_n}{d\nu}(x_n) + \log(\nu(\cU_n)) \bigr)
\end{equation}
is unconditionally a.e.\ convergent. By construction,
$$F(\si(x)) - F(x) = \log \frac{d(\si^{-1} \cdot \mu)}{d\mu}(x) \quad\text{for all $\si \in \cS_\infty$ and a.e.\ $x \in X$.}$$
So, the $\sigma$-finite measure $\nutil \sim \mu$ on $X$ given by $d\nutil / d\mu = \exp(-F)$ is $\cS_\infty$-invariant. Defining
$$\cW_k = \{x \in X \mid x_n \in \cU_n \;\;\text{for all $n \geq k$}\;\} \; ,$$
we find that
$$\nutil(\cW_k) = \prod_{n=1}^{k-1} \nu(\cU_n)^{-1} \; .$$
Note that $\mu(X \setminus \cW_k) \recht 0$ if $k \recht +\infty$. We conclude that $\nutil$ is a finite measure if and only if $\sum_n \nu(X_0 \setminus \cU_n) < +\infty$. If this holds, we have that $\mu \sim \nu^\N$ and that $\cS_\infty \actson (X,\mu)$ is of type II$_1$. Conversely, if $\mu \sim \nu^\N$, it is immediate that $\cS_\infty \actson (X,\mu)$ is of type II$_1$.

When $\sum_n \nu(X_0 \setminus \cU_n) = +\infty$, the measure $\nutil$ is infinite. We have proven that all the conclusions of the second point of the theorem hold and that $\cS_\infty \actson (X,\mu)$ is of type II$_\infty$. If the conclusions of the second point of the theorem hold, then \eqref{eq.my-function-F} provides a well defined function $F$ and $\exp(-F)$ is the density for an infinite $\cS_\infty$-invariant measure, so that $\cS_\infty \actson (X,\mu)$ is of type II$_\infty$.

To prove point~3, assume that $p \neq 0$ and that $2\pi / p$ belongs to the $T$-invariant of $\cS_\infty \actson (X,\mu)$. By the second point of Proposition \ref{prop.T-set-tail}, we find $t_n \in \R$ such that
$$\sum_{n=1}^\infty \int_{X_0} d(\al_n(x) - t_n - \be(x),p\Z)^2 \, d\mu_n(x) < +\infty \; .$$
Reducing modulo $p$, we may assume that $t_n$ and $\be$ are bounded. Adding a constant to $\be$ and modifying $t_n$ accordingly, we may moreover assume that $\be = \log d\nu / d\mu_1$ for some probability measure $\nu \sim \mu_1$. Uniquely define $P_n : X_0 \recht p \Z$ such that
$$0 \leq \al_n(x) - t_n - \be(x) - P_n(x) < p \quad\text{for all $n \in \N$, $x \in X_0$.}$$
Defining $\nu_n \sim \mu_1$ as the unique probability measure such that $d\nu_n / d \mu_1$ is a multiple of $\exp(\be + P_n)$ and appropriately modifying $t_n$, we find that
$$\sum_{n=1}^\infty \int_{X_0} \Bigl| \log \frac{d\nu_n}{d\mu_n}(x) + t_n \Bigr|^2 \, d\mu_n(x) < +\infty \; ,$$
where the sequence $t_n$ is still bounded. The same argument as above implies that
$$\sum_{n=1}^\infty H^2(\nu_n,\mu_n) < +\infty \; .$$
By construction, $d\nu_n/d\nu$ is a multiple of $\exp(P_n)$. So the conclusions of point~3 hold.

Conversely, if the conclusions of point~3 hold, we have that $\mu \sim \nutil = \prod_n \nu_n$ and $d(\si \cdot \nutil)/d\nutil$ only takes values in $\exp(p \Z)$, implying that $2\pi/p$ belongs to the $T$-invariant.
\end{proof}

\begin{remark}\label{rem.maps-sums}
Assume that the hypotheses of Theorem \ref{thm.T-invariant-permutation-action} hold. Consider the Maharam extension $\cS_\infty \actson X \times \R$ of the permutation action $\cS_\infty \actson (X,\mu)$.

In point~2 of Theorem \ref{thm.T-invariant-permutation-action}, the sum
$$F : X \recht \R : F(x) = \sum_{n=1}^\infty \Bigl( \log \frac{d\mu_n}{d\nu}(x_n) + \log(\nu(\cU_n)) \Bigr)$$
is unconditionally a.e.\ convergent. The equivalent infinite $\cS_\infty$-invariant measure on $X$ is given by $\exp(-F(x))\, d\mu(x)$. The $\cS_\infty$-invariant and $\R$-equivariant map $(x,t) \mapsto t-F(x)$ identifies the associated flow of $\cS_\infty \actson (X,\mu)$ with the translation action $\R \actson \R$.

In point~3 of Theorem \ref{thm.T-invariant-permutation-action}, if $2\pi / p$ belongs to the $T$-invariant, then the sum
$$F : X \recht \R/p\Z : F(x) = \sum_{n=1}^\infty \Bigl( \log \frac{d\mu_n}{d\nu}(x_n) - \log \rho_n \Bigr) + p\Z$$
is unconditionally a.e.\ convergent. The $\cS_\infty$-invariant and $\R$-equivariant map $X \times \R \recht \R/p\Z : (x,t) \mapsto t-F(x)$ provides the eigenfunction with eigenvalue $2\pi / p$ for the associated flow of $\cS_\infty \actson (X,\mu)$.
\end{remark}

Once we can rule out that $\cS_\infty \actson (X,\mu)$ is of type III$_0$, Theorem \ref{thm.T-invariant-permutation-action} provides a complete description of the type of the permutation action. The following result provides several sufficient conditions ruling out type III$_0$.

\begin{theorem}\label{thm.permutation-condition-periodic}
Let $X_0$ be a standard Borel space and $(\mu_n)_{n \geq 1}$ a sequence of probability measures on $X_0$ that are all equivalent. Assume that \eqref{eq.local-bound} holds and consider the permutation action $\cS_\infty \actson (X,\mu) = \prod_{n \geq 1} (X_0,\mu_n)$.
\begin{enumlist}
\item If there exists a $C > 0$ such that $C^{-1} \leq d\mu_n/d\mu_1(x) \leq C$ for all $n \in \N$ and a.e.\ $x \in X_0$, then the permutation action is of type II$_1$ or of type III$_\lambda$ with $\lambda \in (0,1]$.

\item If there is no probability measure $\nu \sim \mu_1$ and bounded sequence $t_n \in \R$ such that the functions
$$X_0 \recht \R : x \mapsto \log \frac{d\mu_n}{d\nu}(x) - t_n$$
are converging to zero in measure when $n \recht +\infty$, then the permutation action is of type III$_\lambda$ for some $\lambda \in (0,1]$.

\item If there are at least two distinct probability measures on $X_0$ that arise as limit points in Hellinger distance of the sequence $(\mu_n)_{n \in \N}$, then the permutation action is of type III$_\lambda$ for some $\lambda \in (0,1]$.
\end{enumlist}
\end{theorem}
\begin{proof}
By Lemma \ref{lem.link-to-tail-boundary}, we can fix a measurable function $\be : X_0 \recht \R$ such that the associated flow of $\cS_\infty \actson (X,\mu)$ is given by the tail boundary flow of the probability measures $\nu_n = (\gamma_n)_*(\mu_n)$ defined in \eqref{eq.maps-gamma-n}. Write $\al_n(x) = \log d\mu_n/d\mu_1(x)$.

We start by proving the following claim: if the associated flow of $\cS_\infty \actson (X,\mu)$ is not periodic, there exists a bounded sequence $t_n \in \R$ such that the sequence of functions $\al_n - t_n$ converges to $\be$ in measure. For every $C > 0$, denote
$$\cU_C = \bigl\{x \in X_0 \bigm| |\al_n(x)| \leq C \;\;\text{for all $n \in \N$, and}\;\; |\be(x)| \leq C \bigr\} \; .$$
By our assumption \eqref{eq.local-bound}, we have that $\mu_1(X_0 \setminus \cU_C) \recht 0$ when $C \recht +\infty$. Fix $C > 0$ such that $\cU_C$ has positive measure. For all $n \in \N$,
$$\nu_n([-2C,2C]) \geq \mu_n(\cU_C) \geq \exp(-C) \, \mu_1(\cU_C) > 0 \; .$$
By point~2 of Theorem \ref{thm.tail-periodic-criterion}, we find a sequence $t_n \in [-2C,2C]$ such that for all $D \geq 2C$ and $\eps > 0$,
$$\lim_{n \recht +\infty} \nu_n([-D,D] \setminus [t_n - \eps,t_n+\eps]) = 0 \; .$$
This implies that for all $C > 0$ and $\eps > 0$,
$$\lim_{n \recht +\infty} \mu_n\bigl( \bigl\{ x \in \cU_C \bigm| |\al_n(x) - \be(x) - t_n| > \eps \bigr\} \bigr) = 0 \; .$$
Since $d\mu_1 / d\mu_n (x) \leq \exp(C)$ for all $x \in \cU_C$, we also find that for every $\eps > 0$ and every $C > 0$
$$\lim_{n \recht +\infty} \mu_1\bigl( \bigl\{ x \in \cU_C \bigm| |\al_n(x) - \be(x) - t_n| > \eps \bigr\} \bigr) = 0 \; .$$
Since $\lim_{C \recht +\infty}\mu_1(X_0 \setminus \cU_C) = 0$, the claim is proven.

If the associated flow of $\cS_\infty \actson (X,\mu)$ is periodic, each of the conclusions in 1, 2 and 3 hold. So for the rest of the proof, we may assume that we have a bounded sequence $t_n \in \R$ such that $\al_n - t_n$ converges to $\be$ in measure.

1.\ Since the functions $\al_n$ are uniformly bounded, also $\be$ is a bounded function. So, there exists a $D > 0$ such that $\nu_n([-D,D]) = 1$ for all $n \in \N$. By point~1 of Theorem \ref{thm.tail-periodic-criterion}, either the associated flow is periodic, or $\sum_{n=1}^\infty \Var(\nu_n) < +\infty$. In the latter case, we find a bounded sequence $s_n \in \R$ such that
$$\sum_{n=1}^\infty \int_{X_0} |\al_n(x) - \be(x) - s_n|^2 \, d\mu_n(x) < +\infty \; .$$
Since all functions and sequences involved are uniformly bounded, denoting by $\nu \sim \mu_1$ the unique probability measure such that $d\nu/d\mu_1$ is a multiple of $\exp(\be)$, it follows that
$$\sum_{n=1}^\infty H^2(\mu_n,\nu) = \sum_{n=1}^\infty \int_{X_0} \Bigl| \sqrt{\frac{d\nu}{d\mu_n}}(x) - 1 \Bigr|^2 \, d\mu_n(x) < +\infty \; .$$
So, $\mu \sim \nu^\N$ and $\cS_\infty \actson (X,\mu)$ is of type II$_1$.

2.\ It suffices to prove that $\exp(\be)$ is $\mu_1$-integrable, since we can then define $\nu$ as the unique probability measure such that $d\nu/d\mu_1$ is a multiple of $\exp(\be)$. Write $D = \sup_n |t_n|$. For every $C > 0$, write
$$\cU_C = \bigl\{x \in X_0 \bigm| \sup_n |\al_n(x)| \leq C \;\;\text{and}\;\; |\be(x)| \leq C \bigr\} \; .$$
Because of \eqref{eq.local-bound}, we have that $\bigcup_{C > 0} \cU_C$ is conull. Since $\al_n - t_n \recht \be$ in measure and since the exponential function is Lipschitz on bounded sets, we find that for every $C > 0$,
$$\lim_{n \recht +\infty} \exp(-t_n) \int_{\cU_C} \frac{d\mu_n}{d\mu_1} \, d\mu_1 = \int_{\cU_C} \exp(\be(x)) \, d\mu_1(x) \; .$$
The expression at the left is bounded by $\exp(D)$. We conclude that
$$\int_{\cU_C} \exp(\be(x)) \, d\mu_1(x) \leq \exp(D) \quad\text{for all $C > 0$.}$$
So, $\exp(\be)$ is $\mu_1$-integrable.

3.\ By point~2, we may assume that there is a probability measure $\nu \sim \mu_1$ on $X_0$ and a bounded sequence $t_n \in \R$ such that $\log d\mu_n / d\nu - t_n$ converges to zero in measure. If $\nu' \sim \mu_1$ is any probability measure that is a limit point in Hellinger distance of the sequence $(\mu_n)_{n \in \N}$, we prove that $\nu' = \nu$. Take a subsequence $n_k$ such that $H(\mu_{n_k},\nu') \recht 0$. Then,
$$\lim_{k \recht +\infty} \int_{X_0} \Bigl| \sqrt{\frac{d\mu_{n_k}}{d\nu'}}(x) - 1 \Bigr|^2 \, d\nu'(x) = 0 \; .$$
It follows that $\sqrt{d\mu_{n_k}/d\nu'} \recht 1$ in measure, so that $\al_{n_k} \recht \log(d\nu' / d\mu_1)$ in measure. We already know that $\al_n - t_n \recht \log(d\nu / d\mu_1)$ in measure. So, $\log(d\nu / d\mu_1) - \log(d\nu'/d\mu_1)$ is constant a.e. This means that $d\nu' / d\nu$ is constant a.e. Since both $\nu$ and $\nu'$ are probability measures, we get that $\nu = \nu'$.
\end{proof}

\section{Ergodicity and type of nonsingular Bernoulli actions}

We prove Theorems \ref{thmB.type-amenable} and \ref{thmC.type-nonamenable}, which are special cases of the main results of this section: Theorems \ref{thm.Krieger-type-Bernoulli} and \ref{thm.rule-out-III0} below.

Recall that a group homomorphism $\om : G \recht \T$ is called an eigenvalue of an ergodic nonsingular action $G \actson (Y,\eta)$ if there exists a measurable $F : Y \recht \T$ such that $F(g \cdot y) = \om(g) \, F(y)$ for all $g \in G$ and a.e.\ $y \in Y$.

\begin{theorem}\label{thm.Krieger-type-Bernoulli}
Let $G$ be a countably infinite group and $G \actson (X,\mu) = \prod_{g \in G} (X_0,\mu_g)$ a nonsingular Bernoulli action. Assume that the weak boundedness property \eqref{eq.boundedness-weak} holds and assume that one of the following assumptions hold.
\begin{itemlist}
\item $G$ is amenable and the action $G \actson (X,\mu)$ is not dissipative.
\item The growth condition \eqref{eq.strong-conservative-assumption} holds.
\end{itemlist}
Then $G \actson (X,\mu)$ is weakly mixing. Let $G \actson (Y,\eta)$ be any ergodic pmp action. Then, the following holds.
\begin{enumlist}
\item The action $G \actson X \times Y$ is of type II$_1$ if and only if $\mu \sim \nu^G$ for some probability measure $\nu \sim \mu_e$ on $X_0$.
\item The action $G \actson X \times Y$ is of type II$_\infty$ if and only if there exists a probability measure $\nu \sim \mu_e$ on $X_0$ and subsets $\cU_g \subset X_0$ such that
\begin{align*}
& \sum_{h \in G} H^2\bigl(\mu_g, \nu(\cU_g)^{-1} \nu|_{\cU_g}\bigr) <+\infty \quad , \quad \sum_{h \in G} \mu_g(X_0 \setminus \cU_g) < +\infty \quad , \quad \sum_{h \in G} \nu(X_0 \setminus \cU_g) = +\infty \; , \\
& \text{and for every $g \in G$, we have}\quad \sum_{h \in G} (\log(\nu(\cU_{gh})) - \log(\nu(\cU_h)) = 0 \; .
\end{align*}
\item Let $p \neq 0$. Then $2\pi / p$ belongs to the $T$-invariant of $G \actson X \times Y$ if and only if there exist probability measures $\nu_g \sim \mu_g$ and $\rho_g > 0$ such that
    $$
    \sum_{h \in G} H^2(\mu_g,\nu_g) < +\infty \quad , \quad \frac{d\nu_g}{d\nu_e}(x) \in \rho_g \exp(p\Z) \;\;\text{for all $g \in G$ and a.e.\ $x \in X_0$,}
    $$
    and with
    $$\al : G \recht R/p\Z : \al(g) = \sum_{h \in G} (\log \rho_{gh} - \log \rho_h) + p \Z \; ,$$
    the character $\exp(2\pi i p^{-1} \al)$ is an eigenvalue for $G \actson (Y,\eta)$.
\end{enumlist}
\end{theorem}

In the proof of Theorem \ref{thm.Krieger-type-Bernoulli}, we will see that the series in point~2 and point~3 automatically converge absolutely if the conditions on the preceding line hold. In Remark \ref{rem.operational-T-invariant}, we give a more hands on description of the $T$-invariant, which is useful to make computations in concrete examples later.

Note that Theorem \ref{thm.Krieger-type-Bernoulli} entirely determines the type of $G \actson (X,\mu)$ and of all diagonal products $G \actson X \times Y$ with an ergodic pmp action, once we can rule out the occurrence of type III$_0$. The following result precisely provides sufficient conditions for this.

Recall that an infinite group $G$ is said to have \emph{one end} if there are no nontrivial almost invariant subsets $W \subset G$, i.e.\ whenever $W \subset G$ is a subset such that $g W \symdiff W$ is finite for every $g \in G$, either $W$ or $G \setminus W$ is finite. By Stallings' theorem (see also \cite[Theorem IV.6.10]{DD89}), infinite groups with more than one end have a special structure: either they are locally finite, or virtually cyclic, or they have a nontrivial decomposition as an amalgamated free product or an HNN extension over a finite subgroup.

\begin{theorem}\label{thm.rule-out-III0}
Make the same assumptions as in Theorem \ref{thm.Krieger-type-Bernoulli}. Let $G \actson (Y,\eta)$ be any ergodic pmp action and consider the diagonal product $G \actson X \times Y$.
\begin{enumlist}
\item If $G$ is nonamenable, $G \actson X \times Y$ is never of type II$_\infty$ or type III$_0$.

\item If $G$ is nonamenable and has only one end, $G \actson X \times Y$ is either of type II$_1$ or of type III$_1$.

\item If the strong boundedness property \eqref{eq.boundedness-strong} holds, $G \actson X \times Y$ is never of type II$_\infty$ or type III$_0$.

\item If the family of probability measures $(\mu_g)_{g \in G}$ on $X_0$ admits two distinct limit points in Hellinger distance, then $G \actson X \times Y$ is of type III$_\lambda$ with $\lambda \in (0,1]$.
\end{enumlist}
\end{theorem}

Before proving Theorems \ref{thm.Krieger-type-Bernoulli} and \ref{thm.rule-out-III0}, we prove the ergodicity of the Bernoulli action in Proposition \ref{prop.ergodic-general} and we relate its associated flow to the associated flow of the permutation action in Lemma \ref{lem.link-to-permutation-action}. The proofs of Proposition \ref{prop.ergodic-general} and Lemma \ref{lem.link-to-permutation-action} follow quite closely \cite[Theorem 5.1 and Lemma 6.5]{BKV19}. Also, there is a certain repetition in the arguments for the proof of Proposition \ref{prop.ergodic-general} and Lemma \ref{lem.link-to-permutation-action}, but given the technicality of the latter, we believe that this repetition makes the argument more transparent. Recall from Definition \ref{def.strongly-conservative} the notion of strong conservativeness.

\begin{proposition}\label{prop.ergodic-general}
Let $G$ be a countably infinite group and $G \actson (X,\mu) = \prod_{g \in G} (X_0,\mu_g)$ a nonsingular Bernoulli action. Assume that the weak boundedness property \eqref{eq.boundedness-weak} holds.
\begin{enumlist}
\item The action $G \actson (X,\mu)$ is either conservative or dissipative.
\item If the action $G \actson (X,\mu)$ is strongly conservative, then it is weakly mixing.
\end{enumlist}
\end{proposition}
\begin{proof}
As in the proof of Proposition \ref{prop.ergodic-permutation-action}, define the probability space $(\Xtil,\mutil)$ by ``doubling the coordinate $x_e$''. So, we write
\begin{equation}\label{eq.def-Xtil}
(\Xtil,\mutil) = (X_0,\mu_e) \times (X_0,\mu_e) \times \prod_{h \neq e} (X_0,\mu_h) \; .
\end{equation}
We denote the elements of $\Xtil$ as $\xtil$ and then denote by $x_e$ and $x'_e$ their first two coordinates in $(X_0,\mu_e)$ and by $x_h$, $h \neq e$, their remaining coordinates in $(X_0,\mu_h)$. We consider the natural measure preserving factor maps
\begin{equation}\label{eq.factor-maps-pi}
\begin{split}
& \pi_0 : (\Xtil,\mutil) \recht (X,\mu) : \pi_0(\xtil) = (x_e,(x_h)_{h \neq e}) \quad\text{and}\\
& \pi_1 : (\Xtil,\mutil) \recht (X,\mu) : \pi_1(\xtil) = (x'_e,(x_h)_{h \neq e}) \; .
\end{split}
\end{equation}
Write
$$\om(g,x) = \frac{d(g^{-1} \mu)}{d\mu}(x) = \prod_{h \in G} \frac{d\mu_{gh}}{d\mu_h}(x_h) = \frac{d\mu_g}{d\mu_e}(x_e) \cdot \prod_{h \neq e} \frac{d\mu_{gh}}{d\mu_h}(x_h) \; .$$
For every $C > 0$, define $\cU_C \subset X_0$ by
\begin{equation}\label{eq.boundedness-sets}
\cU_C = \Bigl\{ x \in X_0 \Bigm| C^{-1} \leq \frac{d\mu_g}{d\mu_e}(x) \leq C \;\;\text{for all $g \in G$}\;\Bigr\} \; .
\end{equation}
Note that $\lim_{C \recht +\infty} \mu_e(X_0 \setminus \cU_C) = 0$.

1.\ Denote the dissipative part of $G \actson (X,\mu)$ by
$$\cD = \Bigl\{ x \in X \Bigm| \sum_{g \in G} \om(g,x) < +\infty \Bigr\} \; .$$
Define $\cW_C \subset \Xtil$ by $\cW_C = \{x \in \Xtil \mid x_e,x'_e \in \cU_C\}$. By construction, $\cW_C \cap \pi_0^{-1}(\cD) = \cW_C \cap \pi_1^{-1}(\cD)$, for all $C > 0$. Hence, $\pi_0^{-1}(\cD)$ equals $\pi_1^{-1}(\cD)$ up to a set of measure zero. This means that $\cD$ is essentially independent of the first coordinate $x_e$. Since $\cD$ is $G$-invariant, it follows that $\cD$ is essentially independent of each of the coordinates $x_g$, $g \in G$. So, $\cD$ is either of measure zero or measure one.

2.\ Assume that $G \actson (X,\mu)$ is strongly conservative. Fix an ergodic pmp action $G \actson (Y,\eta)$ and consider the diagonal action $G \actson X \times Y$. Let $F \in L^\infty(X \times Y)$ be a $G$-invariant function. We have to prove that $F$ is essentially constant. Define the $\|\,\cdot\,\|_1$-isometries
$$V_i : L^\infty(X \times Y) \recht L^\infty(\Xtil \times Y) : (V_i(H))(\xtil,y) = H(\pi_i(\xtil),y) \; .$$
We prove that $V_0(F) = V_1(F)$, meaning that $F$ is essentially independent of the coordinate $x_e$. Once this statement is proven, since $F$ is $G$-invariant, it follows that $F$ is essentially independent of any of the coordinates $x_g$. Since $G \actson (Y,\eta)$ is ergodic, it then follows that $F$ is constant a.e.

Fix a sequence of finitely supported probability measures $\eta_n$ on $G$ that is strongly recurrent for $G \actson (X,\mu)$. Write
$$q_n(g,x) = \sum_{h \in G} \eta_n(h) \frac{\eta_n(hg^{-1}) \, \om(g,x)}{\sum_{k \in G} \eta_n(hk^{-1}) \, \om(k,x)} \; .$$
So, the unital positive maps
$$\Phi_n : L^\infty(X \times Y) \recht L^\infty(X \times Y) : \Phi_n(H)(x,y) = \sum_{g \in G} q_n(g,x) H(g \cdot x,g \cdot y)$$
satisfy $\|\Phi_n(H)\|_1 \leq \|H\|_1$ for all $H \in L^\infty(X \times Y)$, while the strong recurrence implies that for every fixed finite subset $\cF \subset G$, we have that
$$\lim_{n \recht +\infty} \sum_{g \in \cF} \int_{X} q_n(g,x) \, d\mu(x) = 0 \; .$$
Fix $C > 0$ and define $\cU := \cU_C$ by \eqref{eq.boundedness-sets}. We then define a variant $\omtil$ of $\om$ in which all coordinates $x_e$ are removed. More precisely, we write
\begin{align*}
& \omtil(g,x) = \prod_{h \neq e} \frac{d\mu_{gh}}{d\mu_h}(x_h) \quad\text{and}\\
& p_n(g,x) = \sum_{h \in G} \eta_n(h) \frac{\eta_n(hg^{-1}) \, \omtil(g,x)}{\sum_{k \in G} \eta_n(hk^{-1}) \, \omtil(k,x)} \; .
\end{align*}
Define $\cV \subset X$ by $\cV = \{x \in X \mid x_e \in \cU\}$. By construction, $p_n(g,x)$ does not depend on the variable $x_e$ and $p_n(g,x) \leq C^2 \, q_n(g,x)$ for all $n \in \N$, $g \in G$ and $x \in \cV$. So, the unital positive maps
$$\Psi_n : L^\infty(X \times Y) \recht L^\infty(\cV \times Y) : \Psi_n(H)(x,y) = \sum_{g \in G} p_n(g,x) H(g \cdot x,g \cdot y)$$
satisfy $\|\Psi_n(H)\|_1 \leq C^2 \, \|H\|_1$ for all $H \in L^\infty(X \times Y)$. Also, we still have that for every fixed finite subset $\cF \subset G$,
\begin{equation}\label{eq.neglect-some}
\lim_{n \recht +\infty} \sum_{g \in \cF} \int_{\cV} p_n(g,x) \, d\mu(x) = 0 \; .
\end{equation}
Define $\cW \subset \Xtil$ by $\cW = \{\xtil \in \Xtil \mid x_e,x'_e \in \cU\}$. We claim that for all $H \in L^\infty(X \times Y)$,
\begin{equation}\label{eq.main-claim-ergodic}
\lim_{n \recht +\infty} \|(1_\cW \ot 1) V_0(\Psi_n(H)) - (1_\cW \ot 1) V_1(\Psi_n(H)) \|_1 = 0 \; .
\end{equation}
By the uniform boundedness property of $\Psi_n$, it suffices to prove \eqref{eq.main-claim-ergodic} when $H$ only depends on the coordinate $y \in Y$ and finitely many coordinates $x_h$, $h \in \cF$, for some finite subset $\cF \subset G$. But then, for all $n$ and all $(\xtil,y) \in \cW \times Y$,
$$\sum_{g \in G \setminus \cF} p_n(g,\pi_0(\xtil)) H(g \cdot \pi_0(\xtil),g \cdot y) = \sum_{g \in G \setminus \cF} p_n(g,\pi_1(\xtil)) H(g \cdot \pi_1(\xtil),g \cdot y) \; ,$$
while the remaining terms tend to $0$ in $\|\,\cdot\,\|_1$ by \eqref{eq.neglect-some}. So \eqref{eq.main-claim-ergodic} holds.

We apply \eqref{eq.main-claim-ergodic} to the $G$-invariant function $F$ and conclude that $(1_\cW \ot 1) V_0(F) = (1_\cW \ot 1) V_1(F)$. Taking $C > 0$ larger and larger, the measure of $\cW$ tends to $1$, so that $V_0(F) = V_1(F)$. This ends the proof of the proposition.
\end{proof}

\begin{lemma}\label{lem.link-to-permutation-action}
Let $G$ be a countably infinite group and $G \actson (X,\mu) = \prod_{g \in G} (X_0,\mu_g)$ a nonsingular Bernoulli action. Let $G \actson (Y,\eta)$ be any ergodic pmp action. Consider the product $G \actson X \times \R \times Y$ of the Maharam extension $G \actson X \times \R$ of $G \actson (X,\mu)$ and the action $G \actson Y$. Also consider the permutation action $\cS_G \actson X$ of the group of finite permutations of the countable set $G$, and its Maharam extension $\cS_G \actson X \times \R$.

Assume that the weak boundedness property \eqref{eq.boundedness-weak} holds. For every $\al > 0$, consider the probability measure $\nu_\al$ on $\R$ given by $d\nu_\al(t) = (\al/2) \exp(-\al |t|) \, dt$. Assume that $\al_0 > 0$ and that $\eta_n$ is a sequence of finitely supported probability measures on $G$ that is strongly recurrent for $G \actson (X \times \R,\mu \times \nu_\al)$ for all $0 < \al < \al_0$.

Then, $L^\infty(X \times \R \times Y)^G \subset L^\infty(X \times \R)^{\cS_G} \ovt L^\infty(Y)$.
\end{lemma}

\begin{proof}
Let $F \in L^{\infty}(X\times \R \times Y)^G$ be a $G$-invariant function generating $L^{\infty}(X\times \R\times Y)^G$ as a von Neumann algebra. We have to prove that $F \in L^\infty(X \times \R)^{\cS_G} \ovt L^\infty(Y)$.

We fix the following notations. For every $\al > 0$, write $\mu_\al = \mu \times \nu_\al$. We further write
\begin{align*}
& \om(g,x) = \frac{d(g^{-1} \mu)}{d\mu}(x) = \prod_{h \in G} \frac{d\mu_{gh}}{d\mu_h}(x_h) \; ,\\
& \gamma(g,x) = \log \om(g,x) \quad\text{so that}\quad g \cdot (x,t,y) = (g \cdot x, t+ \gamma(g,x) , g \cdot y) \; ,\\
& \om_\al(g,x,t) = \frac{d(g^{-1} \mu_\al)}{d\mu_\al}(x,t) = \om(g,x) \, F_\al(t,\gamma(g,x)) \quad\text{where}\quad F_\al(t,s) = \exp(-\al |t+s|+\al|t|) \; ,\\
& \|H\|_{1,\al} = \int_{X \times \R \times Y} |H(x,t,y)| \, d\mu_\al(x,t) \, d\eta(y) \; .
\end{align*}
We often use the following estimates, which hold for all $s,t,s',t' \in \R$~:
\begin{equation}\label{eq.favorite}
\begin{split}
& \exp(-\al|s-s'|) \leq \frac{F_\al(t,s)}{F_\al(t,s')} \leq \exp(\al|s-s'|) \quad\text{and}\\
& \exp(-2\al|t-t'|) \leq \frac{F_\al(t,s)}{F_\al(t',s)} \leq \exp(2\al|t-t'|) \; .
\end{split}
\end{equation}

As above, we denote for every $H \in L^\infty(\R)$, by $\per(H) \subset \R$ the closed subgroup of all $a \in \R$ satisfying $H(a+t) = H(t)$ for a.e.\ $t \in \R$. By Proposition \ref{prop.ergodic-general}, the action $G \actson X \times Y$ is ergodic. So by \cite[Lemma 6.6]{BKV19}, we are in one of the following possible cases.

{\bf Case 1.} The function $F$ is essentially constant.

{\bf Case 2.} For a.e.\ $(x,y) \in X \times Y$, we have that $\per F(x,\cdot,y) = \{0\}$.

{\bf Case 3.} There exists a $p \in \R$, $p \neq 0$, such that $\per F(x,\cdot,y) = p \Z$ for a.e.\ $(x,y) \in X \times Y$.

In case~1, there is nothing to prove. Assume that we are in case~2. Fix $\al_0 > 0$ and fix a sequence of finitely supported probability measures $\eta_n$ on $G$ that is strongly recurrent for $G \actson (X \times \R,\mu \times \nu_\al)$ for all $0 < \al < \al_0$. This means that writing
$$q_{\al,n}(g,x,t) = \sum_{h \in G} \eta_n(h) \frac{\eta_n(hg^{-1}) \, \om_\al(g,x,t)}{\sum_{k \in G} \eta_n(hk^{-1}) \, \om_\al(k,x,t)} \; ,$$
the unital positive maps
$$\Phi_{\al,n} : L^\infty(X \times \R \times Y) \recht L^\infty(X \times \R \times Y) : \Phi_{\al,n}(H)(x,t,y) = \sum_{g \in G} q_{\al,n}(g,x,t) H(g \cdot (x,t,y))$$
satisfy $\|\Phi_{\al,n}(H)\|_{1,\al} \leq \|H\|_{1,\al}$ for all $H \in L^\infty(X \times \R \times Y)$, while the strong recurrence implies that for every fixed finite subset $\cF \subset G$ and fixed $\al \in (0,\al_0)$, we have that
$$\lim_{n \recht +\infty} \sum_{g \in \cF} \int_{X \times \R} q_{\al,n}(g,x,t) \, d\mu_\al(x,t) = 0 \; .$$
Fix $C > 0$ and define $\cU \subset X_0$ by
$$\cU = \Bigl\{ x \in X_0 \Bigm| C^{-1} \leq \frac{d\mu_g}{d\mu_e}(x) \leq C \;\;\text{for all $g \in G$}\;\Bigr\} \; .$$
Taking $C > 0$ large enough, $\mu_e(X_0 \setminus \cU)$ can have arbitrarily small measure. We then define a variant $\omtil_\al$ of $\om_\al$ in which all coordinates $x_e$ are removed. More precisely, we write
\begin{align*}
& \omtil(g,x) = \prod_{h \neq e} \frac{d\mu_{gh}}{d\mu_h}(x_h) \quad\text{and}\quad \gammatil(g,x) = \log \omtil(g,x) \; , \\
& \omtil_\al(g,x,t) = \omtil(g,x) \, F_\al(t,\gammatil(g,x)) \; .
\end{align*}
Note that using \eqref{eq.favorite}, we have for all $g \in G$, $x \in X$ and $t,t'\in \R$,
\begin{equation}\label{eq.behav-omtil-t}
\exp(-2\al |t-t'|) \, \omtil_\al(g,x,t) \leq \omtil_\al(g,x,t') \leq \exp(2\al|t-t'|) \, \omtil_\al(g,x,t) \; .
\end{equation}

By construction, $\omtil(g,x)$ and $\omtil_\al(g,x)$ do not depend on the coordinate $x_e$ of $x \in X$. By our definition of $\cU$ and by \eqref{eq.favorite}, we find a $D > 0$ such that for all $\al \in (0,\al_0)$, all $g \in G$ and all $x \in X$ with $x_e \in \cU$,
$$D^{-1} \, \om_\al(g,x) \leq \omtil_\al(g,x) \leq D \, \om_\al(g,x) \; .$$
Define the subset $\cV \subset X$ by $\cV = \{x \in X \mid x_e \in \cU\}$. Defining
$$p_{\al,n}(g,x,t) = \sum_{h \in G} \eta_n(h) \frac{\eta_n(hg^{-1}) \, \omtil_\al(g,x,t)}{\sum_{k \in G} \eta_n(hk^{-1}) \, \omtil_\al(k,x,t)} \; ,$$
we get that
$$D^{-2} \, q_{\al,n}(g,x,t) \leq p_{\al,n}(g,x,t) \leq D^2 \, q_{\al,n}(g,x,t) \quad\text{whenever $x \in \cV$,}$$
so that the unital positive maps
$$\Psi_{\al,n} : L^\infty(X \times \R \times Y) \recht L^\infty(\cV \times \R \times Y) : \Psi_{\al,n}(H)(x,t,y) = \sum_{g \in G} p_{\al,n}(g,x,t) H(g \cdot (x,t,y))$$
satisfy $\|\Psi_{\al,n}(H)\|_{1,\al} \leq D^2 \, \|H\|_{1,\al}$ for all $H \in L^\infty(X \times \R \times Y)$. Also, for every fixed finite subset $\cF \subset G$ and every fixed $\al \in (0,\al_0)$, we have that
\begin{equation}\label{eq.discard-finitely-many}
\lim_{n \recht +\infty} \sum_{g \in \cF} \int_{\cV \times \R} p_{\al,n}(g,x,t) \, d\mu_\al(x,t) = 0 \; .
\end{equation}
Finally, it follows from \eqref{eq.behav-omtil-t} that
\begin{equation}\label{eq.est-star-star}
\exp(-4 \al |t-t'|) \, p_{\al,n}(g,x,t) \leq p_{\al,n}(g,x,t') \leq \exp(4\al |t-t'|) \, p_{\al,n}(g,x,t)
\end{equation}
for all $g \in G$, $x \in X$ and $t,t' \in \R$.

Define $(\Xtil,\mutil)$ as in \eqref{eq.def-Xtil}. We again denote the elements of $\Xtil$ as $\xtil$ and then denote by $x_e$ and $x'_e$ their first two coordinates in $(X_0,\mu_e)$ and by $x_h$, $h \neq e$, their remaining coordinates in $(X_0,\mu_h)$. We denote $\mutil_\al = \mutil \times \nu_\al$ and
$$\|H\|_{1,\al} = \int_{\Xtil \times \R \times Y} |H(\xtil,t,y)| \, d\mutil_\al(\xtil,t) \, d\eta(y) \; .$$
Define the corresponding factor maps $\pi_i : (\Xtil,\mutil) \recht (X,\mu)$ as in \eqref{eq.factor-maps-pi}. Define the corresponding isometries for the norms $\|\,\cdot\,\|_{1,\al}$~:
$$V_i : L^\infty(X \times \R \times Y) \recht L^\infty(\Xtil \times \R \times Y) : (V_i(H))(\xtil,t,y) = H(\pi_i(\xtil),t,y) \; .$$
Denote by $\cW \subset \Xtil$ the subset given as $\cW = \{\xtil \in \Xtil \mid x_e,x'_e \in \cU\}$. For $g \in G$ and $x \in X_0$, we write
$$\gamma_g(x) = \log \frac{d\mu_g}{d\mu_e}(x) \; .$$
We then define the unital positive maps
\begin{multline*}
\Theta_{\al,n} : L^\infty(X \times \R \times Y) \recht L^\infty(\cW \times \R \times Y) : \\
\Theta_{\al,n}(H)(\xtil,t,y) = \sum_{g \in G} p_{\al,n}(g,\pi_1(\xtil),t) H(g \cdot (\pi_1(\xtil),t+\gamma_g(x_e)-\gamma_g(x'_e),y)) \; .
\end{multline*}
Using \eqref{eq.est-star-star}, we find $D_1 > 0$ such that for all $\al \in (0,\al_0)$, all $n \in \N$ and all $H \in L^\infty(X \times \R \times Y)$,
$$\|\Theta_{\al,n}(H)\|_{1,\al} \leq D_1 \, \|V_1(\Psi_{\al,n}(H))\|_{1,\al} = D_1 \, \|\Psi_{\al,n}(H)\|_{1,\al} \leq D_1 \, D^2 \, \|H\|_{1,\al} \; .$$
We claim that for all fixed $\al \in (0,\al_0)$ and $H \in L^\infty(X \times \R \times Y)$,
\begin{equation}\label{eq.crucial-claim-here}
\lim_{n \recht +\infty} \|(1_{\cW} \ot 1 \ot 1) V_0(\Psi_{\al,n}(H)) - \Theta_{\al,n}(H)\|_{1,\al} = 0 \; .
\end{equation}
Because of the uniform boundedness properties of $\Psi_{\al,n}$ and $\Theta_{\al,n}$ proven above, it suffices to prove \eqref{eq.crucial-claim-here} for functions $H$ that only depend on the coordinates $t \in \R$, $y \in Y$ and finitely many coordinates $(x_h)_{h \in \cF}$, for some finite subset $\cF \subset G$. By construction,
$$p_{\al,n}(g,\pi_1(\xtil),t) = p_{\al,n}(g,\pi_0(\xtil),t) \quad\text{and}\quad \gamma(g,\pi_1(\xtil)) + \gamma_g(x_e)-\gamma_g(x'_e) = \gamma(g,\pi_0(\xtil)) \; .$$
Therefore,
$$\Theta_{\al,n}(H)(\xtil,t,y) = \sum_{g \in G} p_{\al,n}(g,\pi_0(\xtil),t) H(g \cdot \pi_1(\xtil),t+ \gamma(g,\pi_0(\xtil)), g \cdot y) \; .$$
When $g \in G \setminus \cF$, we have
$$H(g \cdot \pi_1(\xtil),t+ \gamma(g,\pi_0(\xtil)), g \cdot y) = H(g \cdot \pi_0(\xtil),t+ \gamma(g,\pi_0(\xtil)), g \cdot y) = H(g \cdot (\pi_0(\xtil),t,y)) \; .$$
By \eqref{eq.discard-finitely-many}, both in the definition of $\Psi_{\al,n}$ and in the definition of $\Theta_{\al,n}$, the finitely many terms with $g \in \cF$ tend to zero in $\|\,\cdot\,\|_{1,\al}$. So, the claim \eqref{eq.crucial-claim-here} is proven.

We apply \eqref{eq.crucial-claim-here} to the $G$-invariant function $F$. Defining the unital positive maps
\begin{multline*}
\Gamma_{\al,n} : L^\infty(X \times \R \times Y) \recht L^\infty(\cW \times \R \times Y) : \\
\Gamma_{\al,n}(H)(\xtil,t,y) = \sum_{g \in G} p_{\al,n}(g,\pi_1(\xtil),t) H(\pi_1(\xtil),t+\gamma_g(x_e)-\gamma_g(x'_e),y) \; ,
\end{multline*}
we find that for all $\al \in (0,\al_0)$,
$$\lim_{n \recht +\infty} \|(1_{\cW} \ot 1 \ot 1) V_0(F) - \Gamma_{\al,n}(F)\|_{1,\al} = 0 \; .$$
Fixing $\al_i \in (0,\al_0)$ such that $\al_i \recht 0$, we can then find a sequence $n_i \recht +\infty$ in $\N$ such that
\begin{equation}\label{eq.alm-ev}
\lim_{i \recht +\infty} \Gamma_{\al_i,n_i}(F)(\xtil,t,y) = F(\pi_0(\xtil),t,y) \quad\text{for a.e.\ $(\xtil,t,y) \in \cW \times \R \times Y$.}
\end{equation}
For fixed $\xtil \in \Xtil$ and $t \in \R$, it follows from \eqref{eq.est-star-star} that
\begin{equation}\label{eq.remove-t}
\lim_{i \recht +\infty} \sum_{g \in G} |p_{\al_i,n_i}(g,\pi_1(\xtil),t) - p_{\al_i,n_i}(g,\pi_1(\xtil),0)| = 0 \; .
\end{equation}
We then define the probability measures $\rho_i(\xtil)$ on $\R$ by
\begin{equation}\label{eq.def-rho-i}
\rho_i(\xtil) = \sum_{g \in G} p_{\al_i,n_i}(g,\pi_1(\xtil),0) \, \delta(\gamma_g(x_e)-\gamma_g(x'_e)) \; .
\end{equation}
Using \eqref{eq.remove-t} and using a convolution product notation, it follows from \eqref{eq.alm-ev} that for a.e.\ $(\xtil,y) \in \cW \times Y$, we have that
\begin{equation}\label{eq.good-alm-ev-conv}
\lim_{i \recht +\infty} \rho_i(\xtil) * F(\pi_1(\xtil), \cdot , y) = F(\pi_0(\xtil),\cdot,y) \quad\text{with a.e.\ convergence on $\R$.}
\end{equation}
Denote by $\tau : \Xtil \recht \Xtil$ the map that exchanges the coordinates $x_e$ and $x'_e$. Note that $\tau(\cW) = \cW$. Denote by $\cZ \subset \cW \times Y$ the subset of all $(\xtil,y) \in \cW \times Y$ such that \eqref{eq.good-alm-ev-conv} holds for both $(\xtil,y)$ and $(\tau(\xtil),y)$, and such that $\per(F(\pi_1(\xtil),\cdot,y))=\{0\}$. Note that $(\cW \times Y) \setminus \cZ$ has measure zero. We prove that for all $(\xtil,y) \in \cZ$, the sequence of probability measures $\rho_i(\xtil)$ converges weakly to a Dirac measure. Note that for all $i \in \N$ and $\xtil \in \cW$, the probability measure $\rho_i(\xtil)$ is supported on the interval $[-2 \log C, 2 \log C]$.

Fix $(\xtil,y) \in \cZ$. Let $\om$ be any weak limit point of the sequence $\rho_i(\xtil)$. Take a subsequence $i_j$ such that $\rho_{i_j}(\xtil) \recht \om$ weakly and such that $\rho_{i_j}(\tau(\xtil))$ converges weakly to some probability measure $\om'$. For a fixed $H \in L^\infty(\R)$, the map $\om \recht \om * H$ is weakly continuous. It thus follows from \eqref{eq.good-alm-ev-conv} that
$$\om * F(\pi_1(\xtil),\cdot,y) = F(\pi_0(\xtil),\cdot,y) \quad\text{and}\quad \om' * F(\pi_1(\tau(\xtil)),\cdot,y) = F(\pi_0(\tau(\xtil)),\cdot,y) \; .$$
The last equation is saying that $\om' * F(\pi_0(\xtil),\cdot,y) = F(\pi_1(\xtil),\cdot,y)$. It thus follows that
$$(\om' * \om) * F(\pi_1(\xtil),\cdot,y) = F(\pi_1(\xtil),\cdot,y) \; .$$
By the Choquet-Deny theorem \cite[Th\'{e}or\`{e}me 1.1]{CD60} and because $\per(F(\pi_1(\xtil),\cdot,y))=\{0\}$, the probability measure $\om' * \om$ must be the Dirac measure at $0$. Therefore, $\om$ must be a Dirac measure satisfying $\om * F(\pi_1(\xtil),\cdot,y) = F(\pi_0(\xtil),\cdot,y)$. Since $\per(F(\pi_1(\xtil),\cdot,y))=\{0\}$, there can be at most one Dirac measure with this property.

We have thus proven that for a.e.\ $\xtil \in \cW$, the probability measures $\rho_i(\xtil)$ converge to the Dirac measure in $\rho(\xtil)$, where $\rho : \cW \recht \R$ is a bounded measurable map and
$$
F(\pi_1(\xtil),t+\rho(\xtil),y) = F(\pi_0(\xtil),t,y) \quad\text{for a.e.\ $(\xtil,t,y) \in \cW \times \R \times Y$.}
$$
We next prove that the map $\rho : \cW \recht \R$ essentially only depends on the coordinates $x_e$ and $x'_e$. Fix $h \neq e$ and $E > 0$. Define
$$\cU_E = \Bigl\{ x \in X_0 \Bigm| E^{-1} \leq \frac{d\mu_{gh}}{d\mu_h}(x) \leq E \;\;\text{for all $g \in G$}\;\Bigr\} \; .$$
We then find $E_1 > 0$ such that for all $\xtil, \ztil \in \cW$ that only differ in the coordinate $h$ and satisfy $x_h,z_h \in \cU_E$, we have
$$p_{\al_i,n_i}(g,\pi_1(\ztil),0) \leq E_1 \, p_{\al_i,n_i}(g,\pi_1(\xtil),0) \quad\text{for all $g \in G$ and all $i \in \N$.}$$
Then, $\rho_i(\ztil) \leq E_1 \, \rho_i(\xtil)$ and thus $\rho(\ztil) = \rho(\xtil)$. Taking $E > 0$ larger and larger, the measure of the complement of $\cU_E$ tends to zero. We can thus conclude that the map $\xtil \mapsto \rho(\xtil)$ does not depend on the $h$'th coordinate. This holds for all $h\neq e$. We thus rewrite $\rho$ as a map from $\cU \times \cU$ to $\R$ satisfying
\begin{equation}\label{eq.looksgood}
F(\pi_1(\xtil),t+\rho(x_e,x'_e),y) = F(\pi_0(\xtil),t,y) \quad\text{for a.e.\ $(\xtil,t,y) \in \cW \times \R \times Y$.}
\end{equation}
Fix $x_0 \in \cU$ such that for a.e.\ $(x,t,y) \in \cV \times \R \times Y$, the equality in \eqref{eq.looksgood} holds for $(x_0,x,t,y)$. Define $\be : \cU \recht \R$ by $\be(x_1) = \rho(x_0,x_1)$. We have found a measurable map $\be : \cU \recht \R$ such that on $\cV \times \R \times Y$, the function $F(x,t+\be(x_e),y)$ is essentially independent of the coordinate $x_e$. By the lack of periodicity of $F(x,\cdot,y)$, such a map $\be$ is essentially unique up to a constant. Enlarging the constant $C > 0$ fixed a the start of the proof, we can define $\be$ on larger and larger subsets $\cU \subset X_0$ and ultimately find a measurable map $\be : X_0 \recht \R$ such that $F(x,t+\be(x_e),y)$ is essentially independent of the coordinate $x_e$.

In the third case, when $\per F(x,\cdot,y) = p \Z$ for a.e.\ $(x,y) \in X \times Y$, we can view $F$ as a function on $X \times \R/p\Z \times Y$, with $\per F(x,\cdot,y)$ being $\{0\}$ as a subgroup of $\R/p\Z$. The proof above can be repeated and we again find a measurable map $\be : X_0 \recht \R$ such that $F(x,t+\be(x_e),y)$ is essentially independent of the coordinate $x_e$.

Denote by $\cR \subset X \times X$ the tail equivalence relation, with $(x,z) \in \cR$ if and only if $x_g = z_g$ for all but finitely many $g \in G$. Define the map
$$\Om : \cR \recht \R : \Om(x,z) = \sum_{g \in G} \bigl( - \gamma_g(x_g) - \be(x_g) + \gamma_g(z_g) + \be(z_g) \bigr) \; .$$
For any subset $\cF \subset G$, we write $(X_\cF,\mu_\cF) = \prod_{g \in \cF}(X_0,\mu_g)$. For every finite subset $\cF \subset G$, we define $\cR_\cF \subset \cR$ by $(x,z) \in \cR_\cF$ if and only if $x_g = z_g$ for all $g \in G \setminus \cF$. We identify $\cR_\cF$ with $X_\cF \times X_\cF \times X_{G \setminus \cF}$. Since $\gamma_e = 0$, we have proven above that
\begin{equation}\label{eq.goal}
F(x,t,y) = F(z,t+\Om(x,z),y)
\end{equation}
for a.e.\ $(x,z) \in \cR_{\{e\}}$ and a.e.\ $(t,y) \in \R \times Y$. Since $F$ is $G$-invariant, it follows that for every finite subset $\cF \subset G$, \eqref{eq.goal} holds for a.e.\ $(x,z) \in \cR_\cF$ and a.e.\ $(t,y) \in \R \times Y$. We conclude that $F$ is invariant under all permutations of $\cF$. Since this holds for every finite subset $\cF \subset G$, the lemma is proven.
\end{proof}

We are now ready to prove Theorems \ref{thm.Krieger-type-Bernoulli} and \ref{thm.rule-out-III0}.

\begin{proof}[{Proof of Theorem \ref{thm.Krieger-type-Bernoulli}}]
Combining Lemma \ref{lem.strongly-conservative} and Proposition \ref{prop.ergodic-general}, it follows that under both hypotheses on $G \actson (X,\mu)$, the action is weakly mixing and Lemma \ref{lem.link-to-permutation-action} is applicable.

Fix an ergodic pmp action $G \actson (Y,\eta)$. Consider the product $G \actson X \times \R \times Y$ of the Maharam extension $G \actson X \times \R$ of $G \actson (X,\mu)$ and the action $G \actson Y$. From Lemma \ref{lem.link-to-permutation-action}, we get that $L^\infty(X \times \R \times Y)^G \subset L^\infty(X \times \R)^{\cS_G} \ovt L^\infty(Y)$.

First assume that the action $G \actson X \times Y$ is semifinite. It follows from Lemma \ref{lem.trivial-flow-result} below that also $\cS_G \actson (X,\mu)$ is semifinite. We apply Theorem \ref{thm.T-invariant-permutation-action}. Either, $\mu \sim \nu^G$ for some probability measure $\nu \sim \mu_e$ on $X_0$, or the second point of Theorem \ref{thm.T-invariant-permutation-action} applies. In the first case, $\nu^G$ is also $G$-invariant and $G \actson X \times Y$ is of type II$_1$. The converse is obvious: whenever $\mu \sim \nu^G$, the action $G \actson X \times Y$ is of type II$_1$.

In the second case, we find a probability measure $\nu \sim \mu_e$ on $X_0$ and subsets $\cU_g \subset X_0$ such that
$$\sum_{g \in G} \mu_g(X_0 \setminus \cU_g) < +\infty \;\; , \;\; \sum_{g \in G} H^2(\mu_g , \nu(\cU_g)^{-1} \nu|_{\cU_g}) < +\infty \quad\text{and}\quad \sum_{g \in G} \nu(X_0 \setminus \cU_g) = +\infty \; .$$
By Remark \ref{rem.maps-sums}, the sum
\begin{equation}\label{eq.map-F-density}
F : X \recht \R : F(x) = \sum_{h \in G} \Bigl( \log \frac{d\mu_h}{d\nu}(x_h) + \log(\nu(\cU_h)) \Bigr)
\end{equation}
is unconditionally a.e.\ convergent and the map
\begin{equation}\label{eq.my-theta}
\theta : X \times \R \recht \R : \theta(x,t) = t - F(x)
\end{equation}
implements the $\R$-equivariant isomorphism $L^\infty(X \times \R)^{\cS_G} \cong L^\infty(\R)$. Since the sum defining $F$ is unconditionally a.e.\ convergent, we have for every $g \in G$ and a.e.\ $x \in X$,
$$F(g \cdot x) = \sum_{h \in G} \Bigl( \log \frac{d\mu_{gh}}{d\nu}(x_h) + \log(\nu(\cU_{gh})) \Bigr) \; .$$
We also know that
$$\log \frac{d(g^{-1} \mu)}{d\mu}(x) = \sum_{h \in G}  \log \frac{d\mu_{gh}}{d\mu_h}(x_h)$$
with unconditional a.e.\ convergence. We conclude that
$$F(g \cdot x) - F(x) - \log \frac{d(g^{-1} \mu)}{d\mu}(x) = \sum_{h \in G} \bigl( \log(\nu(\cU_{gh})) - \log(\nu(\cU_{h})) \bigr)$$
with unconditional a.e.\ convergence. So, the right hand side is absolutely convergent and its limit $\al(g)$ defines a group homomorphism $\al : G \recht \R$. We have also proven that
$$\theta(g \cdot (x,t)) = \theta(x,t) - \al(g) \; .$$
So, defining the action $G \actson \R \times Y$ by $g \cdot (t,y) = (t-\al(g),g \cdot y)$, we have found an $\R$-equivariant identification $L^\infty(X \times \R \times Y)^G \cong L^\infty(\R \times Y)^G$. Recall that we assumed $G \actson X \times Y$ to be semifinite. So there exists an $\R$-equivariant and $G$-invariant factor map $\psi : \R \times Y \recht \R$. The $\R$-equivariance implies that $\psi(t,y) = t + H(y)$ for some measurable map $H : Y \recht \R$. Then the $G$-invariance implies that $H(g\cdot y) =\al(g) + H(y)$ for all $g \in G$ and a.e.\ $y \in Y$. So, the probability measure $H_*(\eta)$ on $\R$ is invariant under the subgroup $\al(G) \subset \R$. This forces $\al(G) = \{0\}$. We have proven that the conclusions of point~2 of the theorem hold.

Conversely, if the conclusions of point~2 hold, as in Remark \ref{rem.maps-sums}, the map $F$ in \eqref{eq.map-F-density} is well defined and $\exp(-F(x)) \, d\mu(x)$ is an infinite $G$-invariant measure on $X$. So, $G \actson X \times Y$ is of type II$_\infty$.

To prove the third statement, assume that $p \neq 0$ and that $2\pi / p$ belongs to the $T$-invariant of $G \actson X \times Y$. Because $L^\infty(X \times \R \times Y)^G \subset L^\infty(X \times \R)^{\cS_G} \ovt L^\infty(Y)$ and because of Lemma \ref{lem.trivial-flow-result}, we get that $2\pi / p$ belongs to the $T$-invariant of $\cS_G \actson X$. We apply Theorem \ref{thm.T-invariant-permutation-action} and make the same reasoning as above. We find the measures $\nu_g$ and the homomorphism $\al : G \recht \R/p\Z$ as in the formulation of the theorem. We also find $H : Y \recht \R / p \Z$ satisfying $H(g\cdot y) = \al(g) + H(y)$ for all $g \in G$ and a.e.\ $y \in Y$. This says that $\exp(2\pi i p^{-1} \al)$ is an eigenvalue for $G \actson (Y,\eta)$. Again, the converse implication holds by construction.
\end{proof}

\begin{lemma}\label{lem.trivial-flow-result}
Let $\R \actson (Z,\zeta)$ be any nonsingular ergodic action. Let $(Y,\eta)$ be any standard probability space. Consider the diagonal action $\R \actson Z \times Y : t \cdot (z,y) = (t \cdot z,y)$.
\begin{enumlist}
\item If there exists a measurable map $\pi : Z \times Y \recht \R$ that is $\R$-equivariant, i.e.\ for every $t \in \R$, we have $\pi(t \cdot z,y) = t + \pi(z,y)$ for a.e.\ $(z,y) \in Z \times Y$, then $\R \actson Z$ is conjugate to $\R \actson \R$.
\item If $s \in \R$ and if $F \in L^\infty(Z \times Y)$ is not a.e.\ zero and satisfies $F(t \cdot z,y) = \exp(its) \, F(z,y)$ for all $t \in \R$ and a.e.\ $(z,y) \in Z \times Y$, then $s$ is an eigenvalue of $\R \actson Z$.
\end{enumlist}
\end{lemma}
\begin{proof}
1.\ The measurable function $F : Z \times Y \recht Z : F(z,y) = (-\pi(z,y))\cdot z$ is $\R$-invariant. Since $\R \actson (Z,\zeta)$ is ergodic, we find a measurable map $H : Y \recht Z$ such that $F(z,y) = H(y)$ a.e., and thus $z = \pi(z,y) \cdot H(y)$ for a.e.\ $(z,y) \in Z \times Y$. Pick $y \in Y$ such that $z = \pi(z,y) \cdot H(y)$ for a.e.\ $z \in Z$. It follows that the measure $\zeta$ is supported on a single orbit of the action $\R \actson Z$. So, the action $\R \actson (Z,\zeta)$ is either periodic, or conjugate to $\R \actson \R$. If the action $\R \actson (Z,\zeta)$ is periodic, also $\R \actson Z \times Y$ is periodic, which is incompatible with the existence of the $\R$-equivariant $\pi : Z \times Y \recht \R$. So, $\R \actson Z$ is conjugate to $\R \actson \R$.

2.\ Take $K \in L^\infty(Y)$ such that the function $H(z) = \int_Y F(z,y) \, K(y) \, d\eta(y)$ is not a.e.\ zero. Then $H$ is an eigenvector with eigenvalue $s$ for the ergodic action $\R \actson (Z,\zeta)$.
\end{proof}

\begin{proof}[{Proof of Theorem \ref{thm.rule-out-III0}}]
We start by proving that if $G$ is nonamenable, then either $\mu \sim \nu^G$ for some probability measure $\nu \sim \mu_e$, or the associated flow of $\cS_G \actson (X,\mu)$ is periodic. So assume that $G$ is nonamenable and that the associated flow of $\cS_G \actson (X,\mu)$ is not periodic. We prove that there exists a probability measure $\nu \sim \mu_e$ such that $\mu \sim \nu^G$.

By Theorem \ref{thm.permutation-condition-periodic}, we find a probability measure $\nu \sim \mu_e$ and a bounded family $t_g \in \R$ such that $\log(d\mu_g / d\nu) - t_g$ converges to zero in measure as $g \recht \infty$. Define the $1$-cocycle $c_g \in \ell^2(G) \ot L^2(X_0,\nu)$ by
$$c_g(h,x) = \sqrt{\frac{d\mu_{g^{-1}h}}{d\nu}}(x) - \sqrt{\frac{d\mu_{h}}{d\nu}}(x) \; .$$
We also consider the $1$-cocycle $\ctil_g \in \ell^2(G) \ot L^2(X_0 \times X_0,\nu \times \nu)$ given by
$$\ctil_g(h,x,y) = c_g(h,x) - c_g(h,y) \; .$$
For every $C > 0$, define
$$\cU_C = \bigl\{x \in X_0 \bigm| |\log(d\mu_g/d\nu)| \leq C \;\;\text{for all $g \in G$}\bigr\} \; .$$
By \eqref{eq.boundedness-weak}, the set $\bigcup_{C > 0} \cU_C$ is conull. Denote $p_C = 1_{\cU_C}$ and $\rho_g = \exp(t_g/2)$. Since $\log(d\mu_g/d\nu) - t_g \recht 0$ in measure, we find that for every $C > 0$,
$$\lim_{h \recht \infty} \Bigl\| \sqrt{\frac{d\mu_h}{d\nu}} \, p_C - \rho_h \, p_C \Bigr\|_{2,\nu} = 0 \; .$$
Defining
$$F_C : G \recht L^2(\cU_C \times \cU_C,\nu \times \nu) : (F_C(h))(x,y) = \sqrt{\frac{d\mu_h}{d\nu}}(x) - \sqrt{\frac{d\mu_h}{d\nu}}(y) \; ,$$
it follows that for every $C > 0$, $\lim_{g \recht \infty} \|F_C\|_{2,\nu \times \nu} = 0$. Note that
$$(1 \ot p_C \ot p_C)\ctil_g(h) = F_C(g^{-1}h) - F_C(h) \; ,$$
meaning that $F_C$ implements an $\ell^2$-cocycle. In \cite[Appendix A]{BK18}, it is proven that a function $H : G \recht \C$ that tends to zero at infinity on a finitely generated nonamenable group and that implements an $\ell^2$-cocycle, must belong to $\ell^2(G)$. For an arbitrary nonamenable group, a $1$-cocycle with values in a multiple of the regular representation that is an approximate coboundary on each finitely generated subgroup, must be a coboundary. We conclude that $(1 \ot p_C \ot p_C)\ctil_g$ is a coboundary for every $C > 0$, and hence that also $\ctil_g$ is a coboundary.

Defining
$$F : G \recht L^2(X_0 \times X_0,\nu \times \nu) : (F(h))(x,y) = \sqrt{\frac{d\mu_h}{d\nu}}(x) - \sqrt{\frac{d\mu_h}{d\nu}}(y) \; ,$$
we have that $\ctil_g(h) = F(g^{-1}h) - F(h)$. Since $\ctil_g$ is a coboundary, we thus find $\xi \in L^2(X_0 \times X_0,\nu \times \nu)$ such that $F - \xi \in \ell^2(G) \ot L^2(X_0 \times X_0,\nu \times \nu)$. So, for every $C > 0$, we get that $F_C - (p_C \ot p_C)\xi \in \ell^2(G) \ot L^2(X_0 \times X_0,\nu \times \nu)$. Since $\|F_C(h)\|_{2,\nu\times \nu} \recht 0$ when $h \recht \infty$, we conclude that $(p_C \ot p_C)\xi = 0$ for all $C > 0$. Hence, $\xi = 0$ and $F \in \ell^2(G) \ot L^2(X_0 \times X_0,\nu \times \nu)$. Since
$$\|F(h)\|_{2,\nu \times \nu}^2 = 2\Bigl(1 - \Bigl(\int_{X_0} \sqrt{\frac{d\mu_h}{d\nu}} \, d\nu\Bigr)^2 \Bigr) \geq 2 H^2(\mu_h,\nu) \; ,$$
we find that
$$\sum_{h \in G} H^2(\mu_h,\nu) < +\infty \; .$$
So, $\mu \sim \nu^G$.

Fix an ergodic pmp action $G \actson (Y,\eta)$. Consider the product $G \actson X \times \R \times Y$ of the Maharam extension $G \actson X \times \R$ of $G \actson (X,\mu)$ and the action $G \actson Y$. As explained at the start of the proof of Theorem \ref{thm.Krieger-type-Bernoulli}, we get $L^\infty(X \times \R \times Y)^G \subset L^\infty(X \times \R)^{\cS_G} \ovt L^\infty(Y)$.

1.\ By the statement proven above, either $\mu \sim \nu^G$, in which case $G \actson X \times Y$ is of type II$_1$, or $\R \actson L^\infty(X \times \R)^{\cS_G}$ is periodic, in which case also the associated flow of $G \actson X \times Y$ is periodic, so that $G \actson X \times Y$ is of type III$_\lambda$ for some $\lambda \in (0,1]$.

2.\ Assume that $G$ is nonamenable and has only one end. Assume that the $T$-invariant of $\cS_G \actson (X,\mu)$ is different from $\{0\}$. Theorem \ref{thm.T-invariant-permutation-action} provides $p \neq 0$, a probability measure $\nu \sim \mu_e$ on $X_0$ and a family $t_g \in [0,p)$ such that
$$\sum_{g \in G} \int_{X_0} d\Bigl( \log \frac{d\mu_g}{d\nu}(x) - t_g , p \Z \Bigr)^2 \, d\mu_g(x) < +\infty \; .$$
We uniquely define $\gamma_g : X_0 \recht \Z$ such that
\begin{equation}\label{eq.bound-1}
-\frac{p}{2} < \log \frac{d\mu_g}{d\nu}(x) - t_g - p \gamma_g(x) \leq \frac{p}{2}
\end{equation}
for all $g \in G$, $x \in X_0$, and conclude that
$$\sum_{g \in G} \int_{X_0} \Bigl( \log \frac{d\mu_g}{d\nu}(x) - t_g -p \gamma_g(x) \Bigr)^2 \, d\mu_g(x) < +\infty \; .$$
For every $C > 0$, define
$$\cU_C = \Bigl\{x \in X_0 \Bigm| \bigl| \log \frac{d\mu_g}{d\nu}(x) \bigr| \leq C \;\;\text{and}\;\; |\gamma_g(x)| \leq C \;\;\text{for all $g \in G$}\;\Bigr\} \; .$$
By \eqref{eq.bound-1} and \eqref{eq.boundedness-weak}, it follows that $\bigcup_{C > 0} \cU_C$ is conull. Since $d\nu / d\mu_g$ stays bounded on $\cU_C$, we get that for all $C > 0$,
\begin{equation}\label{eq.finite-1}
\sum_{h \in G} \int_{\cU_C} \Bigl( \log \frac{d\mu_h}{d\nu}(x) - t_h -p \gamma_h(x) \Bigr)^2 \, d\nu(x) < +\infty \; .
\end{equation}
Since for every $g \in G$,
$$\sum_{h \in G} \int_{X_0} \Bigl( \sqrt{\frac{d\mu_{g^{-1}h}}{d\nu}}(x) - \sqrt{\frac{d\mu_h}{d\nu}}(x) \Bigr)^2 \, d\nu(x) = 2 \, \sum_{h \in G} H^2(\mu_{g^{-1}h},\mu_h) < +\infty$$
and since $\log$ is Lipschitz on sets of the form $[D^{-1},D]$, we find that for every $g \in G$ and $C > 0$,
\begin{equation}\label{eq.finite-2}
\sum_{h \in G} \int_{\cU_C} \Bigl( \log \frac{d\mu_{g^{-1}h}}{d\nu}(x) - \log \frac{d\mu_{h}}{d\nu}(x) \Bigr)^2 \, d\nu(x) < +\infty \; .
\end{equation}
For every $C > 0$, define the function
$$F_C : G \recht L^2(\cU_C,\nu) : (F_C(h))(x) = t_h + p \gamma_h(x) \; .$$
Combining \eqref{eq.finite-1} and \eqref{eq.finite-2}, it follows that $F_C$ implements an $L^2$-cocycle, meaning that
$$\sum_{h \in G} \|F_C(g^{-1}h) - F_C(h)\|_{2,\nu}^2 < +\infty \quad\text{for all $g \in G$.}$$
So, also the function
$$H_C : G \recht L^2(\cU_C \times \cU_C,\nu \times \nu) : (H_C(h))(x,y) = p (\gamma_h(x) - \gamma_h(y))$$
implements an $L^2$-cocycle. In particular, for a.e.\ $(x,y) \in \cU_C \times \cU_C$, the map $h \mapsto (H_C(h))(x,y)$ implements a cocycle with values in $\ell^2(G)$. By construction, this function takes only finitely many values. Since $G$ has only one end, it follows that $h \mapsto (H_C(h))(x,y)$ is constant outside a finite set. We conclude that for all $C > 0$ and a.e.\ $(x,y) \in \cU_C \times \cU_C$, the limit
$$\lim_{h \recht \infty} (\gamma_h(x) - \gamma_h(y))$$
exists. So, this limit exists for a.e.\ $(x,y) \in X_0 \times X_0$. Pick $C > 0$ and $y \in \cU_C$ such that, writing $s_h = \gamma_h(y)$, the limit $\lim_{h \recht \infty} (\gamma_h(x) - s_h)$ exists for a.e.\ $x \in X_0$. Note that $(s_h)$ is a bounded family. From \eqref{eq.finite-1}, we know that for a.e.\ $x \in X_0$,
$$\lim_{h \recht \infty} \Bigl(\log \frac{d\mu_h}{d\nu}(x) - t_h -p \gamma_h(x)\Bigr) = 0 \; .$$
Therefore, for a.e.\ $x \in X_0$, the limit
$$\be(x) = \lim_{h \recht \infty} \Bigl(\log \frac{d\mu_h}{d\nu}(x) - t_h - p s_h\Bigr)$$
exists. We again deduce that $\exp(\be)$ is $\nu$-integrable. We thus find a probability measure $\nu_1 \sim \mu_e$ on $X_0$ and a bounded family $r_h \in \R$ such that
$$\lim_{h \recht \infty} \Bigl(\log \frac{d\mu_h}{d\nu_1}(x) - r_h\Bigr) = 0$$
for a.e.\ $x \in X_0$. So, we are again in the same situation as in the beginning of the proof of the theorem. Since $G$ is nonamenable, we conclude that $\mu \sim \nu_1^G$.

We have thus proven that in the case where $G$ is nonamenable with one end, either $\mu \sim \nu^G$ for some probability measure $\nu \sim \mu_e$ on $X_0$, or $L^\infty(X \times \R)^{\cS_G} = \C 1$. In the latter case, since $L^\infty(X \times \R \times Y)^G \subset L^\infty(X \times \R)^{\cS_G} \ovt L^\infty(Y)$ and since $G \actson Y$ is ergodic, it follows that $G \actson X \times Y$ is of type III$_1$.

3.\ By Theorems \ref{thm.permutation-condition-periodic} and \ref{thm.T-invariant-permutation-action}, either $\mu \sim \nu^G$ or the associated flow of $\cS_G \actson (X,\mu)$ is periodic. In the first case, $G \actson X \times Y$ is of type II$_1$ and in the second case, also the associated flow of $G \actson X \times Y$ is periodic. We have thus ruled out that $G \actson X \times Y$ is of type II$_\infty$ or type III$_0$.

4.\ By Theorem \ref{thm.permutation-condition-periodic}, the associated flow of $\cS_G \actson (X,\mu)$ is periodic. Hence, also $G \actson X \times Y$ has a periodic associated flow.
\end{proof}

\begin{remark}\label{rem.operational-T-invariant}
Assume that we are in the context of Theorem \ref{thm.Krieger-type-Bernoulli}. The following provides a more practical approach to check whether for some $p \neq 0$, $2\pi/p$ belongs to the $T$-invariant of $G \actson X \times Y$. Write $\phi_g = \log(d\mu_g / d\mu_e)$. Since the proof of Theorem \ref{thm.Krieger-type-Bernoulli} invokes Theorem \ref{thm.T-invariant-permutation-action}, note that \eqref{eq.condition-T-inv-perm} provides as a first necessary condition that for a.e.\ $(x,y) \in X_0 \times X_0$, we have that $\phi_g(x)-\phi_g(y) + p \Z$ converges in $\R/p\Z$ to a limit function $\Om(x,y)$ satisfying
\begin{equation}\label{eq.good-condition}
\sum_{g \in G} \int_{X_0 \times X_0} d(\phi_g(x)-\phi_g(y) - \Om(x,y),p\Z)^2 \, d\mu_g(x) \, d\mu_g(y) < +\infty \; .
\end{equation}
This limit function is automatically of the form $\Om(x,y) = \be(x) - \be(y)$ for some measurable $\be : X_0 \recht \R / p\Z$. Lifting $\be$ to a bounded function $\be : X_0 \recht \R$ and adding to $\be$ a constant value if needed, we may assume that $\int_{X_0} \exp(\be(x)) \, d\mu_e(x) = 1$.

Define the probability measure $\nu \sim \mu_e$ by $d\nu / d\mu_e = \exp(\be)$. Write $\psi_g = \log(d\mu_g / d\nu)$. Uniquely define the measurable functions $\gamma_g : X_0 \recht \Z$ such that
$$-p/2 < \psi_g(x) - p \gamma_g(x) \leq p/2 \quad\text{for all $g \in G$, $x \in X_0$.}$$
Since the ratio between $d\mu_g / d\nu$ and $\exp(p \gamma_g)$ is bounded by construction, we can uniquely define probability measures $\nu_g \sim \nu$ and $\rho_g > 0$ such that $d\nu_g/d\nu(x) = \rho_g \exp(p \gamma_g(x))$. Since \eqref{eq.good-condition} holds, we have that
$$\sum_{g \in G} H^2(\mu_g , \nu_g) < +\infty \; .$$
As we saw in the proof of Theorem \ref{thm.Krieger-type-Bernoulli}, the series $\al(g) = \sum_{h \in G} (\log \rho_{gh} - \log \rho_h)+p\Z$ is absolutely convergent in $\R/p\Z$ for all $g \in G$ and defines a group homomorphism $\al : G \recht \R / p \Z$.

Define the subgroup of $\Z$ consisting of all $k \in \Z$ such that $\exp(2\pi i (k/p) \al)$ is an eigenvalue for $G \actson (Y,\eta)$. If this subgroup equals $\{0\}$, the action $G \actson X \times Y$ is of type III$_1$. If this subgroup equals $n \Z$ with $n \in \N$, the action $G \actson X \times Y$ is of type III$_\lambda$ with $\lambda = \exp(-|p|/n)$.

In particular, the action $G \actson X$ is of type III$_1$ if $\al(G) \subset \R / p \Z$ is a dense subgroup, and it is of type III$_\lambda$ with $\lambda = \exp(-|p|/n)$ if $\al(G) = ((p/n)\Z)/p\Z$.
\end{remark}

\section{Nonsingular Bernoulli actions of amenable groups}

The main goal of this section is to prove Theorem \ref{thmA.amenable-all-types}, which will follow immediately from Corollary \ref{cor.amenable-type-III-lambda} and Theorems \ref{thm.amenable-type-III-0} and \ref{thm.amenable-type-II-infty} below.

Before moving to the more delicate case of types II$_\infty$ and III$_0$, we provide a general construction for Bernoulli actions of type III$_\lambda$ with $\lambda \in (0,1)$. In \cite[Proposition 6.8]{BKV19}, the following examples of Bernoulli actions were associated to an almost invariant subset $W \subset G$
$$G \actson \prod_{g \in G} (\{0,1\},\mu_g) \quad\text{where}\;\; \mu_g(0) = \frac{\lambda}{1+\lambda}\;\;\text{if $g \in W$, and}\;\; \mu_g(0) = \frac{1}{1+\lambda}\;\; \text{if $g \not\in W$.}$$
For nonamenable groups with more than one end, as well as for infinite locally finite groups, this gave examples of Bernoulli actions of type III$_\lambda$ with $0 < \lambda < 1$. We now allow for larger base spaces $X_0$ and show that type III$_\lambda$ with $0 < \lambda < 1$ occurs for all amenable groups, in particular for the group of integers. The same result was obtained in the very recent paper \cite{KS20}. Note however that \cite{KS20} is mainly constructing examples, while our Theorems \ref{thm.Krieger-type-Bernoulli} and \ref{thm.rule-out-III0} above, complemented with the examples in Proposition \ref{prop.general-type-III-lambda-construction} below, provide an almost complete picture when an arbitrary nonsingular Bernoulli action is of type III$_\lambda$.

For the following proposition, fix a countably infinite group $G$ and a standard probability space $(X_0,\mu_0)$. Consider the action $G \actson G \times X_0$ given by $g \cdot (h,x) = (gh,x)$. Equip $G \times X_0$ with the measure $\zeta$ given as the product of the counting measure on $G$ and $\mu_0$ on $X_0$. When $A_1 , A_2 \subset G \times X_0$ are measurable subsets, we write $A_1 \sim A_2$ if the symmetric difference $A_1 \symdiff A_2$ has finite measure. When $A \subset G \times X_0$, we denote by $A_g \subset X_0$ the sections consisting of all $x \in X_0$ such that $(g,x) \in A$.

\begin{proposition}\label{prop.general-type-III-lambda-construction}
Let $A \subset G \times X_0$ be a measurable subset. Fix $\lambda \in (0,1)$ and define the probability measures $\mu_g$ on $X_0$ by
$$\frac{d\mu_g}{d\mu_0}(x) = \begin{cases} \rho_g \, \lambda &\;\;\text{if $x \in A_g$,}\\ \rho_g &\;\;\text{if $x \not\in A_g$,}\end{cases} \quad\text{where}\quad \rho_g = (\lambda \mu_0(A_g) + \mu_0(X_0 \setminus A_g))^{-1} \; .$$
Make the following assumptions.
\begin{enumlist}
\item For every $g \in G$, we have $g \cdot A \sim A$.
\item We have that
$$\limsup_{s \recht +\infty} \frac{\log |\{g \in G \mid \zeta(g \cdot A \symdiff A) \leq s\}|}{s} > 3 \lambda^{-2} \; .$$
\item Whenever $B \subset G$ and $\cU \subset X_0$ is measurable, we have $A \not\sim (B \times \cU) \cup (B^c \times \cU^c)$.
\end{enumlist}
Then the Bernoulli action $G \actson (X,\mu) = \prod_{g \in G} (X_0,\mu_g)$ is nonsingular, weakly mixing and of type~III$_\gamma$ with $\gamma \in \{1\} \cup \{\lambda^{1/n} \mid n \in \N\}$.

Write $p = \log \lambda$. Then, $\al(g) = \sum_{h \in G} (\log \rho_{gh} - \log \rho_h) + p \Z$ is absolutely convergent in $\R/p\Z$ and defines a group homomorphism $\al : G \recht \R/p\Z$. If $\al(G) = ((p/n)\Z) / p\Z$ for some $n \in \N$, then $G \actson (X,\mu)$ is of type III$_\gamma$ with $\gamma = \lambda^{1/n}$. If $\al(G) \subset \R/p\Z$ is dense, $G \actson (X,\mu)$ is of type III$_1$.

Under the additional hypotheses that a.e.\ section $A^x \subset G$ is finite and that $\sum_{g \in G} \mu_0(A_g)^2 < +\infty$, we have $\al(G) = \{0\}$ and the action $G \actson (X,\mu)$ is of stable type III$_\lambda$.
\end{proposition}
\begin{proof}
For all $g,h \in G$, since $\rho_g \leq \lambda^{-1}$, we have that
\begin{align*}
\int_{X_0} \frac{d\mu_h}{d\mu_{gh}} \, d\mu_h &= \frac{\rho_h}{\rho_{gh}} \bigl( \lambda \, \mu_h(A_h \setminus A_{gh}) + \mu_h(A_h \cap A_{gh}) + \mu_h(A_h^c \cap A_{gh}^c) + \lambda^{-1} \, \mu_h(A_{gh} \setminus A_h)\bigr) \\
& \leq \frac{\rho_h}{\rho_{gh}} \bigl( 1 + (\lambda^{-1} - 1) \, \mu_h(A_{gh} \setminus A_h)\bigr) \leq \frac{\rho_h}{\rho_{gh}} \bigl( 1 + (\lambda^{-1} - 1)\lambda^{-1} \, \mu_0(A_{gh} \setminus A_h)\bigr) \; .
\end{align*}
Since $\log(1+t) \leq t$ for all $t \geq 0$ and $|\log a - \log b| \leq \lambda^{-1} |a-b|$ for all $a,b \geq \lambda$, we conclude that
\begin{align*}
\log \int_{X_0} \frac{d\mu_h}{d\mu_{gh}} \, d\mu_h & \leq \lambda^{-1} \, |\mu_0(A_h) - \mu_0(A_{gh})| + (\lambda^{-1} - 1)\lambda^{-1} \, \mu_0(A_{gh} \setminus A_h) \\
& \leq \lambda^{-2} \, \mu_0(A_{gh} \symdiff A_h) \; .
\end{align*}
Using the notation in \eqref{eq.notation-C}, we find that
$$C(g^{\pm 1}) \leq \frac{\lambda^{-2}}{2} \, \zeta(g \cdot A \symdiff A) < +\infty \; .$$
By assumption~1, the Bernoulli action $G \actson (X,\mu)$ is nonsingular. By assumption~2 and Lemma \ref{lem.strongly-conservative}, the action is strongly conservative. By Proposition \ref{prop.ergodic-general}, the action is weakly mixing. Also, Theorems \ref{thm.Krieger-type-Bernoulli} and \ref{thm.rule-out-III0} apply. By point~3 of Theorem \ref{thm.rule-out-III0}, we know that $G \actson (X,\mu)$ is either of type II$_1$ or of type III$_\mu$ for some $\mu \in (0,1]$. The precise type is then determined by the $T$-invariant and we use Remark \ref{rem.operational-T-invariant} to determine this $T$-invariant.

Let $k \in \R \setminus \{0\}$ such that $2\pi / (k \log \lambda)$ belongs to the $T$-invariant of $G \actson (X,\mu)$. We prove that $1 \in k \Z$. Assume the contrary. Denote $p_g = 1_{A_g}$. By Remark \ref{rem.operational-T-invariant}, for a.e.\ $(x,y) \in X_0 \times X_0$, the family $p_g(x) - p_g(y)$ converges in $\R / k \Z$ as $g \recht \infty$.
The only possible limit function is of the form $(x,y) \mapsto p(x) - p(y)$ where $p = 1_\cU$ for a Borel set $\cU \subset X_0$. By \eqref{eq.good-condition} and the uniform boundedness of $d\mu_0/d\mu_g$, we have that
\begin{equation}\label{eq.aa}
\sum_{g \in G} (\mu_0 \times \mu_0)\bigl(\{(x,y) \in X_0 \times X_0 \mid p_g(x) - p(x) \neq p_g(y) - p(y) \}\bigr) < +\infty \; .
\end{equation}
We restrict the subsets in \eqref{eq.aa} to $(x,y) \in \cU \times \cU$ and $(x,y) \in \cU^c \times \cU^c$. Since $1 \neq 0$ in $\R/k\Z$, we find that
$$\sum_{g \in G} \bigl(\min\{\mu_0(\cU \cap A_g),\mu_0(\cU \setminus A_g)\} + \min \{\mu_0(\cU^c \cap A_g),\mu_0(\cU^c \setminus A_g)\}\bigr) < + \infty \; .$$
It follows that $A \sim (B \times \cU) \cup (B^c \times \cU^c)$ for some subset $B \subset G$, contrary to assumption~3 of the proposition. So, we have proven that $1 \in k \Z$.

Taking $k = 1$, as explained in Remark \ref{rem.operational-T-invariant}, the series defining $\al$ is absolutely convergent and determines the type as stated in the proposition.

Finally assume that the additional hypotheses hold. We prove that $\al(G) = \{0\}$. The additional hypotheses imply that
$$\log \rho_h = (1-\lambda) \mu_0(A_h) + \eps_h \quad\text{where $\eps_h \geq 0$ satisfies}\;\;\sum_{h \in G} \eps_h < +\infty \; .$$
Then, $\sum_{h \in G} (\eps_{gh} - \eps_h) = 0$ for all $g \in G$, while
$$\sum_{h \in G} (\mu_0(A_{gh}) - \mu_0(A_h)) = \sum_{h \in G} \bigl(\mu_0(A_{gh} \setminus A_h) - \mu_0(A_h \setminus A_{gh})\bigr) \; .$$
But,
$$\sum_{h \in G} \mu_0(A_{gh} \setminus A_h) = \zeta(g^{-1} \cdot A \setminus A) \quad\text{and}\quad \sum_{h \in G} \mu_0(A_h \setminus A_{gh}) = \zeta(A \setminus g^{-1} \cdot A) \; .$$
Both values are finite. To conclude the proof of the proposition, it thus suffices to prove that $\zeta(g \cdot A \setminus A) = \zeta(A \setminus g \cdot A)$ for all $g \in G$. Note that
$$\zeta(g \cdot A \setminus A) = \int_{X_0} |g A^x \setminus A^x| \, d\mu_0(x) \; .$$
Since for a.e.\ $x \in X_0$, the set $A^x$ is finite, we have
$$|g A^x \setminus A^x| = |g A^x \cup A^x| - |A^x| = |gA^x \cup A^x| - |g A^x| = |A^x \setminus g A^x| \; .$$
It follows that $\zeta(g \cdot A \setminus A) = \zeta(A \setminus g \cdot A)$.
\end{proof}

\begin{corollary}\label{cor.amenable-type-III-lambda}
Let $G$ be an infinite amenable group and $\lambda \in (0,1)$. Then $G$ admits a nonsingular Bernoulli action that is weakly mixing and of stable type III$_\lambda$, with countably infinite base space $X_0$.
\end{corollary}
\begin{proof}
Fix an increasing sequence of finite subsets $\cG_n \subset G$ such that $(\log |\cG_n|)/n \recht +\infty$ and $\bigcup_n \cG_n = G$. Choose finite F{\o}lner sets $\cF_n \subset G$ such that
$$\frac{|g \cF_n \symdiff \cF_n|}{|\cF_n|} \leq 2^{-n} \quad\text{for all $g \in \cG_n$}$$
and make this choice such that $|\cF_1| \geq 2$, $\cF_{n-1} \subset \cF_n$ and $|\cF_n| \geq 2 \, |\cF_{n-1}|$ for all $n \geq 2$. Take $X_0 = \N$ and define the probability measure $\mu_0$ on $X_0$ by
$$\mu_0(n) = \rho^{-1} \, |\cF_n|^{-1} \quad\text{where}\quad \rho = \sum_{n=1}^\infty |\cF_n|^{-1} \; ,$$
which is well defined because $|\cF_n| \geq 2^n$. Define the subset $A \subset G \times X_0$ by $(g,n) \in A$ if and only if $g \in \cF_n$.

Fix $\lambda \in (0,1)$. We prove that all hypotheses of Proposition \ref{prop.general-type-III-lambda-construction} are satisfied. For every $g \in \cG_n$, we have
\begin{align*}
\zeta(g \cdot A \symdiff A) &= \sum_{k=1}^\infty |g \cF_k \symdiff \cF_k| \, \mu_0(k) \leq \rho^{-1} \sum_{k=1}^\infty \frac{|g \cF_k \symdiff \cF_k|}{|\cF_k|} \\
&\leq \rho^{-1} \sum_{k=1}^{n-1} 2 + \rho^{-1} \sum_{k=n}^\infty 2^{-k} \leq 2\rho^{-1} n \; .
\end{align*}
So assumption~1 of Proposition \ref{prop.general-type-III-lambda-construction} holds. Writing $s_n = 2\rho^{-1} n$, since $(\log |\cG_n|)/n \recht +\infty$, also assumption~2 holds. By construction, $\zeta(A) =+\infty$. Whenever $\cU \subset X_0$ and $B \subset G$, the symmetric difference of $A \cap (G \times \cU)$ and $B \times \cU$ has infinite measure, unless both $B$ and $\cU$ are finite. Therefore, assumption~3 holds.

By construction, the sections $A^n \subset G$ are given by the finite subsets $\cF_n \subset G$. When $n \in \N$ and $h \in \cF_n \setminus \cF_{n-1}$, we have $A_h = [n,+\infty)$, so that
$$\mu_0(A_h) \leq \rho^{-1} \sum_{k=n}^\infty |\cF_k|^{-1} \leq 2 \rho^{-1} \, |\cF_n|^{-1}$$
because $|\cF_m| \geq 2 |\cF_{m-1}|$ for all $m \geq 2$. It follows that
$$\sum_{h \in \cF_n \setminus \cF_{n-1}} \mu_0(A_h)^2 \leq 4 \rho^{-2} \frac{|\cF_n \setminus \cF_{n-1}|}{|\cF_n|^2} \leq 4 \rho^{-2} |\cF_n|^{-1} \; .$$
We conclude that $\sum_{h \in G} \mu_0(A_h)^2 <+\infty$, so that the corollary is proven.
\end{proof}

\begin{example}\label{ex.integers-type-III-lambda}
Let $\lambda \in (0,1)$. For every $k \in \Z$, define the probability measure $\mu_k$ on $\N$ by
$$\mu_k(n) = \|\xi_k\|_1^{-1} \, \xi_k(n) \quad\text{where}\quad \xi_k(n) = \begin{cases} 2^{-n^2} &\;\;\text{if $2^{n^2} < |k|$,}\\ \lambda \, 2^{-n^2} &\;\;\text{if $2^{n^2} \geq |k|$.}\end{cases}$$
Then, $\Z \actson \prod_{k \in \Z} (\N,\mu_k)$ is a nonsingular Bernoulli action that is weakly mixing and of stable type III$_\lambda$.
\end{example}

To streamline the notations in the proof of Theorems \ref{thm.amenable-type-III-0} and \ref{thm.amenable-type-II-infty}, we formulate the following elementary lemma.

\begin{lemma}\label{lem.function-N-Folner}
Let $G$ be a countable group and $\cG_n \subset G$ an increasing sequence of subsets satisfying $\cG_n = \cG_n^{-1}$. Let $\cF_n \subset G$ be another increasing sequence of subsets satisfying $\cF_0 = \{e\}$ and $\cG_n \cF_{n-1} \subset \cF_n$ for all $n \geq 1$. Assume that $\bigcup_{n=0}^\infty \cF_n = G$ and define the function
\begin{equation}\label{eq.def-function-N}
N : G \recht \N \cup \{0\} : \begin{cases} N(g) = 0 &\;\; \text{if $g = e$,} \\ N(g) = n &\;\;\text{if $n \geq 1$ and $g \in \cF_n \setminus \cF_{n-1}$.}\end{cases}
\end{equation}
For every $n \geq 1$ and $g \in \cG_n$, we have that the set $\{h \in G \mid N(gh) \neq N(h)\}$ is contained in
$$\cF_{n-1} \cup g^{-1}\cF_{n-1} \cup \bigcup_{k=n}^\infty (\cF_k \setminus g^{-1} \cF_k) \cup \bigcup_{k=n}^\infty (g^{-1} \cF_k \setminus \cF_k) \; .$$
Moreover, for every $k \geq n \geq 1$ and $g \in \cG_n$, we have that
\begin{align*}
& \cF_k \setminus g^{-1} \cF_k = \{h \in G \mid N(h) = k \;\; , \;\; N(gh) = k+1\} \quad\text{and}\\
& g^{-1}\cF_k \setminus \cF_k = \{h \in G \mid N(h) = k+1 \;\;,\;\; N(gh) = k\} \; .
\end{align*}
\end{lemma}
\begin{proof}
The proof is straightforward, repeatedly using that $g \cF_k \subset \cF_{k+1}$ and $g^{-1} \cF_k \subset \cF_{k+1}$ for all $k \geq n-1$ and $g \in \cG_n = \cG_n^{-1}$.
\end{proof}

We are now ready to construct nonsingular Bernoulli actions of type III$_0$ and type II$_\infty$.

\begin{theorem}\label{thm.amenable-type-III-0}
Every infinite amenable group admits a nonsingular Bernoulli action that is weakly mixing and of stable type III$_0$.
\end{theorem}

\begin{proof}
For every $n \in \N$, write $\lambda_n = 2^{2^n}$. Note that $\lambda_n^2 = \lambda_{n+1}$. Fix an increasing sequence of finite subsets $\cH_n \subset G$ with $\bigcup_{n=0}^\infty \cH_n = G$ and $\cH_0 = \{e\}$. Put $\cF_0 = \{e\}$ and $\cG_0 = \emptyset$. We inductively construct finite subsets $\cF_n \subset G$ and $\cG_n \subset G$.

Assume that $n \geq 1$ and that $\cF_{n-1}$ and $\cG_{n-1}$ have been fixed. Choose a large enough finite subset $\cG_n \subset G$ such that $\cH_n \subset \cG_n$, $\cG_n^{-1} = \cG_n$ and $|\cG_n| \geq \exp(n \cdot |\cF_{n-1}|)$. Then choose a large enough finite F{\o}lner set $\cF_n \subset G$ such that
$$\cG_n \cF_{n-1} \cup \cF_{n-1} \subset \cF_n \quad , \quad |\cF_n| \geq \max\bigl\{2^n \lambda_{n+1} , \lambda_n \, |\cF_{n-1}| \bigr\} \quad , \quad \frac{|\cG_n \cF_n \setminus \cF_n|}{|\cF_n|} \leq 2^{-n} \lambda_{n+1}^{-1} \; .$$
Write $\eps_n = |\cF_n|^{-1}$ and $\gamma_n = \lambda_{n+1} \, \eps_n$. By construction, $\gamma_n$ and $\eps_n$ are decreasing sequences and $\gamma_n \leq 1$ for all $n$.

Fix any probability space $(X_0,\nu)$ with an increasing sequence of subsets $\cU_n \subset X_0$ such that $\nu(X_0 \setminus \cU_n) = \eps_n$. Define the probability measures $\mu_n \sim \nu$ by
$$\frac{d\mu_n}{d\nu}(x) = \begin{cases} \rho_n^{-1} &\;\;\text{if $x \in \cU_n$}, \\ \rho_n^{-1} \, \lambda_n &\;\;\text{if $x \in X_0 \setminus \cU_n$,}\end{cases} \quad\text{where}\quad \rho_n = 1 + \eps_n(\lambda_n - 1) \; .$$
Note that $1 \leq \rho_n \leq 1 + \gamma_n \leq 2$. In particular, for a.e.\ $x \in X_0$, we have that
\begin{equation}\label{eq.local-bound-example}
\sup_n \bigl|\log \frac{d\mu_n}{d\nu}(x)\bigr| < + \infty \; .
\end{equation}
For all $0 \leq n < m$, we have that
\begin{align*}
\int_{X_0} \frac{d\mu_n}{d\mu_m} \, d\mu_n &= \rho_n^{-2} \, \rho_m \, (1-\eps_n + \lambda_n^2(\eps_n - \eps_m) + \lambda_n^2 \lambda_m^{-1} \eps_m) \\ &\leq (1+\gamma_m)(1 + \gamma_n) \leq 1 + 3 \gamma_n \leq \exp(3 \gamma_n) \; .
\end{align*}
Similarly, for all $0 \leq n < m$,
$$
\int_{X_0} \frac{d\mu_m}{d\mu_n} \, d\mu_m \leq \exp(3 \gamma_n) \; .
$$
Define the map $N : G \recht \N \cup \{0\}$ by \eqref{eq.def-function-N}. We prove that
$$G \actson (X,\mu) = \prod_{g \in G} (X_0,\mu_{N(g)})$$
is a nonsingular Bernoulli action that is weakly mixing and of stable type III$_0$.

We claim that for every $n \geq 1$ and $g \in \cG_n$,
\begin{equation}\label{eq.crucial-estimate-towards-conservative}
\prod_{h \in G} \int_{X_0} \frac{d\mu_{N(h)}}{d\mu_{N(gh)}} \, d\mu_{N(h)} \leq \exp(6 (|\cF_{n-1}|+1)) \; .
\end{equation}
Fix $n \geq 1$ and fix $g \in \cG_n$. To analyze the different parts of the infinite product in \eqref{eq.crucial-estimate-towards-conservative}, we denote for every finite subset $\cK \subset G$,
$$P(\cK) = \prod_{h \in \cK} \int_{X_0} \frac{d\mu_{N(h)}}{d\mu_{N(gh)}} \, d\mu_{N(h)} \; .$$
Obviously, $P(\cK) = 1$ if $\cK \subset \{h \in G \mid N(gh) = N(h)\}$. We then use Lemma \ref{lem.function-N-Folner}. Since $P(\{h\}) \leq \exp(3)$ for every $h \in G$, we have that
$$P(\cF_{n-1} \cup g^{-1}\cF_{n-1}) \leq \exp(6|\cF_{n-1}|) \; .$$
For every $k \geq n$, we have that
\begin{align*}
P(\cF_k \setminus g^{-1}\cF_k) &\leq \exp(3 \gamma_k |\cF_k \setminus g^{-1}\cF_k|) = \exp(3 \gamma_k |g \cF_k \setminus \cF_k|)\\
& \leq \exp(3 \gamma_k 2^{-k} \lambda_k^{-2} |\cF_k|) = \exp(3 \cdot 2^{-k}) \; .
\end{align*}
Similarly,
$$P(g^{-1} \cF_k \setminus \cF_k) \leq \exp(3 \gamma_k |g^{-1} \cF_k \setminus \cF_k|) \leq \exp(3 \cdot 2^{-k}) \; .$$
Multiplying all contributions, we conclude that \eqref{eq.crucial-estimate-towards-conservative} holds.

Using the notation in \eqref{eq.notation-C}, we find that for all $n \geq 1$ and $g \in \cG_n$,
$$C(g) \leq 3(|\cF_{n-1}| + 1) < +\infty \; .$$
So, the Bernoulli action $G \actson (X,\mu)$ is nonsingular. Using $s = 3 (|\cF_{n-1}| + 1)$ and using that $\cG_n = \cG_n^{-1}$, we find that
\begin{align*}
\limsup_{s \recht +\infty} \frac{\log\bigl| \{g \in G \mid C(g^{\pm 1}) \leq s\}\bigr|}{s} & \geq \limsup_{n \recht +\infty} \frac{\log |\cG_n|}{3 (|\cF_{n-1}| + 1)} \\
& \geq \limsup_{n \recht+\infty} \frac{n \cdot |\cF_{n-1}|}{3 (|\cF_{n-1}| + 1)} = +\infty \; .
\end{align*}
By Lemma \ref{lem.strongly-conservative}, $G \actson (X,\mu)$ is strongly conservative. By \eqref{eq.local-bound-example} and Proposition \ref{prop.ergodic-general}, the action is weakly mixing. Also, Lemma \ref{lem.link-to-permutation-action} applies.

Fix an ergodic pmp action $G \actson (Y,\eta)$ and consider the diagonal action $G \actson X \times Y$. We prove that $G \actson X \times Y$ is not semifinite. Assume the contrary. By Lemmas \ref{lem.link-to-permutation-action} and \ref{lem.trivial-flow-result}, also the permutation action $\cS_G \actson (X,\mu)$ is semifinite. As explained in the first paragraphs of the proof of Theorem \ref{thm.T-invariant-permutation-action}, we then find a bounded family $t_h \in \R$ and a measurable map $\be : X_0 \recht \R$ such that
\begin{align}
& \lim_{h \recht \infty} \bigl( \log \frac{d\mu_h}{d\nu}(x) - t_h \bigr) = \be(x) \quad\text{for a.e.\ $x \in X_0$,} \notag\\
& \sum_{h \in G} \int_{X_0} T_\kappa\bigl(\log \frac{d\mu_h}{d\nu}(x) - t_h - \be(x)\bigr)^2 \, d\mu_h(x) < +\infty \; ,\label{eq.this-sum-is-finite}
\end{align}
for $\kappa > 0$. By construction, $\lim_{h \recht \infty} \rho_{N(h)} = 1$ and
$$\lim_{h \recht \infty} \log \frac{d\mu_h}{d\nu}(x) = 0 \quad\text{for a.e.\ $x \in X_0$.}$$
We conclude that $\be$ is constant a.e.\ and we may thus assume that $\be = 0$. Also by construction, $\log(d\mu_h / d\nu)$ only takes the values $-\log(\rho_{N(h)})$ on $\cU_{N(h)}$ and $\log(\lambda_{N(h)}) - \log(\rho_{N(h)})$ on the complement $X_0 \setminus \cU_{N(h)}$. The difference between both values tends to infinity, so that $t_h$ can only be close to one of both. It thus follows from \eqref{eq.this-sum-is-finite} that
\begin{align*}
+\infty &> \sum_{h \in G} \min \{\mu_{N(h)}(\cU_{N(h)}), \mu_{N(h)}(X_0 \setminus \cU_{N(h)})\} \geq \sum_{h \in G} \eps_{N(h)} \\
& \geq \sum_{n = 1}^\infty \eps_n \, |\cF_n \setminus \cF_{n-1}| \geq \frac{1}{2} \sum_{n = 1}^\infty \eps_n \, |\cF_n| = +\infty \; ,
\end{align*}
because $\eps_n = |\cF_n|^{-1}$. This result is absurd, so that $G \actson X \times Y$ is not semifinite.

To prove that $G \actson (X,\mu)$ is of stable type III$_0$, we prove that for every $n \geq 1$, there exists a probability measure $\mutil \sim \mu$ such that
$$\frac{d(g^{-1} \cdot \mutil)}{d\mutil}(x) \in 2^{2^n \Z} \quad\text{for all $g \in G$ and a.e.\ $x \in X$,}$$
because then Krieger's ratio set of $G \actson X \times Y$ is contained in $\{0,1\}$.

Fix $n \geq 1$. Define $\mutil \sim \mu$ as the product of the probability measure $\mu_n$ in the coordinates $g \in \cF_n$ and the probability measures $\mu_{N(g)}$ in the coordinates $g \in G \setminus \cF_n$. For every $g \in G$, write $\rho_g = \rho_{N(g)}$. Also define $\rhotil_g = \rho_n$ for all $g \in \cF_n$ and $\rhotil_g = \rho_{N(g)}$ for all $g \in G \setminus \cF_n$.

Define the continuous map $\theta : (0,+\infty) \recht \R/2^n\Z : \theta(r) = \log_2(r) + 2^n \Z$. By construction, for all $g \in G$ and a.e.\ $x \in X$,
\begin{equation}\label{eq.my-series}
\theta\bigl( \frac{d(g^{-1} \cdot \mutil)}{d\mutil}(x) \bigr) = \sum_{h \in G} (-\theta(\rhotil_{gh}) + \theta(\rhotil_h)) \; .
\end{equation}
To conclude the proof of the theorem, it thus suffices to show that the series at the right hand side of \eqref{eq.my-series} converges to $0$ for all $g \in G$. We actually prove that the series in $\R$ given by
$$\sum_{h \in G} (\log(\rhotil_{gh}) - \log(\rhotil_h))$$
is absolutely convergent to $0$. Dividing by $\log(2)$ and reducing modulo $2^n \Z$, the required conclusion then follows.

Since
$$\sum_{n=1}^\infty |\cF_n| \, (\eps_n (\lambda_n-1))^2 \leq \sum_{n=1}^\infty |\cF_n|^{-1} \, \lambda_n^2 \leq \sum_{n=1}^\infty 2^{-n} < +\infty \; ,$$
we get that, writing $\zeta_h = \eps_{N(h)}(\lambda_{N(h)} - 1)$,
$$\sum_{h \in G} |\log(\rhotil_{h}) - \zeta_h| < +\infty \; .$$
It thus suffices to prove that for every $g \in G$, the series $\sum_{h \in G} (\zeta_{gh} - \zeta_h)$ is absolutely convergent with sum zero. Fix $g \in G$. We will prove that
\begin{equation}\label{eq.final-goal}
\lim_{m \recht +\infty} \sum_{h \in G \setminus \cF_m} |\zeta_{gh} - \zeta_h| = 0 \quad\text{and}\quad \lim_{m \recht +\infty} \sum_{h \in \cF_m} (\zeta_{gh} - \zeta_h) = 0 \; .
\end{equation}
Fix $n_0 \in \N$ such that $g \in \cG_{n_0}$. By Lemma \ref{lem.function-N-Folner}, for every $m \geq n_0$,
$$\sum_{h \in G \setminus \cF_m} |\zeta_{gh} - \zeta_h| \leq \sum_{k=m}^\infty |\cF_k \setminus g^{-1} \cF_k| \, \eps_k \lambda_k + \sum_{k=m}^\infty |g^{-1} \cF_k \setminus \cF_k| \, \eps_k \lambda_k \leq 2 \sum_{k=m}^\infty 2^{-k} = 2^{-m+2} \; .$$
On the other hand,
\begin{align*}
\Bigl| \sum_{h \in \cF_m} (\zeta_{gh} - \zeta_h) \Bigr| &\leq \sum_{h \in g \cF_m \setminus \cF_m} \zeta_h + \sum_{h \in \cF_m \setminus g \cF_m} \zeta_h \\
& \leq |g \cF_m \setminus \cF_m| \, \eps_{m+1} \lambda_{m+1} + |\cF_m \setminus g \cF_m| \, \eps_m \lambda_m \leq 2^{-m+1} \; .
\end{align*}
So, \eqref{eq.final-goal} holds and the theorem is proven.
\end{proof}

\begin{theorem}\label{thm.amenable-type-II-infty}
Every infinite amenable group admits a nonsingular Bernoulli action that is weakly mixing and of type II$_\infty$.
\end{theorem}
\begin{proof}
Fix an increasing sequence of finite subsets $\cH_n \subset G$ such that $\cH_0 = \{e\}$ and $G = \bigcup_n \cH_n$. We inductively define the following increasing sequences of finite subsets $\cF_n \subset G$ and $\cG_n \subset G$, together with subsets $\cS_n \subset \cG_n \setminus \cG_{n-1}$ and a strictly increasing sequence $0 < \rho_n < 1$ converging to $1$.

Start by defining $\rho_0 = 1/2$, $\cF_0 = \cG_0 = \{e\}$ and $\cS_0 = \emptyset$. Assume that $n \geq 1$ and that $\cF_{n-1}$, $\cG_{n-1}$, $\cS_{n-1}$ and $\rho_{n-1}$ have been defined. Put
$$\delta_n = \rho_0 \cdot \prod_{k=1}^{n-1} \rho_k^{|\cF_k \setminus \cF_{k-1}|} > 0 \; .$$
Take an integer $k_n \geq 1$ large enough such that $(1-\delta_n)^{k_n} < 2^{-n}$. Choose a subset $\cS_n \subset G \setminus (\cG_{n-1} \cup \cF_{n-1}\cF_{n-1}^{-1})$ with $|\cS_n| = k_n$ and such that the $k_n + 1$ sets $g \cF_{n-1}$ with $g \in \{e\} \cup \cS_n$ are all disjoint. Choose a finite subset $\cG_n \subset G$ such that $\cG_n = \cG_n^{-1}$ and $\cG_{n-1} \cup \cS_n \cup \cH_n \subset \cG_n$. Then pick a large enough finite F{\o}lner set $\cF_n \subset G$ such that
$$\cG_n \cF_{n-1} \subset \cF_n \quad , \quad |\cF_n| \geq 2 |\cF_{n-1}| \quad , \quad (1-\rho_{n-1}) |\cF_n \setminus \cF_{n-1}| > 1  \quad , \quad
 \frac{|\cG_n \cF_n \setminus \cF_n|}{|\cF_n|} \leq 2^{-n-3} \; .$$
Then take $\rho_n$ such that $(1-\rho_n) |\cF_n \setminus \cF_{n-1}| = 1$. Note that $\rho_{n-1} < \rho_n < 1$. Also note that $\cF_{n-1} \subset \cF_n$ and $\cH_n \subset \cG_n$. Since $1/2 = \rho_0 < \rho_n$, we have
$$- \log\bigl(\rho_n^{|\cG_n \cF_n \setminus \cF_n|}\bigr) \leq 2 \, |\cG_n \cF_n \setminus \cF_n| \, (1-\rho_n) = 2 \, \frac{|\cG_n \cF_n \setminus \cF_n|}{|\cF_n \setminus \cF_{n-1}|} \leq 4 \, \frac{|\cG_n \cF_n \setminus \cF_n|}{|\cF_n|} \leq 2^{-n-1}$$
and thus
\begin{equation}\label{eq.important-bound}
\rho_n^{|\cG_n \cF_n \setminus \cF_n|} \geq 1 - 2^{-n-1} \; .
\end{equation}

Choose any standard probability space $(X_0,\nu)$ equipped with an increasing sequence of subsets $\cU_n \subset X_0$ such that $\nu(\cU_n) = \rho_n$. Denote by $\nu_n$ the probability measure $\nu_n = \rho_n^{-1} \nu|_{\cU_n}$ given by restricting $\nu$ to $\cU_n$. Define the function $N : G \recht \N \cup \{0\}$ as in \eqref{eq.def-function-N}. We write $\cU_g := \cU_{N(g)}$, $\nu_g = \nu_{N(g)}$ and $\rho_g = \rho_{N(g)}$. By taking arbitrarily small values on $X_0 \setminus \cU_g$, we can choose probability measures $\mu_g \sim \nu$ such that
$$\sum_{g \in G} H^2(\mu_g,\nu_g) < +\infty \; .$$
We prove that the Bernoulli action
$$G \actson (X,\mu) = \prod_{g \in G} (X_0,\mu_g)$$
is nonsingular, weakly mixing and of type II$_\infty$.

For every $n \in \N$, define the subset $Y_n \subset X$ by
$$Y_n = \{x \in X \mid x_h \in \cU_h \;\;\text{for all $h \in G \setminus \cF_n$, and}\;\; x_h \in \cU_n \;\;\text{for all $h \in \cF_n$}\;\}$$
and equip $Y_n$ with the probability measure $\eta_n$ given by the product of the probability measure $\nu_n$ in the coordinates $x_h$, $h \in \cF_n$, and the probability measures $\nu_h$ in the coordinates $x_h$, $h \in G \setminus \cF_n$.

When $n \in \N$ and $g \in \cG_n$, it follows from Lemma \ref{lem.function-N-Folner} that
$$g \cdot Y_n \cap Y_n = \{ x \in Y_n \mid \;\text{for all $k \geq n$ and $h \in g \cF_k \setminus \cF_k$, we have}\; x_h \in \cU_k \;\} \; .$$
Note that whenever $1 \leq m \leq n$ and $g \in \cG_n$, we have
\begin{equation}\label{eq.estimate-measure-Yn}
\eta_m(Y_m \cap (g \cdot Y_n \cap Y_n)) = \prod_{k=n}^\infty \Bigl(\frac{\rho_k}{\rho_{k+1}}\Bigr)^{|g\cF_k \setminus \cF_k|} \geq \prod_{k=n}^\infty \rho_k^{|g \cF_k \setminus \cF_k|} \; .
\end{equation}
By \eqref{eq.important-bound}, the right hand side of \eqref{eq.estimate-measure-Yn} tends to $1$ when $n \recht +\infty$. When $n \in \N$ and $g \in \cG_n$, \eqref{eq.estimate-measure-Yn} says that $\eta_n(g \cdot Y_n \cap Y_n)>0$. Moreover, for $g \in \cG_n$, the map from $g \cdot Y_n \cap Y_n$ to $g^{-1}\cdot Y_n \cap Y_n$ given by $x \mapsto g^{-1} \cdot x$ scales the measure $\eta_n$ with a factor $\al(n,g)$ that, using Lemma \ref{lem.function-N-Folner}, we can write as
$$\al(n,g) = \lim_{l \recht +\infty} \prod_{h \in \cF_n \setminus g \cF_n} \frac{\rho_n}{\rho_{n+1}} \cdot \prod_{k=n+1}^l \Bigl( \prod_{h \in g \cF_{k-1} \setminus \cF_{k-1}} \frac{\rho_k}{\rho_{k-1}} \cdot \prod_{h \in \cF_k \setminus g \cF_k} \frac{\rho_k}{\rho_{k+1}} \Bigr) \; .$$
For all $k \in \N$, we have
$$|g \cF_k \setminus \cF_k| = |g \cF_k \cup \cF_k| - |\cF_k| = |g \cF_k \cup \cF_k| - |g\cF_k| = |\cF_k \setminus g \cF_k| \; .$$
Using \eqref{eq.important-bound}, we find that
$$\al(n,g) = \lim_{l \recht +\infty} \Bigl(\frac{\rho_l}{\rho_{l+1}}\Bigr)^{|g \cF_l \setminus \cF_l|} = 1 \; .$$
So, whenever $g \in \cG_n$, the map $x \mapsto g^{-1} \cdot x$ from $g \cdot Y_n \cap Y_n$ to $g^{-1}\cdot Y_n \cap Y_n$ preserves the measure $\eta_n$. Whenever $n \geq m \geq 1$, the restriction of $\eta_n$ to $Y_m \subset Y_n$ is by construction a multiple of $\eta_m$. We thus conclude that for all $g \in G$ and all $m \in \N$, the map $x \mapsto g^{-1} \cdot x$ from $g \cdot Y_m \cap Y_m$ to $g^{-1}\cdot Y_m \cap Y_m$ preserves the measure $\eta_m$. This means that the restriction of the orbit equivalence relation of $G \actson X$ to the subset $Y_m$, given by
$$\cR_m = \{(x,y) \in Y_m \times Y_m \mid y \in G \cdot x\} \; ,$$
preserves the probability measure $\eta_m$.

Note that by construction, the subset $\bigcup_{m=1}^\infty Y_m$ of $X$ is $\mu$-conull. Also, the restriction of $\mu$ to $Y_m$ is equivalent with $\eta_m$ for every $m \in \N$. For every $g \in G$, write $Y_g = \bigcup_{n=1}^\infty (g \cdot Y_n \cap Y_n)$. By \eqref{eq.estimate-measure-Yn}, we have that $\eta_m(Y_m \setminus Y_g) = 0$ for all $m \in \N$ and $g \in G$. Therefore, $\mu(Y_m \setminus Y_g) = 0$ for all $m \in \N$ and $g \in G$. We conclude that $Y_g \subset X$, and thus also $Y_{g^{-1}} \subset X$, are $\mu$-conull for all $g \in G$. Acting by $g^{-1}$ from $g \cdot Y_n \cap Y_n$ to $g^{-1} \cdot Y_n \cap Y_n$ is $\eta_n$-preserving and thus $\mu$-nonsingular. This holds for all $n$, so that acting by $g^{-1}$ from $Y_g$ to $Y_{g^{-1}}$ is $\mu$-nonsingular. Since both $Y_g$ and $Y_{g^{-1}}$ are $\mu$-conull, we have proven that $G \actson (X,\mu)$ is a nonsingular action.

Below we prove that for every $m \in \N$ the equivalence relation $\cR_m$ has infinite orbits a.e. Since it preserves the probability measure $\eta_m$, it then follows that $\cR_m$ is conservative. So, it follows that $G \actson (X,\mu)$ is conservative. Since $G$ is amenable, we then conclude using Proposition \ref{prop.ergodic-general} that $G \actson (X,\mu)$ is weakly mixing and that each $\cR_m$ is an ergodic equivalence relation of type II$_1$.

For every $n \geq m$, the restriction of $\eta_n$ to $Y_m \subset Y_n$ is a multiple of $\eta_m$. So one of the following holds: either $\lim_{n \recht +\infty} \eta_n(Y_m) = 0$ for all $m \in \N$, or there exists a probability measure $\eta$ on $X$ such that the restriction of $\eta$ to $Y_n$ equals $\eta(Y_n) \eta_n$ for all $n \in \N$. We claim that the second option does not hold. Assume that it does. In particular, $\eta \sim \mu$. Choose a cylinder set $A \subset X$, given by a finite family of sides $A_g \subset X_0$, $g \in \cF$, such that
$$A = \{x \in X \mid x_g \in A_g \;\;\text{for all}\;\; g \in \cF \} \; .$$
Whenever $n \in \N$ is large enough so that $\cF \subset \cF_n$, we have that
$$\eta_n(Y_n \cap A) = \prod_{g \in \cF} \nu_n(A_g \cap \cU_n) \recht \prod_{g \in \cF} \nu(A_g) = \nu^G(A) \; .$$
It follows that $\eta = \nu^G$, so that $\nu^G \sim \mu$. Since $\mu(Y_1) > 0$, also $\nu^G(Y_1) > 0$, meaning that
$$\sum_{g \in G} \nu(X_0 \setminus \cU_g) < +\infty \; .$$
It follows that
$$+\infty > \sum_{n=1}^\infty |\cF_n \setminus \cF_{n-1}| \, (1-\rho_n) = + \infty \; ,$$
because $|\cF_n \setminus \cF_{n-1}| \, (1-\rho_n) = 1$ for all $n$. We obtained a contradiction and conclude that $\lim_{n \recht +\infty} \eta_n(Y_m) = 0$ for all $m \in \N$. It follows that the union of all $\cR_m$ is an ergodic equivalence relation of type II$_\infty$. By construction, this union equals the orbit equivalence relation of $G \actson (X,\mu)$, up to measure zero. We conclude that $G \actson (X,\mu)$ is of type II$_\infty$.

Fix $m \in \N$. It remains to prove that $G \cdot x \cap Y_m$ is infinite for a.e.\ $x \in Y_m$. For every $n \geq m+1$, define the subset $Z_n \subset Y_m$ by
$$Z_n = \{x \in Y_m \mid \exists g \in \cS_n \;\;\text{such that}\;\; g^{-1} \cdot x \in Y_m \} \; .$$
We prove that $\eta_m(Z_n) > 1 - 2^{-n+1}$ for all $n \geq m+1$. It then follows that a.e.\ $x \in Y_m$ has the property that $x \in Z_n$ for all $n$ large enough. Since the sets $\cS_n$ are disjoint, this means that for a.e.\ $x \in Y_m$, there are infinitely many $g \in G$ with $g^{-1} \cdot x \in Y_m$.

For every subset $\cL \subset G \setminus \cF_m$, we define
$$Y_\cL = \prod_{h \in \cL} \cU_h \; ,$$
we consider the natural map $Y_m \recht Y_\cL : x \mapsto x_{\cL}$ and equip $Y_\cL$ with the probability measure $\eta_\cL$ such that $(Y_m,\eta_m) \recht (Y_\cL,\eta_\cL) : x \mapsto x_\cL$ is measure preserving.

Fix $n \geq m+1$. By our choice of $\cS_n$, we have that $\cS_n \cF_{n-1} \subset G \setminus \cF_{n-1}$. By our choice of $\cF_n$, we have that $\cS_n \cF_{n-1} \subset \cF_n$. So, for every $g \in \cS_n$, we have that $g \cF_{n-1} \subset \cF_n \setminus \cF_{n-1} \subset G \setminus \cF_m$. We define, for every $g \in \cS_n$,
$$\cW_g = \{x \in Y_{g \cF_{n-1}} \mid x_{gh} \in \cU_h \;\;\text{for all}\;\; h \in \cF_{n-1} \}$$
and note that
$$\eta_{g \cF_{n-1}}(\cW_g) \geq \prod_{h \in \cF_{n-1}} \rho_h = \delta_n \; .$$
Since the subsets $(g \cF_{n-1})_{g \in \cS_n}$ of $\cF_n \setminus \cF_{n-1}$ are disjoint, the subset $\cW \subset Y_{\cF_n \setminus \cF_{n-1}}$ defined by
$$\cW = \{x \in Y_{\cF_n \setminus \cF_{n-1}} \mid \;\text{there exists a $g \in \cS_n$ such that $x_{g \cF_{n-1}} \in \cW_g$}\;\}$$
has measure
$$\eta_{\cF_n \setminus \cF_{n-1}}(\cW) \geq 1 - (1-\delta_n)^{|\cS_n|} = 1 - (1-\delta_n)^{k_n} > 1 - 2^{-n} \; .$$
For every $k \geq n$, define $\cV_k \subset Y_{\cF_{k+1} \setminus \cF_k}$ by
$$\cV_k = \{ x \in Y_{\cF_{k+1} \setminus \cF_k} \mid x_h \in \cU_k \;\;\text{for all}\;\; h \in (\cF_{k+1} \setminus \cF_k) \cap \cS_n \cF_k \;\} \; .$$
Since $\cS_n \subset \cG_n \subset \cG_k$, we have that
$$|(\cF_{k+1} \setminus \cF_k) \cap \cS_n \cF_k| \leq |\cG_k \cF_k \setminus \cF_k| \quad\text{and}\quad \eta_{\cF_{k+1} \setminus \cF_k}(\cV_k) \geq \rho_k^{|\cG_k \cF_k \setminus \cF_k|} \geq 1 - 2^{-k-1} \; .$$
Defining the subset $\cV \subset Y_{G \setminus \cF_n}$ by
$$\cV = \{x \in Y_{G \setminus \cF_n} \mid x_{\cF_{k+1} \setminus \cF_k} \in \cV_k \;\;\text{for all $k \geq n$}\;\} \; ,$$
we get that $\eta_{G \setminus \cF_n}(\cV) \geq 1-2^{-n}$. Finally, define $\cZ \subset Y_m$ by
$$\cZ = \{x \in Y_m \mid x_{\cF_n \setminus \cF_{n-1}} \in \cW \;\;\text{and}\;\; x_{G \setminus \cF_n} \in \cV \} \; .$$
We have that $\eta_m(\cZ) \geq 1 - 2^{-n+1}$. We prove that $\cZ \subset Z_n$.

Take $x \in \cZ$. Since $x_{\cF_n \setminus \cF_{n-1}} \in \cW$, pick $g \in \cS_n$ such that $x_{g \cF_{n-1}} \in \cW_g$. We prove that $g^{-1} \cdot x \in Y_m$. So, we have to show that $x_{gh} \in \cU_m$ for all $h \in \cF_m$ and $x_{gh} \in \cU_h$ for all $h \in G \setminus \cF_m$.

When $h \in \cF_m \subset \cF_{n-1}$, because $x_{g \cF_{n-1}} \in \cW_g$, we get that $x_{gh} \in \cU_h \subset \cU_m$. When $h \in \cF_{n-1} \setminus \cF_m$, because $x_{g \cF_{n-1}} \in \cW_g$, we have $x_{gh} \in \cU_h$. Finally, if $k \geq n$ and $h \in \cF_k \setminus \cF_{k-1}$, we either have $gh \in \cF_k$, in which case $x_{gh} \in \cU_k = \cU_h$, or we have $g h \not\in \cF_k$, in which case $gh \in \cF_{k+1} \setminus \cF_k$ because $\cS_n \cF_k \subset \cF_{k+1}$. So, $g h \in (\cF_{k+1} \setminus \cF_k) \cap \cS_n \cF_k$. Because $x_{\cF_{k+1} \setminus \cF_k} \in \cV_k$, we find that $x_{gh} \in \cU_k = \cU_h$. This ends the proof of the theorem.
\end{proof}

\section{Examples of nonsingular Bernoulli actions with diffuse base}

We prove Theorem \ref{thmD.examples-diffuse}. Consider the setting introduced in the paragraphs preceding Theorem \ref{thmD.examples-diffuse}, with the probability measures $\mu_g = \nu_{F(g)}$ on $\R$, where $d\nu_s(t) = \vphi(t+s) \, dt$, and
$$G \actson (X,\mu) = \prod_{g \in G} (\R,\mu_g) \; .$$
Note that
\begin{equation}\label{eq.formula-R-N-transl}
\frac{d\mu_g}{d\mu_e}(t) = \frac{\vphi(t + F(g))}{\vphi(t + F(e))} \quad\text{for all $g \in G$, $t \in \R$.}
\end{equation}

\begin{proof}[{Proof of Theorem \ref{thmD.examples-diffuse}}]
The function $F : G \recht \R$ is bounded. So by adding a constant and replacing $\vphi$ by a translation, we may assume that $0$ is a limit value of $(F(g))_{g \in G}$.

We start by estimating in both cases of Theorem \ref{thmD.examples-diffuse}, the function
\begin{equation}\label{eq.function-theta}
\theta(s) = \int_\R \frac{\vphi(t+s)^2}{\vphi(t)} \, dt \; .
\end{equation}
Denote $\psi = \log \vphi$.

1.\ First assume that $\psi$ is a Lipschitz function with constant $M \geq 0$. Since
$$\theta(s) - 1 = \int_\R \frac{\bigl(\vphi(t+s) - \vphi(t)\bigr)^2}{\vphi(t)} \, dt \; ,$$
we find that
$$\theta(s) - 1 \leq \int_\R \bigl(\exp(M |s|) - 1\bigr)^2 \, \vphi(t) \, dt = \bigl(\exp(M |s|) - 1\bigr)^2 \; .$$
Since for every $t \geq 0$, $(\exp(t) - 1)^2 \leq \exp(3 t^2 / 2) - 1$, we conclude that
$$\frac{1}{2} \log(\theta(s)) \leq \frac{3}{4} M^2 \, s^2 \quad\text{for all $s \in \R$.}$$

2.\ Next assume that $\psi$ is differentiable and that $\psi'$ is Lipschitz with constant $M \geq 0$. Note that
$$\theta(s) = \int_\R \exp\bigl( 2 \psi(t+s) -\psi(t)\bigr) \, dt = \int_\R \exp\bigl( 2\psi(t-s) - \psi(t-2s) - \psi(t)\bigr) \, \vphi(t) \, dt \; .$$
For all $s,t \in \R$, there exist $r_1$ between $t-s$ and $t$, and $r_2$ between $t-2s$ and $t-s$ such that
$$\psi(t-s) - \psi(t) = - s \psi'(r_1) \quad\text{and}\quad \psi(t-2s) - \psi(t-s) = - s \psi'(r_2) \; .$$
It follows that
$$|2\psi(t-s) - \psi(t-2s) - \psi(t)| \leq |s| \, |\psi'(r_1) - \psi'(r_2)| \leq 2 M s^2 \; .$$
We conclude that $(1/2) \log(\theta(s)) \leq M s^2$ for all $s \in \R$.

Writing $\kappa = 3M^2/4$ in case~1 of the theorem and $\kappa = M$ in case~2, we find that for all $g \in G$,
\begin{align*}
& \sum_{h \in G} H^2(\mu_{gh},\mu_h) \leq \frac{1}{2} \sum_{h \in G} \log(\theta(c_{g^{-1}}(h))) \leq \kappa \, \|c_g\|^2 < +\infty \quad\text{and}\\
& \sum_{h \in G} H^2(\mu_h,\nu_0) \leq \frac{1}{2} \sum_{h \in G} \log(\theta(F(h))) \leq \kappa \, |F(h)|^2 \; .
\end{align*}
We conclude that $G \actson (X,\mu)$ is nonsingular. Since $0$ is a limit value of $(F(g))_{g \in G}$, we have that $c_g$ is a coboundary if and only if $F \in \ell^2(G)$. If this is the case, we conclude that $\mu \sim \nu_0^G$.

When $\delta(c) > 6 \kappa$, we get that \eqref{eq.strong-conservative-assumption} holds. So, by Lemma \ref{lem.strongly-conservative} and Proposition \ref{prop.ergodic-general}, the action $G \actson (X,\mu)$ is weakly mixing. When $c_g$ is a coboundary, we have already proven that $\mu$ is equivalent to the $G$-invariant probability measure $\nu_0^G$, so that $G \actson X$ is of type II$_1$. For the rest of the proof, we assume that $c_g$ is not a coboundary.

Fix an ergodic pmp action $G \actson (Y,\eta)$ and consider the diagonal action $G \actson X \times Y$. We prove that this action is of type III$_1$. We first prove that this action is not of type II$_\infty$ and not of type III$_0$. We again distinguish between the two cases of the theorem, always assuming that $\delta(c) > 6 \kappa$.

1.\ Using that $F$ is a bounded function, it follows from \eqref{eq.formula-R-N-transl} that the strong boundedness property \eqref{eq.boundedness-strong} holds. By Theorem \ref{thm.rule-out-III0}, the action $G \actson X \times Y$ is not of type II$_\infty$ and not of type III$_0$.

2.\ We now additionally assume that $\int_\R |t|^\al \, \vphi(t) \, dt < +\infty$ for some $\al > 2$. Since $F$ is bounded and $\vphi$ is continuous, it follows from \eqref{eq.formula-R-N-transl} that the weak boundedness property \eqref{eq.boundedness-weak} holds. Consider the permutation action $\cS_G \actson (X,\mu)$. By Lemma \ref{lem.link-to-permutation-action}, it suffices to prove that either $\mu \sim \nu^G$ for some probability measure $\nu$ on $\R$, or the associated flow of $\cS_G \actson (X,\mu)$ is periodic. Take a sequence $g_n \recht \infty$ in $G$ such that $F(g_n) \recht 0$. By \eqref{eq.formula-R-N-transl}, we have that $\log(d\mu_{g_n}/d\nu_0(t)) \recht 0$ for all $t \in \R$.

If $a \neq 0$ is another limit point of $(F(g))_{g \in G}$, say $F(h_n) \recht a$, we find that $\log(d\mu_{h_n}/d\mu_0(t)) \recht \psi(t+a) - \psi(t)$ for all $t \in \R$. Since $\vphi$ is strictly positive and integrable, the function $t \mapsto \psi(t+a) - \psi(t)$ cannot be constant. It then follows from the second point of Theorem \ref{thm.permutation-condition-periodic} that the associated flow of $\cS_G \actson (X,\mu)$ is periodic.

Next assume that $\lim_{g \recht \infty} F(g) = 0$. By the third point of Lemma \ref{lem.link-to-tail-boundary}, the associated flow of $\cS_G \actson (X,\mu)$ is given by the tail boundary flow associated with the probability measures $\zeta_g$ given by pushing forward $\mu_g$ with the map $t \mapsto - \psi(t+F(g)) + \psi(t)$. Note that
$$\zeta_g := (\gamma_g)_*(\nu_0) \quad\text{where}\quad \gamma_g(t) = - \psi(t) + \psi(t-F(g)) \; .$$
We apply point~5 of Theorem \ref{thm.tail-periodic-criterion}. Since $\psi'$ is Lipschitz with constant $M$ and since the function $F$ is bounded, we find $A > 0$ such that
\begin{equation}\label{eq.some-bounds}
|\psi'(t)| \leq A + M |t| \quad\text{and}\quad |\gamma_g(t)| \leq (A + M |t|) \, |F(g)| \quad\text{for all $g \in G$, $t \in \R$.}
\end{equation}
In particular, because $\nu_0$ has a finite second moment, we find $B > 0$ such that
$$\int_\R t^2 \, d\zeta_g(t) \leq  F(g)^2 \, \int_\R (A + M |t|)^2 \, d\nu_0(t) \leq B \, F(g)^2$$
for all $g \in G$. Also, the limits
\begin{align*}
& \lim_{\eps \recht 0} \frac{1}{\eps^2} \int_\R (- \psi(t) + \psi(t-\eps))^2 \, d\nu_0(t) = \int_\R (\psi'(t))^2 \, \vphi(t) \, dt > 0 \quad\text{and} \\
& \lim_{\eps \recht 0} \frac{1}{\eps} \int_\R (- \psi(t) + \psi(t-\eps)) \, d\nu_0(t) = - \int_\R \psi'(t) \, \vphi(t) \, dt = 0
\end{align*}
exist. Since $\lim_{g \recht \infty} F(g) = 0$ and $\Var \zeta_g = 0$ if and only if $F(g) = 0$, we find a $\delta > 0$ such that
\begin{equation}\label{eq.variances}
\delta \, F(g)^2 \leq \Var \zeta_g \leq B \, F(g)^2 \quad\text{for all $g \in G$.}
\end{equation}
By \eqref{eq.some-bounds}, we find $C_1 > 0$ such that $(\gamma_g(t))^2 \leq C_1 \, F(g)^2 \, (1 + t^2)$ for all $g \in G$, $t \in \R$. It follows that for all $D \geq 1$ and all $g \in G$,
\begin{align*}
\int_{|t| \geq D} (\gamma_g(t))^2 \, \vphi(t) \, dt & \leq 2 C_1 \, F(g)^2 \, \int_{|t| \geq D} t^2 \, \vphi(t) \, dt \\
& \leq 2 C_1 \, F(g)^2 \,  D^{2-\al} \, \int_\R |t|^\al \, \vphi(t) \, dt = C \, F(g)^2 \, D^{2-\al} \; ,
\end{align*}
where $C > 0$ does not depend on $g$ or $D$. Take a finite subset $\cF \subset G$ such that $|F(g)| \leq M/A$ for all $g \in G \setminus \cF$. If $g \in G \setminus \cF$ and $|\gamma_g(t)| \geq 2 M$, it follows from \eqref{eq.some-bounds} that $|t| \geq |F(g)|^{-1}$. So, if $g \in G \setminus \cF$, we find that
$$\int_{\R \setminus [-2M,2M]} t^2 \, d\zeta_g(t) \leq \int_{|t| \geq |F(g)|^{-1}} (\gamma_g(t))^2 \, \vphi(t) \, dt \leq C \, F(g)^\al \; .$$
It then follows from \eqref{eq.variances} that
$$\int_{\R \setminus [-2M,2M]} t^2 \, d\zeta_g(t) = o(\Var \zeta_g) \quad\text{when $g \recht \infty$.}$$
Since $c_g$ is not a coboundary, we know that $F$ is not square summable. It then also follows from \eqref{eq.variances} that
$$\sum_{g \in G} \Var \zeta_g = +\infty \; .$$
So all assumptions of point~5 of Theorem \ref{thm.tail-periodic-criterion} are satisfied. Therefore, the tail boundary flow of $(\zeta_g)_{g \in G}$ is periodic. Hence, the associated flow of $\cS_G \actson (X,\mu)$ is periodic.

We have thus proven that under both hypotheses of Theorem \ref{thmD.examples-diffuse}, the action $G \actson X \times Y$ is not of type II$_\infty$ and not of type III$_0$. We finally prove that if the $T$-invariant is nontrivial, then the $1$-cocycle $c_g$ is a coboundary.

Assume that $p \neq 0$ such that $2\pi /p$ belongs to the $T$-invariant of $G \actson X \times Y$. We apply Theorem \ref{thm.Krieger-type-Bernoulli} and Remark \ref{rem.operational-T-invariant}. We find that the family of functions
$$\R^2 \recht \R/p\Z : (s,t) \mapsto (\psi(s + F(g)) - \psi(s)) - (\psi(t+F(g)) - \psi(t)) + p \Z$$
has a limit $\Om(s,t) \in \R/p\Z$ for a.e.\ $(s,t) \in \R^2$ as $g \recht \infty$. Since $0$ is a limit point of $(F(g))_{g \in G}$, we conclude that $\Om(s,t) = 0$ a.e. If $a \neq 0$ is another limit point of $(F(g))_{g \in G}$, it then follows that
$$(\psi(s + a) - \psi(s)) - (\psi(t+a) - \psi(t)) \in p \Z \quad\text{for a.e.\ $(s,t) \in \R^2$.}$$
By continuity, $s \mapsto \psi(s+a) - \psi(s)$ must be a constant function, which contradicts the integrability of $\vphi$. So, $\lim_{g \recht \infty} F(g) = 0$.

For every $r \in \R$, define
$$\Theta(r) = \int_{\R^2} d\bigl((\psi(s) - \psi(s-r)) - (\psi(t) - \psi(t-r)),p\Z\bigr)^2 \, \vphi(s) \, \vphi(t) \, ds \, dt \; .$$
Using \eqref{eq.good-condition} in Remark \ref{rem.operational-T-invariant}, we get that
\begin{equation}\label{eq.almost-done}
\sum_{g \in G} \Theta(F(g)) < +\infty \; .
\end{equation}
Under both hypotheses of Theorem \ref{thmD.examples-diffuse}, the function $\psi$ is a.e.\ differentiable. Since $\vphi$ is integrable, $\psi'$ is not essentially constant. We can then take $C > 0$ large enough, such that $\psi'$ is not essentially constant on $[-C,C]$. Since $\psi$ is continuous, we can take $\eps_0 > 0$ such that $|\psi(s) - \psi(s-r)| < |p|/4$ whenever $s \in [-C,C]$ and $|r|\leq \eps_0$. It follows that
\begin{multline*}
\lim_{\eps \recht 0} \frac{1}{\eps^2} \int_{[-C,C]^2} d\bigl((\psi(s) - \psi(s-\eps)) - (\psi(t) - \psi(t-\eps)),p\Z\bigr)^2 \, \vphi(s) \, \vphi(t) \, ds \, dt \\
= \int_{[-C,C]^2} (\psi'(s)-\psi'(t))^2 \, \vphi(s) \, \vphi(t) \, ds \, dt  > 0
\end{multline*}
exists and is strictly positive. We thus find $\eps_1 > 0$ such that $\Theta(r) \geq \eps_1 \, r^2$ whenever $|r| \leq \eps_1$. Take a finite subset $\cF \subset G$ such that $|F(g)| \leq \eps_1$ for all $g \in G \setminus \cF$. Then, \eqref{eq.almost-done} implies that
$$\sum_{g \in G \setminus \cF} F(g)^2 < +\infty \; ,$$
and it follows that $c_g$ is a coboundary.
\end{proof}

\begin{remark}\label{rem.about-growth}
Given the formulation of Theorem \ref{thmD.examples-diffuse}, it is tempting to try to prove that $G \actson (X,\mu)$ is dissipative when $\delta(c)$ is smaller than another ``constant'' associated with the function $\vphi$. But the following example shows that this is simply not true. More concretely, with $\vphi(t) = (2/\pi) (1+t^2)^{-2}$, we prove that for every $\eps > 0$, there exists a group $G$ and a cocycle $c_g$ such that $\delta(c) < \eps$ but still, $G \actson (X,\mu)$ is weakly mixing.

Assume that $W \subset G$ is an almost invariant subset. For every $\kappa > 0$, consider the function $F = \kappa \, 1_W$ with associated $1$-cocycle
$$(c_\kappa)_g = \kappa \, (1_{g W} - 1_W) = \kappa \, (1_{g W \setminus W} - 1_{W \setminus gW}) \; .$$
Writing
$$\delta(W) = \limsup_{s \recht +\infty} \frac{\log |\{g \in G \mid |g W \symdiff W| \leq s\}|}{s} \; ,$$
we find that $\delta(c_\kappa) = \kappa^{-2} \, \delta(W)$ for every $\kappa > 0$. Let $\vphi : \R \recht (0,+\infty)$ be any continuous function with $\int_\R \vphi(t) \, dt = 1$, with $\vphi(-t) = \vphi(t)$ for all $t \in \R$ and such that the function $\theta$ defined in \eqref{eq.function-theta} is finite for all $s \in \R$.

Writing $\al_\kappa = (1/2) \log(\theta(\kappa))$ and using the notation \eqref{eq.notation-C}, we find that $C(g) = \al_\kappa \, |g W \symdiff W|$. We also write
$$\beta_\kappa := - \log \Bigl( \int_\R \sqrt{\vphi(t+\kappa) \vphi(t)} \, dt \Bigr)$$
and note that
$$\int_X \sqrt{\frac{d(g^{-1}\mu)}{d\mu}(x)} \, d\mu(x) = \exp(-\beta_\kappa \, |g W \symdiff W|) \; .$$
By Lemma \ref{lem.strongly-conservative} and Proposition \ref{prop.ergodic-general}, the Bernoulli action $G \actson (X,\mu)$ is weakly mixing if $\delta(W) > 6 \, \al_\kappa$. By \cite[Proposition 4.1]{VW17}, the action $G \actson (X,\mu)$ is dissipative if $\delta(W) < \beta_\kappa$. As we have seen, $\delta(c_{\kappa}) = \kappa^{-2} \, \delta(W)$.

Below, we give for $a \in \{2,3,\ldots\}$ a concrete example of $W_a \subset G = \Z * (\Z/a\Z)$ with $\delta(W_a) = \log (2a-1) \in (0,+\infty)$.
Taking $\vphi(t) = (2/\pi) (1+t^2)^{-2}$, one computes that
$$\beta_\kappa = \log(1 + \kappa^2/4) \quad\text{and}\quad \al_\kappa = \log(1+2\kappa^2 + 5 \kappa^4/8) \; ,$$
which, for large values of $\kappa > 0$, is incomparable to $\kappa^2$. So, keeping $\vphi$ fixed, we can choose $\kappa > 0$ and $a \in \N$ large such that $\kappa^{-2} \log(2a-1)$ is arbitrarily small, while $6 \, \al_\kappa < \log(2a-1)$. So, $\delta(c_{\kappa})$ is arbitrarily small, while $G \actson (X,\mu)$ is weakly mixing.

Similarly to the proof of \cite[Proposition 6.8]{BKV19}, the above almost invariant subset $W_a \subset \Z * (\Z/a\Z)$ is defined as the set of elements that can be written by a reduced word either ending with a nonzero element of $\Z/a\Z$ or a positive element of $\Z$. By convention, the neutral element $e$ belongs to $W_a$. Whenever $g = a_0 n_1 a_1 \cdots a_{k-1} n_k a_k$ is a reduced word, with $a_i \in \Z/a\Z$ and $n_i \in \Z$, we define $|g| = |n_1| + \cdots + |n_k|$. By convention, $|g| = 0$ if $g \in \Z/a\Z$. A direct computation shows that
$$|g W_a \symdiff W_a| = |g| \;\;\text{for all $g \in G$, and}\quad |\{g \in G \mid |g| \leq m\}| = \frac{a}{a-1} (a (2a-1)^m -1)$$
for all $m \geq 1$. Hence, $\delta(W_a) = \log(2a-1)$.
\end{remark}

\section{\boldmath Nonsingular Bernoulli actions of $\Z$}

We prove Theorem \ref{thmE.Z-marginals-to-zero}. Assume that $(\mu_n)_{n \in \Z}$ are probability measures on $\{0,1\}$ with $\mu_n(0)$ converging to zero when $|n| \recht +\infty$. Consider the Bernoulli action $\Z \actson (X,\mu) = \prod_{n \in \Z} (\{0,1\},\mu_n)$. While the Kakutani criterion \eqref{eq.kakutani} makes it easy to give examples where $\Z \actson (X,\mu)$ is nonsingular and $\mu$ is nonatomic, the difficulty is to ensure that $\Z \actson (X,\mu)$ is conservative.

For this, we use a rather fine estimate for the quantity $C(n)$ defined in \eqref{eq.notation-C}. We start with the following elementary lemma.

\begin{lemma}\label{lem.estimate-projections}
Let $X_0,X_1$ be standard Borel spaces and $\pi : X_0 \recht X_1$ a Borel map. Let $\mu \sim \nu$ be equivalent probability measures on $X_0$. Then,
$$\int_{X_1} \frac{d(\pi_*(\mu))}{d(\pi_*(\nu))} \, d(\pi_*(\mu)) \leq \int_{X_0} \frac{d\mu}{d\nu} \, d\mu \; .$$
\end{lemma}
\begin{proof}
Denote by $E$ the $\nu$-preserving conditional expectation mapping a positive measurable function $F : X_0 \recht [0,+\infty)$ to the positive measurable function $E(F) : X_1 \recht [0,+\infty)$ uniquely determined, up to $\pi_*(\nu)$-a.e.\ equality, by
$$\int_{X_1} E(F) \, H \, d(\pi_*(\nu)) = \int_{X_0} F \, (H \circ \pi) \, d\nu \quad\text{for all positive measurable $H : X_1 \recht [0,+\infty)$.}$$
Recall that $(E(F))^2 \leq E(F^2)$ for all positive measurable $F : X_0 \recht [0,+\infty)$. Noting that
$$\frac{d(\pi_*(\mu))}{d(\pi_*(\nu))} = E\Bigl(\frac{d\mu}{d\nu}\Bigr) \; ,$$
we find that
\begin{align*}
\int_{X_1} \frac{d(\pi_*(\mu))}{d(\pi_*(\nu))} \, d(\pi_*(\mu)) &= \int_{X_1} \Bigl(\frac{d(\pi_*(\mu))}{d(\pi_*(\nu))}\Bigr)^2 \, d(\pi_*(\nu)) = \int_{X_1} \Bigl(E\Bigl(\frac{d\mu}{d\nu}\Bigr)\Bigr)^2 \, d(\pi_*(\nu)) \\
&\leq \int_{X_1} E\Bigl( \Bigl(\frac{d\mu}{d\nu}\Bigr)^2\Bigr) \, d(\pi_*(\nu)) = \int_{X_0} \Bigl(\frac{d\mu}{d\nu}\Bigr)^2 \, d\nu = \int_{X_0} \frac{d\mu}{d\nu} \, d\mu \; .
\end{align*}
So the lemma is proven.
\end{proof}

\begin{lemma}\label{lem.elementary-estimate}
Denote, for $a \in (0,1)$, by $\mu_a$ the probability measure on $\{0,1\}$ given by $\mu_a(0) = a$ and $\mu_a(1) = 1-a$. Define the increasing continuous map
$$\zeta : (0,1) \recht \R : \zeta(a) = \begin{cases} \log(2a) &\;\; \text{if $0 < a \leq 1/2$,}\\
-\log(2(1-a)) &\;\;\text{if $1/2 \leq a < 1$.}\end{cases}$$
Then, for all $a,b \in (0,1)$,
$$\frac{a^2}{b} + \frac{(1-a)^2}{1-b} = \int_{\{0,1\}} \frac{d\mu_a}{d\mu_b} \, d\mu_a \leq \exp(|\zeta(a) - \zeta(b)|^2) \; .$$
\end{lemma}
\begin{proof}
Define $\vphi : \R \recht (0,+\infty) : \vphi(t) = (1/2) \exp(-|t|)$. For every $s \in \R$, consider the probability measure $\nu_s$ on $\R$ given by $d\nu_s(t) = \vphi(t+s) \, dt$. By construction, the function $\zeta$ is the inverse of the distribution function $x \mapsto \int_{-\infty}^x \vphi(t) \, dt$. Define the Borel map $\pi : \R \recht \{0,1\}$ by putting $\pi(x) = 0$ if $x < 0$ and $\pi(x) = 1$ if $x \geq 0$. Note that $\mu_a = \pi_*(\nu_{\zeta(a)})$. By Lemma \ref{lem.estimate-projections}, we have
$$\int_{\{0,1\}} \frac{d\mu_a}{d\mu_b} \, d\mu_a \leq \int_\R \frac{d\nu_{\zeta(a)}}{d\nu_{\zeta(b)}} \, d\nu_{\zeta(a)} = \int_\R \frac{\vphi(t+\zeta(a))^2}{\vphi(t + \zeta(b))} \, dt = \int_\R \frac{\vphi(t+\zeta(a)-\zeta(b))^2}{\vphi(t)} \, dt \; .$$
By a direct computation, for all $s \in \R$,
$$\int_\R  \frac{\vphi(t+s)^2}{\vphi(t)} \, dt = \frac{2}{3} \exp(|s|) + \frac{1}{3} \exp(-2 |s|) \leq \exp(s^2) \; .$$
This concludes the proof of the lemma.
\end{proof}

Using Lemma \ref{lem.strongly-conservative}, we obtain the following sufficient condition for (strong) conservativeness. Applying \cite[Theorem A]{BKV19}, we get weak mixing, at least when $G$ is an abelian group. When $G = \Z$, a small argument allows us to prove that the action is of type III$_1$. The proof of Theorem \ref{thmE.Z-marginals-to-zero} below then consists in showing that all the hypotheses of the proposition can indeed be realized.

\begin{proposition}\label{prop.sufficient-conservative}
Let $G$ be a countable group and $F : G \recht [0,+\infty)$ a positive function implementing the cocycle $c_g(h) = F(g^{-1}h) - F(h)$ with $c_g \in \ell^2(G)$. Define the probability measures $\mu_g$ on $\{0,1\}$ by $\mu_g(0) = \exp(-F(g))/2$. If the Poincar\'{e} exponent of $c$ (see \eqref{eq.poincare}) satisfies $\delta(c) > 3$, then the Bernoulli action
$$G \actson (X,\mu) = \prod_{g \in G} (\{0,1\},\mu_g)$$
is nonsingular, satisfies \eqref{eq.strong-conservative-assumption} and is strongly conservative.
\begin{enumlist}
\item If moreover $G$ is abelian, then $G \actson (X,\mu)$ is weakly mixing.
\item If moreover $G = \Z$ and the function $F$ satisfies the following two properties, then $G \actson (X,\mu)$ is of stable type III$_1$.
\begin{itemlist}
\item The function $\N \recht \R : n \mapsto F(n)$ is nondecreasing and $\lim_{n \recht +\infty} F(n) = +\infty$.
\item For every $a \in \N$, the function $\N \recht \R : n \mapsto F(n+a) - F(n)$ is nonincreasing.
\end{itemlist}
\end{enumlist}
\end{proposition}
\begin{proof}
By Lemma \ref{lem.elementary-estimate} and using the notation \eqref{eq.notation-C}, we have
$$C(g) = \frac{1}{2} \sum_{h \in G} \log \int_{\{0,1\}} \frac{d\mu_h}{d\mu_{gh}} \, d\mu_h \leq \frac{1}{2} \sum_{h \in G} |F(h) - F(gh)|^2 = \frac{1}{2} \|c_g\|^2 <+\infty \; .$$
So, $G \actson (X,\mu)$ is nonsingular. By Lemma \ref{lem.strongly-conservative}, the action is strongly conservative. If $G$ is abelian, it follows from \cite[Theorem A]{BKV19} that $G \actson (X,\mu)$ is weakly mixing.

For the rest of the proof, assume that $G = \Z$ and that $F$ satisfies the two extra conditions in the formulation of the proposition. We first claim that
\begin{equation}\label{eq.enough-non-atomic}
\sum_{n=1}^\infty \mu_{2n}(0) = +\infty \; .
\end{equation}
Indeed, if this does not hold, translating by $1$ and using the Kakutani criterion for the nonsingularity of the Bernoulli action, we get that $\sum_{n=1}^\infty \mu_n(0) < +\infty$. Then $\cU = \{x \in X \mid x_n = 1 \;\;\text{for all $n \geq 1$ \;}\}$ has positive measure and $(-1) \cdot \cU \subset \cU$ with $\cU \setminus (-1)\cdot \cU$ having positive measure. This contradicts the conservativeness of $\Z \actson (X,\mu)$. So, \eqref{eq.enough-non-atomic} is proven.

In particular, since $\mu_n(1) \geq 1/2$ for all $n \in \Z$ and since \eqref{eq.enough-non-atomic} holds, the measure $\mu$ is nonatomic. By \cite[Theorem 1.1]{AP77}, the permutation action $\cS_\Z \actson (X,\mu)$ is ergodic. By \cite[Lemma 3.1]{BKV19}, it suffices to prove that $\cS_\Z \actson (X,\mu)$ is of type III$_1$. Fix $\rho > 1$. We prove that $\rho$ belongs to Krieger's ratio set for $\cS_\Z \actson (X,\mu)$. Fix a nonnegligible subset $\cU \subset X$ and $\eps > 0$. We construct a nonnegligible subset $\cV \subset \cU$ and an element $\si \in \cS_\Z$ such that $\si \cdot \cV \subset \cU$ and $|(d(\si^{-1}\cdot \mu)/d\mu)(x) - \rho| < \eps$ for all $x \in \cV$.

Fix $0 < a < b < +\infty$ such that $\exp([a,b]) \subset (1-\eps,1+\eps)$. Fix a nonempty open subinterval $(c,d) \subset (a,b)$. We claim that there exists $n_0 \in \N$ and a nondecreasing sequence $(k_n)_{n \geq n_0}$ of odd integers such that $F(2n + k_n) - F(2n) \in (c,d)$ for all $n \geq n_0$. Since $m \mapsto F(m+1) - F(m)$ is square summable, it certainly converges to zero as $m \recht +\infty$. We thus fix $n_0 \in \N$ such that $F(m+1) - F(m) < (d-c)/2$ for all $m \geq 2n_0$ and such that $F(2n_0+1) - F(2n_0) < c$. We then construct $(k_n)_{n \geq n_0}$ by induction. To fix $k_{n_0}$, note that $k \mapsto F(2n_0 + k) - F(2n_0)$ is nondecreasing, tending to $+\infty$, starting at $k=1$ below $c$ and having increments smaller than $(d-c)/2$. We can then define $k_{n_0}$ as the smallest odd integer such that $F(2n_0 + k_{n_0}) - F(2n_0) \in (c,d)$. Assume that $k_{n_0} \leq k_{n_0+1} \leq \cdots \leq k_n$ have been fixed for $n \geq n_0$. Since we assumed that the function $m \recht F(m + k_n) - F(m)$ is nonincreasing, we have
$$F(2(n+1) + k_n) - F(2(n+1)) \leq F(2n + k_n) - F(2n) < d \; .$$
If $F(2(n+1) + k_n) - F(2(n+1)) > c$, we put $k_{n+1} = k_n$. If $F(2(n+1) + k_n) - F(2(n+1)) \leq c$, we note as above that the function $m \mapsto F(2(n+1) + m) - F(2(n+1))$ is nondecreasing, tending to $+\infty$ and having increments smaller than $(d-c)/2$. The value at $m = k_n$ is smaller or equal than $c$. We define $k_{n+1} \geq k_n$ as the first odd integer such that $F(2(n+1)+k_{n+1}) - F(2n) \in (c,d)$. This ends the proof of the claim.

Since $\lim_{n \recht +\infty} \mu_n(1) = 1$, by taking $n_0$ large enough, we may assume that for all $n \geq n_0$,
$$|\log \mu_n(1)| < \min \{(b-d)/2,(c-a)/2\} \; .$$
Our choices guarantee that for all $n \geq n_0$,
\begin{equation}\label{eq.values-ratios}
r_n := \frac{\mu_{2n}(0) \, \mu_{2n+k_n}(1)}{\mu_{2n}(1) \, \mu_{2n+k_n}(0)} \in \exp([a,b]) \; .
\end{equation}
The sequences $(2n)_{n \geq n_0}$ and $(2n + k_n)_{n \geq n_0}$ of even, resp.\ odd integers are both strictly increasing. Therefore, the events
$$X_n = \bigl\{x \in X \bigm| (x_{2n},x_{2n+k_n}) \in \{(0,1),(1,0)\}\, \bigr\} \quad\text{with $n \geq n_0$,}$$
are independent. Since $\mu_m(1) \geq 1/2$ for all $m \in \Z$, it follows from \eqref{eq.enough-non-atomic} that $\sum_{n \geq n_0} \mu(X_n) = + \infty$. By the second Borel-Cantelli lemma, for every $n \geq n_0$, we have that $\bigcup_{m \geq n} X_m$ is conull. For every $n \geq n_0$, denote by $\si_n \in \cS_\Z$ the flip of $2n$ and $2n+k_n$. For every $n \geq n_0$, define the nonsingular automorphism $\vphi_n$ by
$$\vphi_n(x) = \si_m(x) \quad\text{where $m$ is the smallest integer satisfying $m \geq n$ and $x \in X_m$.}$$
By construction, $\vphi_n \circ \vphi_n = \id$ and, for all $n \geq n_0$ and a.e.\ $x \in X$,
\begin{equation}\label{eq.location-rn}
\frac{d(\vphi_n \mu)}{d\mu}(x) \in \{r_m , r_m^{-1} \mid m \geq n\} \subset R \quad\text{where $R:= \exp([a,b] \cup [-b,-a])$.}
\end{equation}
Define the operators $T_n$ on $L^1(X,\mu)$ by $T_n(\xi) = \xi \circ \vphi_n$. Then, $\|T_n(\xi)\|_1 \leq \exp(b) \|\xi\|_1$ for all $n$ and all $\xi \in L^1(X,\mu)$. We claim that $\lim_{n \recht +\infty} \|T_n(\xi) - \xi\|_1 = 0$ for all $\xi \in L^1(X,\mu)$. By the uniform boundedness of the operators $T_n$, it suffices to prove this claim when $\xi$ only depends on finitely many coordinates $x_m$, $m \in \Z$. In that case, $T_n(\xi) = \xi$ for all $n$ large enough.

Applying this to $\xi = 1_\cU$, where $\cU$ was fixed above, we find $n \geq n_0$ such that $\cV := \vphi_n(\cU) \cap \cU$ has positive measure. Note that $\vphi_n(\cV) \subset \cU$. By \eqref{eq.location-rn}, the Radon-Nikodym derivative of $\vphi_n$ takes values in $\exp([a,b]) \cup (\exp([a,b]))^{-1}$. Making $\cV$ smaller and interchanging $\cV$ and $\vphi_n(\cV)$ if needed, we may assume that the Radon-Nikodym derivative of $\vphi_n$ restricted to $\cV$ takes values in $\exp([a,b])$. Making $\cV$ even smaller, we find $\si \in \cS_\Z$ such that $\vphi_n(x) = \si \cdot x$ for all $x \in \cV$. By construction, $|(d(\si^{-1}\cdot \mu)/d\mu)(x) - \rho| < \eps$ for all $x \in \cV$.
\end{proof}

\begin{proof}[{Proof of Theorem \ref{thmE.Z-marginals-to-zero}}]
It suffices to construct a function $F : \Z \recht [0,+\infty)$ such that all hypotheses of Proposition \ref{prop.sufficient-conservative} hold.

Inductively define the increasing sequence $(b_n)_{n \in \N}$ by putting $b_1 = 1$ and choosing
$$b_n \geq \exp\Bigl(n^3 \sum_{k=1}^{n-1} b_k\Bigr)$$
for all $n \geq 2$. Write $a_0 = 0$ and $a_n = \sum_{k=1}^n b_k$ for all $n \geq 1$. Define the function
$$F : \Z \recht [0,+\infty) : F(n) = k + \frac{j}{b_k} \quad\text{whenever $|n| = a_{k-1} + j$ for $k \in \N$ and $0 \leq j \leq b_k$.}$$
Note that this definition is coherent when $|n| = a_{k-1} + b_k = a_k + 0$. Also note that $\N \recht \R : n \mapsto F(n)$ is increasing to $+\infty$ and that for all $a \in \N$, the function $\N \recht \R : n \mapsto F(n+a) - F(n)$ is nonincreasing, since the slopes of the piecewise linear function $F$ are decreasing as $n \recht +\infty$. Finally, note that $F(-n) = F(n)$ for all $n \in \Z$. It remains to prove that $F$ implements an $\ell^2$-cocycle $c$ whose Poincar\'{e} exponent is larger than $3$.

For every $k \in \N$, define $F_k(n) = F(n) - k$ and $G_k(n) = k$ if $|n| \leq a_{k-1}$, and $F_k(n) = 0$ and $G_k(n) = F(n)$ if $|n| \geq a_{k-1}$. Obviously, $F = F_k + G_k$. Note that $-k \leq F_k(n) \leq 0$ for all $n$, so that $\|F_k\|_2^2 \leq 2 k^2 a_{k-1}$. Also, using the symmetry of $G_k$,
\begin{align*}
\sum_{n \in \Z} |G_k(n-1) - G_k(n)|^2 & = 2 \sum_{n=a_{k-1}}^{\infty} (F(n+1) - F(n))^2 \\ & = 2 \sum_{s = k}^{\infty} \sum_{j=0}^{b_s-1} (F(a_{s-1} + j + 1) - F(a_{s-1} + j))^2 \\
& = 2 \sum_{s = k}^{\infty} \sum_{j=0}^{b_s-1} \frac{1}{b_s^2} = 2 \sum_{s=k}^{\infty} \frac{1}{b_s} \leq \frac{4}{b_k}
\end{align*}
because $b_s \geq 2 b_{s-1}$ for all $s$.

For every $N \in \Z$, write $c_N(n) = F(n-N) - F(n)$. We have proven that $c_1 \in \ell^2(\Z)$ and thus, $c_N \in \ell^2(\Z)$ for all $N \in \Z$. For every $\xi \in \ell^2(\Z)$ and $N \in \Z$, write $(\lambda_N \xi)(n) = \xi(n-N)$. For every $k \geq 1$, define $\xi_k(n) = G_k(n-1) - G_k(n)$. By the computation above, $\|\xi_k\|_2^2 \leq 4/b_k$.

For every $k \geq 1$ and $N \geq 1$, because $F = F_k + G_k$, we have
$$c_N = (\lambda_N F_k - F_k) + \sum_{j=0}^{N-1} \lambda_j \xi_k \; .$$
It follows that
$$\|c_N\|_2 \leq 2 \|F_k\|_2 + N \|\xi_k\|_2 \leq 4 k \sqrt{a_{k-1}} + \frac{2N}{\sqrt{b_k}} \; .$$
So, whenever $k \geq 2$ and $1 \leq N \leq \sqrt{b_k}$, we have
$$\|c_N\|_2 \leq 4 k \sqrt{a_{k-1}} + 2 \leq 5 k \sqrt{a_{k-1}}$$
and thus, since $b_k$ was chosen larger than $\exp(k^3 a_{k-1})$,
$$\|c_N\|_2^2 \leq  25 k^2 a_{k-1} \leq \frac{50}{k} \log(\sqrt{b_k}) \; .$$
Put $\al_k = (50/k) \log(\sqrt{b_k})$. We have shown that
$$|\{ N \in \Z \mid \|c_N\|_2^2 \leq \al_k\}| \geq \sqrt{b_k} \; ,$$
so that
$$\frac{\log |\{ N \in \Z \mid \|c_N\|_2^2 \leq \al_k\}|}{\al_k} \geq \frac{k}{50} \quad\text{and}\quad \limsup_{s \recht +\infty} \frac{\log |\{N \in \Z \mid \|c_N\|_2^2 \leq s\}|}{s} = +\infty \; .$$
This means that $\delta(c) = +\infty$ and the theorem is proven.
\end{proof}

\end{document}